\documentclass[a4paper,11pt,reqno]{amsart}

\usepackage[utf8]{inputenc}
\usepackage[T1]{fontenc}
\usepackage{lmodern}
\usepackage[english]{babel}
\usepackage{microtype}

\usepackage{amsmath,amssymb,amsfonts,amsthm,esint}
\usepackage{mathtools,accents}
\usepackage{mathrsfs}
\usepackage{aliascnt}
\usepackage{braket}
\usepackage{bm}

\usepackage[a4paper,margin=3cm]{geometry}
\usepackage[citecolor=blue,colorlinks]{hyperref}

\usepackage{enumerate}
\usepackage{xcolor}

\makeatletter
\g@addto@macro\@floatboxreset\centering
\makeatother

\makeatletter
\def\newaliasedtheorem#1[#2]#3{
	\newaliascnt{#1@alt}{#2}
	\newtheorem{#1}[#1@alt]{#3}
	\expandafter\newcommand\csname #1@altname\endcsname{#3}
}
\makeatother

\numberwithin{equation}{section}

\newtheoremstyle{slanted}{\topsep}{\topsep}{\slshape}{}{\bfseries}{.}{.5em}{}

\theoremstyle{plain}
\newtheorem{theorem}{Theorem}[section]
\newaliasedtheorem{proposition}[theorem]{Proposition}
\newaliasedtheorem{lemma}[theorem]{Lemma}
\newaliasedtheorem{corollary}[theorem]{Corollary}
\newaliasedtheorem{counterexample}[theorem]{Counterexample}

\theoremstyle{definition}
\newaliasedtheorem{definition}[theorem]{Definition}
\newaliasedtheorem{question}[theorem]{Question}
\newaliasedtheorem{openquestion}[theorem]{Open Question}
\newaliasedtheorem{conjecture}[theorem]{Conjecture}

\theoremstyle{remark}
\newaliasedtheorem{remark}[theorem]{Remark}
\newaliasedtheorem{example}[theorem]{Example}

\newcommand{\setN}{\mathbb{N}}

\newcommand{\setR}{\mathbb{R}}

\newcommand{\eps}{\varepsilon}

\let\altphi\phi
\let\phi\varphi
\let\varphi\altphi
\let\altphi\undefined

\newcommand{\abs}[1]{\left\lvert#1\right\rvert}
\newcommand{\norm}[1]{\left\lVert#1\right\rVert}

\let\div\undefined
\DeclareMathOperator{\div}{div}

\newcommand{\di}{\mathop{}\!\mathrm{d}}

\newcommand{\bs}{{\rm bs}}
\newcommand{\loc}{{\rm loc}}

\newcommand{\diam}{{\rm diam}}
\newcommand{\res}{\mathop{\hbox{\vrule height 7pt width .5pt depth 0pt
			\vrule height .5pt width 6pt depth 0pt}}\nolimits}

\DeclareMathOperator{\supp}{supp}

\newcommand{\Ch}{{\sf Ch}}

\DeclareMathOperator{\Lip}{Lip}
\DeclareMathOperator{\Lipb}{Lip_b}
\DeclareMathOperator{\Lipbs}{Lip_\bs}
\DeclareMathOperator{\Cb}{C_b}

\DeclareMathOperator{\lip}{lip} 

\newcommand{\Leb}{\mathscr{L}}

\newcommand{\dist}{\mathsf{d}}

\newcommand{\meas}{\mathfrak{m}}

\DeclareMathOperator{\MCP}{MCP}
\DeclareMathOperator{\CD}{CD}
\DeclareMathOperator{\RCD}{RCD}

\newfont{\tmpf}{cmsy10 scaled 2500}

\def\XXint#1#2#3{{\setbox0=\hbox{$#1{#2#3}{\int}$ }
		\vcenter{\hbox{$#2#3$ }}\kern-.6\wd0}}

\newcommand{\kse}{{\sf ks}}
\newcommand{\CAT}{{\sf CAT}}
\newcommand{\ks}{{\sf KS}}

\begin{document}

\title[Harmonic maps from $\RCD(K,N)$ to $\CAT(0)$ spaces]{Lipschitz continuity and Bochner-Eells-Sampson inequality for harmonic maps from $\RCD(K,N)$ spaces to $\CAT(0)$ spaces}

\author{Andrea Mondino}\thanks{Andrea Mondino: Mathematical Institute, University of Oxford, Radcliffe Observatory, Andrew Wiles Building, Woodstock Rd, Oxford OX2 6GG, UK,\, Andrea.Mondino@maths.ox.ac.uk} 
\author{Daniele Semola}\thanks{Daniele Semola: Mathematical Institute, University of Oxford, Radcliffe Observatory, Andrew Wiles Building, Woodstock Rd, Oxford OX2 6GG, UK,\, daniele.semola.math@gmail.com}

\maketitle

\begin{abstract}
We establish Lipschitz regularity of harmonic maps from $\RCD(K,N)$ metric measure spaces with lower Ricci curvature bounds and dimension upper bounds in synthetic sense with values into $\CAT(0)$ metric spaces with non-positive sectional curvature. Under the same assumptions, we obtain a Bochner-Eells-Sampson inequality with a Hessian type-term. This gives a fairly complete generalization of the classical theory for smooth source and target spaces to their natural synthetic counterparts and an affirmative answer to a question raised several times in the recent literature.\\
The proofs build on a new interpretation of the interplay between Optimal Transport and the Heat Flow on the source space and on an original perturbation argument in the spirit of the viscosity theory of PDEs.
\end{abstract}

\tableofcontents

\section{Introduction}
In this paper we establish local Lipschitz continuity of harmonic maps from non-smooth metric measure spaces with synthetic Ricci curvature lower bounds in the sense of the $\RCD$ theory with values into $\CAT(0)$ metric spaces with non-positive sectional curvature. This answers to a question raised several times in the recent literature, see \cite{GigliPasqualettoSoultanis20,DiMarinoGiglietal21,GigliTulyenev20,GigliTulyenev21,Guo21}. Building on the top of this regularity result, we prove a Bochner-Eells-Sampson type inequality with Hessian term, which is expected to be fundamental for future applications.
\medskip

A smooth map $u:M^n\to N^k$ between Riemannian manifolds is called harmonic when the \emph{tension field}
\begin{equation*}
\Delta u:=\mathrm{tr}\nabla\left(\di u\right)
\end{equation*}
vanishes identically, as a section of the pull-back bundle $u^*TN$. There are several examples of harmonic maps: harmonic functions when the target space is $\setR$, geodesics when the source space is $\setR$, isometries, conformal maps, holomorphic maps between K\"ahler manifolds and inclusions of volume minimizing submanifolds. Their role is ubiquitous in Geometric Analysis.\\ 
The basic question becomes then the existence of harmonic maps, under suitable assumptions. The problem was approached from a parabolic perspective by Eells-Sampson \cite{EellsSampson64} and subsequently by Hamilton \cite{Hamilton}, based on the long time behaviour of the non-linear heat equation
\begin{equation*}
\frac{\di }{\di t}u(x,t)=\Delta u(x,t)\, .
\end{equation*}  
Later on, the variational perspective was put forward by Hildebrandt-Kaul-Widman \cite{HildebrandtKaul1,HildebrandtKaul2} (see also the subsequent work of Schoen \cite{Schoen}), building on  top of the interpretation of harmonic maps as critical points of the energy functional
\begin{equation*}
E(u):=\int_M\abs{\di u}^2\di\mathrm{vol}_M\, .
\end{equation*}

Very much intertwined with the question of existence, there is the issue of regularity.\\
If we denote by $\mathrm{Ric}_M$ and $\mathrm{R}^N$ the Ricci curvature tensor of $M$ and the Riemann curvature tensor of $N$, respectively, and by $\{e_{\alpha}\}_{1\le \alpha\le n}$ an orthonormal base of $TM$, the Bochner-Eells-Sampson formula for harmonic maps
\begin{equation}\label{eq:BES}
\Delta\frac{1}{2}\abs{\di u}^2=\abs{\nabla \di u}^2+\mathrm{Ric}_M(\nabla u,\nabla u)-\sum_{\alpha,\beta=1,\dots,n}\langle\mathrm{R}^N\left(u_*e_{\alpha},u_{*}e_{\beta}\right)u_*e_{\alpha},u_*e_{\beta}\rangle
\end{equation}
hints towards a prominent role of lower bounds on the Ricci curvature of $M$ and upper bounds on the sectional curvature of $N$ in developing a regularity theory. Indeed, if we assume that $\mathrm{Ric}_M\ge K$ and $\mathrm{R}^N\le 0$, then from \eqref{eq:BES} we obtain that
\begin{equation}\label{eq:Bochnerineq}
\Delta\frac{1}{2}\abs{\di u}^2\ge K\abs{\di u}^2\, .
\end{equation}
A priori, local $L^{\infty}$-estimates for $\abs{\di u}$ can be derived from \eqref{eq:Bochnerineq} via the classical De Giorgi-Moser iteration. Smoothness then follows from elliptic regularity, see \cite{EellsSampson64,SchoenUhlenbeck}. We refer also to the more recent work of Sturm \cite{Sturm05} for a different, probabilistic interpretation of the curvature conditions on source and target spaces in the theory of harmonic maps, deeply related to the developments of the present note. 
\medskip

In the last thirty years, starting from the work of Gromov-Schoen \cite{GromovSchoen}, there has been growing interest in developing a theory of harmonic maps between spaces more general than Riemannian manifolds and possibly non-smooth. This has required a completely new set of ideas and techniques, as neither isometric embeddings into Euclidean spaces, nor local charts are in general available.\\ 
The analysis in \cite{GromovSchoen} was dedicated to maps from smooth source spaces with values into locally finite Riemannian simplicial complexes, with striking applications in Geometric Group Theory. Korevaar-Schoen \cite{KorevaarSchoen,KorevaarSchoen2}, and independently Jost \cite{Jost94,Jost95}, later developed a general theory of Sobolev and harmonic maps with values into metric spaces with non-positive curvature in the sense of Alexandrov. From the variational perspective, the curvature assumption on the target guarantees convexity of the energy functional.\\ 
In \cite{KorevaarSchoen,KorevaarSchoen2} source spaces are smooth manifolds and the authors obtain local Lipschitz continuity of harmonic maps. In \cite{Jost94,Jost95} source spaces are locally compact metric spaces with a Dirichlet form and the author obtains local H\"older continuity, assuming a uniform scale invariant Poincar\'e inequality. 
\smallskip

For the sake of the applications, Lipschitz continuity and suitable versions of the Bochner-Eells-Sampson inequality \eqref{eq:BES} are two cornerstones. 
An example from \cite{Koskelaetal} shows that, in general, doubling and Poincar\'e assumptions on the source space do not guarantee Lipschitz regularity of harmonic functions, even in the scalar-valued case. Meanwhile, the developments of the theory of Alexandrov spaces motivated the conjecture that harmonic maps from metric spaces with sectional curvature bounded from below with values into metric spaces with non-positive curvature should be locally Lipschitz. The conjecture was formulated by Lin \cite{Lin97} and, in a more open form, by Jost \cite{Jost98}, and it has been recently settled by Zhang-Zhu \cite{ZhangZhu18}.\\
We mention \cite{EellsFuglede,Gregori,KuwaeShioya,Fuglede,Chen95,DaskaMese08,DaskaMese10} for previous, related developments of the theory of harmonic maps between metric spaces, without pretending of being complete in this list.
Also, more recently, there has been progress in the existence and regularity of harmonic maps into $\CAT(1)$ target spaces by Breiner-Fraser-Huang-Mese-Sargent-Zhang \cite{BFHMSZ1, BFHMSZ2}
\medskip

Nowadays, there is a well-established theory of metric measure spaces with lower bounds on the Ricci curvature in the synthetic sense, the so-called $\RCD(K,N)$ metric measure spaces $(X,\dist,\meas)$. Here $K\in\setR$ plays the role of a synthetic lower bound on the Ricci curvature and $1\le N<\infty$ plays the role of a synthetic upper bound on the dimension, in the sense of the Lott-Sturm-Villani ``Curvature-Dimension'' condition \cite{Sturm06I,Sturm06II,LottVillani09}. The ``Riemannian'' assumption, formulated in terms of linearity of the heat flow, is added to force Hilbertian behaviour in the much broader class of Finsler geometries, that the Curvature-Dimension condition does not rule out a priori.\\ 
We address the reader to \autoref{sub:GARCD} (see also the survey \cite{Ambrosio18} and references therein) for the relevant background on the theory of $\RCD(K,N)$ metric measure spaces. Here we just remark that the $\RCD$ theory is fully consistent with the theory of smooth (weighted) Riemannian manifolds with (weighted) Ricci curvature bounded from below and with the theory of Alexandrov spaces with sectional curvature bounded from below. Furthermore, the $\RCD(K,N)$ condition is stable with respect to the measured Gromov-Hausdorff topology, under cone and spherical suspension constructions, and under quotients by actions of groups of measure preserving isometries.
\medskip

The role of the lower Ricci curvature bound on the source space for the regularity of harmonic maps in the classical theory, together with the remarkable developments of Geometric Analysis on $\RCD(K,N)$ spaces in recent years, give strong motivations for a theory of harmonic maps from $\RCD(K,N)$ spaces with values into $\CAT(0)$ metric spaces with non-positive curvature. A theory of Sobolev maps in this framework has been developed by Gigli-Tyulenev in \cite{GigliTulyenev21}, where the existence and uniqueness of solutions of the Dirichlet problem have been achieved too. Local H\"older regularity of harmonic maps has been recently obtained by Guo in \cite{Guo21}, along the lines of the previous \cite{Jost97,Lin97}.\\
The question of local Lipschitz regularity of harmonic maps from $\RCD(K,N)$ spaces to $\CAT(0)$ spaces has been raised several times in the recent literature, starting from \cite{GigliPasqualettoSoultanis20} and later in \cite{DiMarinoGiglietal21,GigliTulyenev20,GigliTulyenev21,Guo21}.
\smallskip

The first main result of this paper is a positive answer to this question in full generality, which can be considered also as a complete answer to the question raised in \cite{Jost98}. Indeed, we fully generalize the Lipschitz regularity result for smooth manifolds by Eells-Sampson \cite{EellsSampson64} to the natural synthetic framework.\\ 
We address the reader to \autoref{subsec:energyharmonics} for the introduction of the relevant background and terminology about harmonic maps in this setting. Below, with the notation $\mathrm{osc}\, u$ we indicate the oscillation of the harmonic map $u$, which is locally bounded thanks to \cite{Guo21}.

\begin{theorem}[cf. with \autoref{mainthcore}]\label{thm:main}
Let $(X,\dist,\meas)$ be an $\RCD(K,N)$ metric measure space for some $K\in\setR$ and $1\le N<\infty$. Let $(Y,\dist_Y)$ be a $\CAT(0)$ space and let $\Omega\subset X$ be an open domain. Assume that $u:\Omega\to Y$ is a harmonic map. Then for any $0<R\le 1$ there exists a constant $C=C(K,N,R)>0$ such that if $B_{2R}(q)\Subset \Omega$ for some point $q\in X$, then for any $x,y\in B_{R/16}(q)$ it holds
\begin{equation}\label{eq:Lipintro}
\dist_Y(u(x),u(y))\le C(K,N,R)\left(\left(\fint_{B_{R}(q)}\abs{\di u(z)}^2\di\meas(z)\right)^{\frac{1}{2}}+\mathrm{osc}_{\overline{B}_R(q)}u\right)\dist(x,y)\, .
\end{equation}
\end{theorem}

We postpone to the second part of the introduction a description of the strategy of the proof of \autoref{thm:main}. As we shall see, the key step will be establishing a (very) weak form of the Bochner-Eells-Sampson inequality \eqref{eq:Bochnerineq}. Towards this goal, we will need to introduce several original ideas with respect to the previous literature.
\medskip

As of today, there seems to be no notion of Hessian available for harmonic maps from $\RCD(K,N)$ metric measure spaces with values into $\CAT(0)$ metric spaces. This would be required to give meaning to a version of \eqref{eq:BES} in this setting, where only the curvature terms are disregarded. Nevertheless, we are able to prove a Bochner-Eells-Sampson inequality for harmonic maps where a Hessian-type term
\begin{equation*}
\abs{\nabla \abs{\di u}}\le \abs{\nabla \di u}
\end{equation*}  
appears. Below we shall denote by $\lip u$ the pointwise Lipschitz constant of a harmonic map $u:\Omega\to Y$, defined by
\begin{equation*}
\lip u(x):=\limsup_{y\to x}\frac{\dist_Y(u(x),u(y))}{\dist(x,y)}\, .
\end{equation*}
We remark that, by \autoref{thm:main}, the pointwise Lipschitz constant of a harmonic map is locally bounded. 

\begin{theorem}[cf. with \autoref{thm:Bochner}]\label{thm:Bochnerintro}
Let $(X,\dist,\meas)$ be an $\RCD(K,N)$ metric measure space for some $K\in\setR$, $1\le N<\infty$, and let $(Y,\dist_Y)$ be a $\CAT(0)$ space. Let $\Omega\subset X$ be an open domain and let $u:\Omega\to\setR$ be a harmonic map. Then $\lip u\in W^{1,2}_{\loc}(\Omega)\cap L^{\infty}_{\loc}(\Omega)$ and 
\begin{equation}\label{eq:bochnerwithhessianintro}
\Delta \frac{\abs{\lip u}^2}{2}\ge \abs{\nabla \lip u}^2+K\abs{\lip u}^2\, ,\quad\text{on $\Omega$}\, ,
\end{equation}
in the sense of distributions.
\end{theorem}

We refer to the very recent work \cite{ZhangZhongZhu19} by Zhang-Zhong-Zhu for an analogous statement under the assumption that the source space is a smooth $N$-dimensional Riemannian manifold with Ricci curvature bounded from below by $K$ and to \cite{Freidin19,FreidinZhang20} for previous instances of Bochner-Eells-Sampson formulas (without Hessian-type terms) for maps with values into $\CAT(k)$ spaces and with source smooth Riemannian manifolds and polyhedra, respectively.
\smallskip

The validity of a Bochner inequality for scalar-valued maps defined on a non-smooth $\RCD$ space (even without Hessian term) is a very deep result: it was proved for $\RCD(K,\infty)$ spaces by Ambrosio-Gigli-Savar\'e \cite{AGSDuke} (see also \cite{AGS15} by the same authors for the reverse implication). The dimensional improvement for $\RCD^*(K,N)$ spaces was established independently by Erbar-Kuwada-Sturm \cite{EKS} and by Ambrosio-Savar\'e and the first author \cite{AmbrosioMondinoSavare19} (together with the reverse implication). The fact that the scalar Bochner inequality (without Hessian) ``self-improves'' to estimate the norm of the Hessian was noticed in the setting of $\Gamma$-calculus by Bakry \cite{Bakry85} and subsequently proved in the non-smooth setting of $\RCD$ spaces by Savar\'e \cite{Savare14} and Gigli \cite{Gigli18}.
\smallskip

To the best of our knowledge, \autoref{thm:Bochnerintro} is the first instance of a Bochner-Eells-Sampson inequality with Hessian-type term for harmonic maps when both the source and the target spaces are non-smooth.\\
We remark that the appearance of the Hessian-type term in \eqref{eq:bochnerwithhessianintro} is expected to have fundamental importance in the future developments of the theory, see for instance the discussion in the introduction of \cite{ZhangZhongZhu19}.

\subsection*{Strategy of the proof}\label{subsec:strat}

There are two fundamental difficulties in the proof of the local Lipschitz continuity for harmonic maps in this setting. The first one is the need to deduce the information coming from the combination of the Bochner-Eells-Sampson formula \eqref{eq:BES} with the curvature constraints on the source space from ``lower order considerations'', independent of any regularity. This is a fundamental issue in Geometric Analysis of non-smooth spaces, which requires a new argument in the present situation and it is well illustrated already at the level of scalar-valued harmonic functions. The second key point is the need to turn the combination of harmonicity, which is understood here in variational terms, with the curvature constraints of the target into a differential inequality. This difficulty is tied with the non-linearity of the variational problem and it requires an original idea.
\medskip

We illustrate the first idea in the case of linear harmonic functions $u:\Omega\to\setR$, where $\Omega\subset X$ is an open domain and $(X,\dist,\meas)$ is an $\RCD(0,N)$ metric measure space, for the sake of simplicity. The case of general lower Ricci curvature bounds $K\in\setR$ would introduce additional error terms without affecting the general strategy.\\ 
Notice that local Lipschitz estimates in this case follow from Harnack's inequality as soon as one is able to prove that
\begin{equation}\label{eq:nablauSH}
\Delta \abs{\nabla u}^2\ge 0\, ,
\end{equation}
in the sense of distributions. For smooth Riemannian manifolds, and under the assumption that $u$ is smooth, the estimate \eqref{eq:nablauSH} follows from Bochner's inequality. For $\RCD(0,N)$ spaces we refer to \cite{Jiang12,ZhangZhu16} for a distributional approach, building on the top of weak versions of Bochner's inequality (see also \cite{AGS15,EKS,AmbrosioMondinoSavare19}).

Here we present a different strategy, more suitable to be generalized to $\CAT(0)$-valued harmonic maps. This approach was developed in the case of source spaces with sectional curvature bounded from below in the Alexandrov sense by Petrunin and Zhang-Zhu in \cite{Petrunin96,ZhangZhu12}. We remark that the strategy has strong analogies with the so-called two points maximum principle, see for instance \cite{Korevaar,Kruzkov} and \cite{Andrews} for a recent survey.
\smallskip

We assume that $u$ is continuous (a property that is usually proved via De Giorgi-Moser's iteration and Harnack's inequality) and we wish to bootstrap the continuity to local Lipschitz continuity. 
We introduce the evolution via the Hopf-Lax semigroup (up to a sign)
\begin{equation*}
\mathcal{Q}^tu(x):=\sup_{y\in \Omega}\left\{u(y)-\frac{\dist^2(x,y)}{2t}\right\}\, .
\end{equation*}
Then we claim that $\Delta \mathcal{Q}^tu\ge 0$, locally and for $t>0$ sufficiently small. On a smooth Riemannian manifold, neglecting the regularity issues, this inequality can be proved with a computation using the second variation of the arc length and Jacobi fields, see for instance the survey by Andrews \cite{Andrews} for similar arguments. On Alexandrov spaces with lower sectional curvature bounds, the statement is proved by Zhang-Zhu \cite{ZhangZhu12} (following an argument proposed by Petrunin in the unpublished \cite{Petrunin96}) relying on several perturbation arguments and on Petrunin's second variation formula for the arc-length \cite{Petrunin98}. None of these approaches is available in the framework of $\RCD(K,N)$ metric measure spaces. Roughly speaking, the main reason is that they rely on estimates for second-order variations ``in single directions'', from which Laplacian estimates are deduced, as a second step, by average.
\smallskip

In turn, on $\RCD$ spaces there is a completely alternative strategy, where the need to estimate second-order variations in single directions is completely overcome. The sub-harmonicity $\Delta \mathcal{Q}^tu\ge 0$ follows from the interplay between Optimal Transport and the Heat Flow in this setting \cite{SturmVonRenesse,Kuwada10,AGSDuke}, by further developing an argument found by the authors in \cite{MondinoSemola21}. Denoting by $P_s$ the Heat Flow
\begin{equation*}
\frac{\di}{\di s}P_su=\Delta P_su\, ,
\end{equation*}
the so-called Kuwada duality \cite{Kuwada10} guarantees that 
\begin{equation*}
P_s\mathcal{Q}^tu(x)\ge P_su(y)-\frac{\dist^2(x,y)}{2t}\, ,
\end{equation*}
for any $x,y\in X$ and for any $s>0$.
Therefore, formally 
\begin{equation}\label{eq:propagation}
\liminf_{s\to 0}\frac{P_s\mathcal{Q}^tu(x)-\mathcal{Q}^tu(x)}{s}\ge \liminf_{s\to 0}\frac{P_su(x_t)-u(x_t)}{s}=\Delta u(x_t)=0\, , 
\end{equation}
where $x_t\in \Omega$ is any point such that 
\begin{equation*}
\mathcal{Q}^tu(x)=u(x_t)-\frac{\dist^2(x,x_t)}{2t}\, . 
\end{equation*}
This (formally) shows that $\Delta \mathcal{Q}^tu\ge 0$. As $u$ is harmonic, also
\begin{equation}\label{eq:subha}
\Delta \left( \frac{\mathcal{Q}^tu-u}{t}\right)\ge 0\, .
\end{equation}
If we recall that the Hopf-Lax semigroup solves the Hamilton-Jacobi equation
\begin{equation*}
\frac{\di \mathcal{Q}^tu}{\di t}-\frac{1}{2}\abs{\nabla \mathcal{Q}^tu}^2=0\, ,
\end{equation*}
see for instance \cite{AGS14},
then we easily infer that a local gradient estimate for $u$ follows, again formally, from a uniform (as $t\downarrow 0$) local $L^{\infty}$-estimate for 
\begin{equation}\label{eq:diffsemi}
\frac{\mathcal{Q}^tu-u}{t}\, .
\end{equation}
Indeed, taking the limit as $t\downarrow 0$, it holds $(\mathcal{Q}^tu-u)/t\to \frac{1}{2}\abs{\nabla u}^2$.
In order to get the sought local, uniform estimate for \eqref{eq:diffsemi}, it is now sufficient to use Harnack's inequality for sub-harmonic functions, thanks to \eqref{eq:subha}.
\smallskip

It is possible to interpret the strategy outlined above in terms of the two-variables functions $F_t:\Omega\times \Omega\to\setR$,
\begin{equation}\label{eq:twovarF}
F_t(x,y):=u(y)-\frac{\dist^2(x,y)}{2t}\, .
\end{equation}
Second variation arguments exploit the assumption that $u$ is harmonic, therefore the first term above solves the Laplace equation with respect to the $y$ variable, and the lower Ricci curvature bound of the ambient space, encoded into the behaviour of the distance squared on the product $X\times X$.
\medskip

In order to deal with harmonic maps with values into $\CAT(0)$ spaces $(Y,\dist_Y)$, the non-linearity of the target is a fundamental issue, as there is no clear counterpart of \eqref{eq:twovarF}.\\ 
A key idea borrowed from \cite{ZhangZhu18} (see also \cite[Section 5]{Andrews}) is to consider the function of two variables
\begin{equation}\label{eq:Gproduct}
G_t(x,y):=\dist_Y(u(x),u(y))-\frac{\dist^2(x,y)}{2t}\, .
\end{equation}
The strategy is to exploit the curvature constraints of the source and the target spaces and the harmonicity of the map $u$ in terms of the behaviour of the function $G_t$ on the product $X\times X$.
\smallskip

The lower Ricci curvature bound enters into play as in the scalar-valued case, through the Wasserstein contractivity of the Heat Flow and it controls the second term at the right-hand side in \eqref{eq:Gproduct}.
Notice that, again, this is a fundamentally different approach with respect to the previous \cite{EellsSampson64,KorevaarSchoen,ZhangZhu18} and with respect to the strategy illustrated in \cite{Andrews}. It shares some similarities with \cite{Sturm05} where, however, the perspective on harmonic maps was parabolic rather than variational.
\medskip

The $\CAT(0)$ condition on the target is combined with the assumption that $u$ is harmonic to control the first term at the right-hand side in \eqref{eq:Gproduct}. In particular, the combination of these two assumptions, neglecting the regularity issues, leads to the inequality
\begin{equation}\label{eq:Lapla1}
\Delta \left(\dist_Y^2(u(\cdot),u(x_0))-\dist^2_Y(u(\cdot),P)\right)(x_0)\le 0\, ,
\end{equation}
for any $x_0\in\Omega$ and for any $P\in Y$. Moreover, the non-positive curvature assumption allows a decoupling of the two variables in \eqref{eq:Gproduct}, via a quadrilateral comparison finding its roots in \cite{Reshetniak} (see \autoref{lemma:elemCAT} for the precise statement), and it leads from \eqref{eq:Lapla1} to the differential inequality 
\begin{equation}\label{eq:Lapla2}
\Delta_xf(x_0)+\Delta_yg(y_0)\ge  0\, ,
\end{equation}
for suitably constructed auxiliary functions $f:X\to\setR$ and $g:X\to\setR$ such that 
\begin{equation*}
f(x)+g(y)\le \dist_Y(u(x),u(y))\, ,\quad\text{for any $x,y\in X$}\, 
\end{equation*}
and
\begin{equation*}
f(x_0)+g(y_0)=\dist_Y(u(x_0),u(y_0))\, ,
\end{equation*}
see \eqref{eq:rewritten} and the subsequent discussion for the details of the construction.\\
The combination of \eqref{eq:Lapla2}, together with the Wasserstein contractivity of the Heat Flow to control the second term in \eqref{eq:Gproduct}, proves the sub-harmonicity of the function
\begin{equation*}
\mathcal{G}^t(x):=\sup_{y\in\Omega}\left\{\dist_Y(u(x),u(y))-\frac{\dist^2(x,y)}{2t}\right\}\, .
\end{equation*}
This will be rigorously proved in  \autoref{sec:propagation}. The  local Lipschitz regularity of $u$ will be established in  \autoref{sec:Lip}, via a  variant of the aforementioned argument presented for scalar valued harmonic maps, where the sub-harmonicity of the function $\mathcal{G}^t$ plays the role of \eqref{eq:subha}.
\smallskip

A major difficulty that we encounter in making rigorous the strategy above is that we are able to verify \eqref{eq:Lapla1} (and therefore also \eqref{eq:Lapla2}) only away from a set of negligible measure in the source domain, see \autoref{prop:intw}. This is in perfect analogy with similar situations in the classical viscosity theory of Partial Differential Equations, see for instance \cite{CaffarelliCabre95}, and with the case of Alexandrov spaces \cite{Petrunin96,ZhangZhu18}. We overcome this issue with an original perturbation argument of independent interest, in the spirit of Jensen's approximate maximum principle \cite{Jensen88} for semi-concave functions.

In the Euclidean theory, perturbation arguments usually rely on the affine structure. In \cite{Petrunin96,ZhangZhu18} the authors employ a combination of two perturbation arguments: the first one is required to move the minimum of a given function near to a regular point in the sense of the Alexandrov theory and it is achieved with a small additive perturbation using Perelman's concave functions. The second one uses the existence of concave bi-Lipschitz coordinates as a replacement for the Euclidean affine functions.
In the present situation, these techniques seem out of reach. Indeed, the existence of concave auxiliary functions heavily relies on the synthetic lower bound on the sectional curvature and the existence of bi-Lipschitz coordinates goes much beyond the present regularity theory of spaces with lower Ricci curvature bounds, even in the ``non-collapsed'' case.\\ 
Perturbations will be constructed using distance functions, by further developing an idea introduced by Cabr\'e \cite{Cabre98}, with a different aim, in the setting of smooth Riemannian manifolds with lower sectional curvature bounds (see also the subsequent \cite{Kim04,WangZhang13}). The proof of the key estimate will require several new ingredients with respect to the case of Riemannian manifolds and will rely, again, on a new interpretation of the interplay between Optimal Transport and Heat Flow through Kuwada's duality, see \autoref{sec:perturbation}. For more details, we refer the reader to the first few pages of  \autoref{sec:perturbation}, where we present the general perturbation strategy, the difficulties of the present setting, and the new ideas developed in the paper.
\medskip

Once local Lipschitz continuity has been established, the Bochner-Eells-Sampson inequality with Hessian type term \eqref{eq:bochnerwithhessianintro} will be proved via bootstrap. The main idea, borrowed from the recent \cite{ZhangZhongZhu19}, is to run the same arguments above changing the squared distance $\dist^2(x,y)$ with any power $\dist^p(x,y)$ for $1<p<\infty$ and then to let $p\to\infty$. While \cite{ZhangZhongZhu19} considers smooth Riemannian manifolds, following the strategy of \cite{ZhangZhu18}, in the present context the analysis is possible thanks to the Wasserstein contractivity of the Heat Flow in any Wasserstein space with distance $W_p$ for $1\le p\le \infty$, see \cite{Savare14}. Combined with a version of Kirchheim's metric differentiability theorem \cite{Kirchheim94}, recently established in \cite{GigliTulyenev21}, this will lead to the sought Bochner-Eells-Sampson inequality in \autoref{sec:Bochner}. 

\medskip

\textit{Two months after posting the preprint of this work on arXiv, the preprint of \cite{GigliLipschitz} appeared on the same archive, containing partially overlapping results.}

\medskip
\medskip

\textbf{Acknowledgements.}
The authors are  supported by the European Research Council (ERC), under the European Union Horizon 2020 research and innovation programme, via the ERC Starting Grant  “CURVATURE”, grant agreement No. 802689. They are grateful to the reviewer for the careful reading and useful comments on a preliminary version of the paper.

\section{Preliminaries}

This preliminary section is meant to introduce some basic material and to set the terminology that we shall adopt in the paper. In \autoref{sub:GARCD} we collect some mostly well known preliminaries about Geometric Analysis on $\RCD(K,N)$ metric measure spaces. In \autoref{subsec:energyharmonics} we introduce the relevant background and terminology about Sobolev and harmonic maps from $\RCD(K,N)$ spaces into $\CAT(0)$ spaces.

\subsection{Geometric Analysis tools on $\RCD(K,N)$ spaces}\label{sub:GARCD}

Throughout the paper, $(X,\dist,\meas)$ will be a metric measure space, i.e. $(X,\dist)$ is a complete and separable metric space endowed with a non-negative Borel measure which is finite on bounded sets.
\smallskip

Given $f:X\to \setR$, we denote with $\lip f $ the slope of $f$ defined as
\begin{equation*}
\lip f (x_{0}):=\limsup_{x\to x_{0}}  \frac{|f(x)-f(x_{0})|}{\dist(x, x_{0})} \; \text{ if $x_{0}$ is not isolated}\, , \quad \lip f(x_{0})=0 \; \text{ otherwise}\, .
\end{equation*}

We denote by $C(X)$ the space of continuous functions and with $\Lip(X)$ (resp. $\Lipb (X),$ $\Lipbs(X)$) the space of Lipschitz functions on $(X, \dist)$ (resp. bounded Lipschitz functions, and Lipschitz functions with bounded support). Analogous notations will be used for the spaces of continuous and Lipschitz functions on an open domain $\Omega\subset X$.\\
We will indicate by $B_r(x)$ the open ball $B_r(x):=\{y\in X\, :\, \dist(x,y)<r\}$ for $r>0$ and $x\in X$ and by $\overline{B_r(x)}:=\{y\in X\, :\, \dist(x,y)\le r\}$ the closed ball, for $x\in X$ and $r>0$.
\smallskip

The Cheeger energy (introduced in \cite{Cheeger99} and further studied in \cite{AGS14}) is defined as the $L^{2}$-lower semicontinuous envelope of the functional $f \mapsto \frac{1}{2} \int_{X} (\lip f)^2 \, \di \meas$, i.e.:
\begin{equation*}
\Ch(f):=\inf \left\{ \liminf_{n\to \infty}  \frac{1}{2} \int_{X} (\lip f_n)^2\, \di \meas \, :\,   f_n\in \Lip(X), \; f_{n}\to f \text{ in }L^{2}(X,\meas) \right \}\, .
\end{equation*}
If $\Ch(f)<\infty$ it was proved in \cite{Cheeger99,AGS14} that the set
$$
G(f):= \left\{g \in L^{2}(X,\meas) \, :\,   \exists \, f_{n} \in \Lip(X), \, f_n\to f, \, \lip f_n \rightharpoonup h\geq g  \text{ in } L^{2}(X,\meas) \right\}
$$
is closed and convex, therefore it
admits a unique element of minimal norm called  \textit{minimal weak upper gradient} and denoted by $|\nabla f|$.   The Cheeger energy can be then  represented by integration as
$$\Ch(f):=\frac{1}{2} \int_{X} |\nabla f|^{2} \di \meas\, . $$
It is not difficult to see that $\Ch$ is a $2$-homogeneous, lower semi-continuous, convex functional on $L^{2}(X,\meas)$, whose proper domain 
${\rm Dom}(\Ch):=\{f \in L^{2}(X,\meas)\,:\, \Ch(f)<\infty\}$ is a dense linear subspace of $L^{2}(X,\meas)$. It then admits an $L^{2}$-gradient flow which is a continuous semigroup of contractions $(P_{t})_{t\geq 0}$ in $L^{2}(X,\meas)$, whose continuous trajectories $t \mapsto P_{t} f$, for $f \in L^{2}(X,\meas)$,  are locally Lipschitz  curves from $(0,\infty)$ with values into  $L^{2}(X,\meas)$. 
\smallskip

Throughout the paper, we will assume that $\Ch: {\rm Dom}(\Ch)\to \setR$ satisfies the parallelogram identity (i.e. it is a quadratic form) or, equivalently, that $P_t:  L^{2}(X,\meas) \to  L^{2}(X,\meas)$ is a linear operator for every $t\geq 0$.  This condition is known in the literature as \textit{infinitesimal Hilbertianity}, after  \cite{AGSDuke, Gigli15}.
If $(X,\dist, \meas)$ is infinitesimally Hilbertian, then ${\rm Dom}(\Ch)$ endowed with the norm $\|f\|_{H^{1,2}}^2:= \|f\|_{L^2}+ 2 \Ch(f)$ is a Hilbert space (in general it is only a Banach space) that will be denoted by $W^{1,2}(X, \dist,\meas)$.
\medskip

The main subject of our investigation will be the so-called $\RCD(K,N)$ metric measure spaces $(X,\dist,\meas)$, i.e. infinitesimally Hilbertian metric measure spaces with Ricci curvature bounded from below and dimension bounded from above, in synthetic sense.\\

The Riemannian Curvature Dimension condition $\RCD(K,\infty)$ was introduced in \cite{AGSDuke}  (see also \cite{Gigli15, AGMS}) coupling the Curvature Dimension condition $\CD(K,\infty)$, previously proposed in \cite{Sturm06I,Sturm06II} and independently in \cite{LottVillani09}, with the infinitesimally Hilbertian assumption, corresponding to the Sobolev space $W^{1,2}$ being Hilbert.\\
The finite dimensional counterparts  led to the notions of $\RCD(K,N)$ and $\RCD^*(K, N)$ spaces, corresponding to $\CD(K, N)$ (resp. $\CD^*(K, N)$, see \cite{BacherSturm10}) satisfying the infinitesimally Hilbertian assumption. The class $\RCD(K,N)$ was proposed in \cite{Gigli15}. The (a priori more general) $\RCD^*(K,N)$ condition was thoroughly analyzed in \cite{EKS} and (subsequently and independently) in \cite{AmbrosioMondinoSavare19} (see also \cite{CavallettiMilman21} for the equivalence between $\RCD^*$ and $\RCD$ in the case of finite reference measure, and \cite{Li-AMPA} for the extension to the case of $\sigma$-finite measures).

\medskip

We avoid giving a detailed introduction to this notion, addressing the reader to the survey \cite{Ambrosio18} and references therein for the relevant background. Below we recall some of the main properties that will be relevant for our purposes.
\smallskip

Unless otherwise stated, from now on we assume that $(X,\dist,\meas)$ is an $\RCD(K,N)$ metric measure space for some $K\in\setR$ and $1\le N<\infty$.
\medskip

We begin by recalling the notion of Laplacian.
	\begin{definition}\label{def:laplacian}
		The Laplacian $\Delta:D(\Delta)\to L^2(X,\meas)$ is a densely defined linear operator whose domain consists of all functions $f\in W^{1,2}(X,\dist,\meas)$ satisfying 
		\begin{equation*}
		\int hg\di\meas=-\int \nabla h\cdot\nabla f\di\meas \quad\text{for any $h\in W^{1,2}(X,\dist,\meas)$}
		\end{equation*}
		for some $g\in L^2(X,\meas)$. The unique $g$ with this property is denoted by $\Delta f$.
			\end{definition}
	As a consequence of the infinitesimal hilbertianity, it is easily checked that $\Delta$ is an (unbounded) linear operator. More generally, we say that $f\in W^{1,2}_{\loc}(X,\dist,\meas)$ is in the domain of the measure valued Laplacian, and we write $f\in D(\boldsymbol{\Delta})$, if there exists a Radon measure $\mu$ on $X$ such that, for every $\psi\in\Lip_c(X)$, it holds
	\begin{equation*}
	\int\psi\di\mu=-\int\nabla f\cdot\nabla \psi\di\meas\, .
	\end{equation*} 
	In this case we write $\boldsymbol{\Delta}f:=\mu$. If moreover $\boldsymbol{\Delta}f\ll\meas$ with $L^{2}_{\loc}$ density we denote by $\Delta f$ the unique function in $L^{2}_{\loc}(X,\meas)$ such that $\boldsymbol{\Delta}f=\Delta f\, \meas$ and we write $f\in D_{\loc}(\Delta)$. When there is no risk of confusion, we will adopt the simpler notation $\Delta$ even for the measure valued Laplacian.
	
Notice that the definition makes sense even under the assumption that $f\in W^{1,p}_{\loc}(X,\dist,\meas)$ for some $1\le p<\infty$, and we will rely on this observation later.
\medskip

We shall also consider the Laplacian on open sets, imposing Dirichlet boundary conditions. Let us first introduce the local Sobolev space with Dirichlet boundary conditions.

\begin{definition}
Let $(X,\dist,\meas)$ be an $\RCD(K,N)$ metric measure space and let $\Omega\subset X$ be an open and bounded domain. Then we let $W^{1,2}_{0}(\Omega)$ be the $W^{1,2}(X,\dist,\meas)$ closure of $\Lip_c(\Omega,\dist)$.
\end{definition}

 We also introduce the local Sobolev space (i.e. without imposing Dirichlet boundary conditions).

\begin{definition}
Let $(X,\dist,\meas)$ be an $\RCD(K,N)$ metric measure space and let $\Omega\subset X$ be an open and bounded domain. We say that a function $f\in L^2(\Omega,\meas)$ belongs to the local Sobolev space $W^{1,2}(\Omega,\dist,\meas)$ if 
\begin{itemize}
\item[(i)] $f\phi\in W^{1,2}(X,\dist,\meas)$ for any $\phi\in\Lip_c(\Omega,\dist)$;
\item[(ii)] $\abs{\nabla f}\in L^2(X,\meas)$.
\end{itemize}
Above, we intend that $f\phi$ is set to be $0$ outside of $\Omega$. Notice that $\abs{\nabla f}$ is well defined on any $\Omega'\subset \Omega$ (and hence on $\Omega$) as $\abs{\nabla f}:=\abs{\nabla (f\phi)}$ for some $\phi\in \Lip_c(\Omega)$ such that $\phi\equiv 1$ on $\Omega'$. The definition does not depend on the cut-off function by locality.
\end{definition}

\begin{definition}
Let $f\in W^{1,2}(\Omega)$. We say that $f\in D(\Delta,\Omega)$ if there exists a function $h\in L^2(\Omega,\meas)$ such that
\begin{equation*}
\int_{\Omega}gh\di\meas=-\int_{\Omega}\nabla g\cdot\nabla f\di\meas\, ,\quad\text{for any $g\in W^{1,2}_0(\Omega,\dist,\meas)$}\, .
\end{equation*}
\end{definition}

 The Heat Flow $P_t$, previously defined as the $L^2(X,\meas)$-gradient flow of $\Ch$, can be equivalently  characterised by the following property:  for any $u\in L^2(X,\meas)$, the curve $t\mapsto P_tu\in L^2(X,\meas)$ is locally absolutely continuous in $(0,+\infty)$ and satisfies
	\begin{equation*}
	\frac{\di}{\di t}P_tu=\Delta P_tu \quad\text{for $\Leb^1$-a.e. $t\in(0,\infty)$}\, .
	\end{equation*}  
	
	Under our assumptions the Heat Flow provides a linear, continuous and self-adjoint contraction semigroup in $L^2(X,\meas)$. Moreover $P_t$ extends to a linear, continuous and mass preserving operator, still denoted by $P_t$, in all the $L^p$ spaces for $1\le p<+\infty$.  
\medskip
	
	It has been proved in \cite{AGSDuke,AGMS} that, on $\RCD(K,\infty)$ metric measure spaces, the dual heat semigroup $\bar{P}_t:\mathcal{P}_2(X)\to\mathcal{P}_2(X)$ of $P_t$, defined by	
	\begin{equation*}
	\int_X f \di \bar{P}_t \mu := \int_X P_t f \di \mu\qquad\quad \forall \mu\in \mathcal{P}_2(X),\quad \forall f\in \Lipb(X)\, ,
	\end{equation*}
	is $K$-contractive (w.r.t. the $W_2$-distance) and, for $t>0$, maps probability measures into probability measures absolutely continuous w.r.t. $\meas$. Then, for any $t>0$, we can introduce the so called \textit{heat kernel} $p_t:X\times X\to[0,+\infty)$ by
	\begin{equation*}
	p_t(x,\cdot)\meas:=\bar{P}_t\delta_x\, .
	\end{equation*}  
As there is no risk of confusion, we will mostly adopt the notation $P_t$ also for the dual Heat Flow defined on probability measures with finite second order moment. We shall denote by $P_t\delta_x$ the heat kernel (measure) centred at $x\in X$ at time $t>0$.
\smallskip

A key property of the heat kernel follows, namely the so-called stochastic completeness: for any $x\in X$ and for any $t>0$ it holds
\begin{equation}\label{eq:stochcompl}
\int_X p_t(x,y)\di\meas(y)=1\, .
\end{equation}


Let us recall a classical regularity result for solutions of the Poisson equation, see \cite{Jiang12}.

\begin{proposition}
Let $(X,\dist,\meas)$ be an $\RCD(K,N)$ metric measure space. Let $\Omega\subset X$ be an open and bounded domain. If $g\in L^{\infty}(\Omega)$ and $f\in W^{1,2}(\Omega)$ verifies $\Delta f=g$ on $\Omega$, then $f$ is locally Lipschitz.
\end{proposition}

Following \cite{Gigli15} we introduce the notion of Laplacian bound in the sense of distributions.

\begin{definition}\label{def:distributions}
Let $(X,\dist,\meas)$ be an $\RCD(K,N)$ metric measure space and let $\Omega\subset X$ be an open domain. Let $f\in W^{1,2}(\Omega)$ and $\eta\in L^{\infty}(\Omega)$. Then we say that $\Delta f\le \eta$ in the sense of distributions if the following holds. For any non-negative function $\phi\in\Lip_c(\Omega)$,
\begin{equation*}
-\int_{\Omega}\nabla f\cdot\nabla \phi\di\meas\le \int_{\Omega}\phi\eta\di\meas\, .
\end{equation*}
\end{definition}

The following can be obtained with minor modifications from \cite{KinnunenMartio02}. We refer to \cite[Corollary 3.5]{ZhangZhu18} for the case of Alexandrov spaces, the proof works verbatim in the present setting.

\begin{proposition}\label{prop:poissonhelp}
Let $(X,\dist,\meas)$ be an $\RCD(K,N)$ metric measure space. Let $\Omega\subset X$ be an open domain, $g\in L^{\infty}(\Omega)$ and $f\in W^{1,2}_{\loc}(\Omega)\cap C(\Omega)$. Then the following are equivalent:
\begin{itemize}
\item[(i)] $\Delta f\le g$ in the sense of distributions;
\item[(ii)] for any open domain $\Omega'\Subset\Omega$, if $v\in W^{1,2}(\Omega')$ solves $\Delta v=g$ on $\Omega'$ and $f-v\in W^{1,2}_0(\Omega)$, then $v\le f$ in $\Omega'$.
\end{itemize}
\end{proposition}

The following is a slight extension of \cite[Proposition 3.27]{MondinoSemola21}. We consider only constant upper bounds for the Laplacian, but we remove the Lipschitz continuity assumption.

\begin{proposition}\label{prop:hfmeanLapla}
Let $(X,\dist,\meas)$ be an $\RCD(K,N)$ metric measure space for some $K\in\setR$ and $1\le N<\infty$. Let $\Omega\subset X$ be an open domain and let $f:\Omega\to\setR$ be continuous and bounded. If $f\in W^{1,2}(\Omega)$ admits measure valued Laplacian on $\Omega$, with 
\begin{equation}\label{eq:assumedC}
\Delta f\le C\, ,\quad\text{on $\Omega$}
\end{equation}
in the sense of distributions for some constant $C\in\setR$, then the following holds. For any domain $\Omega'\Subset\Omega$ and for any function $g:X\to\setR$ with polynomial growth and such that $f\equiv g$ on $\Omega'$, it holds
\begin{equation}\label{eq:limsuptocompute}
\limsup_{t\to 0}\frac{P_tg(x)-g(x)}{t}\le C\, ,\quad\text{for any $x\in \Omega'$}\, .
\end{equation}
\end{proposition}

\begin{proof}
Thanks to \cite[Lemma 2.53]{MondinoSemola21} the value of the $\limsup$ in \eqref{eq:limsuptocompute} is independent of the chosen extension $g$ of $f$. In particular, thanks to \cite{MondinoNaber19,AmbrosioMondinoSavare}, we can choose a regular cut-off function $\phi:X\to[0,1]$ with compact support inside $\Omega$ and such that $\phi\equiv 1$ on $\Omega'$, Lipschitz and with bounded Laplacian. Then we consider $g:=f\phi$, where it is understood that $g\equiv 0$ outside of $\Omega$.

By the standard Leibniz rule for the measure-valued Laplacian, $g$ admits measure-valued Laplacian on $X$. Moreover,
\begin{equation*}
\Delta g=f\Delta\phi +2\nabla f\cdot\nabla \phi+\phi\Delta f\, .
\end{equation*}
Hence, for any $x\in X$ and for any $t\ge 0$,
\begin{equation}\label{eq:51}
P_tg(x)-g(x)\le \int_0^tP_s\left(f\Delta\phi +2\nabla f\cdot\nabla \phi+\phi\Delta f\right)(x)\di s\, .
\end{equation}
The inequality above can be checked by using the continuity of $g$ and integrating against Test cut-off functions. More precisely, for any sufficiently regular and compactly supported function $\psi:X\to [0,\infty)$ it holds
\begin{align}
\int_X(P_tg-g)\psi\di\meas=&\int_Xg(P_t\psi-\psi)\di\meas\\
=&\int_X\int_0^tg\Delta P_s\psi\di s\di\meas\\
\le & \int_X\left(\int_0^tP_s\left(f\Delta\phi +2\nabla f\cdot\nabla \phi+\phi\Delta f\right)\right)\psi\di\meas\, ,
\end{align}
which is enough to get \eqref{eq:51} by approximation.

Applying \cite[Lemma 2.53]{MondinoSemola21} to the first two terms above, we easily obtain that
\begin{equation*}
\limsup_{t\to 0}\frac{P_tg(x)-g(x)}{t}\le \limsup_{t\to 0}\frac{1}{t}\int_0^tP_s\left(\phi\Delta f\right)(x)\di s\le C\, ,
\end{equation*}
where we employed \eqref{eq:assumedC} and the comparison principle for the Heat Flow to obtain the last inequality.
\end{proof}

We recall that $\RCD(K,N)$ metric measure spaces $(X,\dist,\meas)$ are \emph{strongly rectifiable spaces}, according to \cite[Definition 2.18]{GigliTulyenev21} (see also the previous \cite{GigliPasqualetto21}). This condition means that for any $\eps>0$ there exists a countable covering of $X$ with Borel sets $U_i^{\eps}$, away from a set of $\meas$-measure $0$, such that all these sets are $(1+\eps)$-biLipschitz to Borel subsets of $\setR^n$ with charts $\phi_i^{\eps}$, for a given $n\in\setN$ independent of $i$, and it holds
\begin{equation*}
c_i\Leb^n|_{\phi_i^{\eps}(U_i^{\eps})}\le \left(\phi_{i}^{\eps}\right)_{\sharp}\left(\meas|_{U_i^{\eps}}\right)\le (c_i+\eps)\Leb^n|_{\phi_i^{\eps}(U_i^{\eps})}\, ,
\end{equation*}
for some constant $c_i>0$, where we denoted by $\left(\phi_{i}^{\eps}\right)_{\sharp}$ the pushforward operator for measures through $\phi_{i}^{\eps}$.\\
This statement follows from the combination of \cite{MondinoNaber19}, \cite{KellMondino,DePhilippisMarcheseRindler,GigliPasqualetto21} and \cite{BrueSemolaCPAM}.
\smallskip

We refer again to \cite[Definition 2.18]{GigliTulyenev21} for the notion of \emph{aligned} family of atlas $\mathcal{A}^{\eps_n}$, for a given sequence $\eps_n\downarrow 0$, of a strongly rectifiable space that will be relevant for the subsequent developments of the note.

The unique natural number $1\le n\le N$ such that the above hold will be denoted by \emph{essential dimension} of the $\RCD(K,N)$ metric measure space $(X,\dist,\meas)$.

\subsection{Energy and non-linear harmonic maps}\label{subsec:energyharmonics}

The subject of this paper are harmonic maps from open subsets of $\RCD(K,N)$ metric measure spaces to $\CAT(0)$ spaces. Let us recall the relevant background and terminology. The main reference for this presentation is \cite{GigliTulyenev21}. We refer to \cite{KorevaarSchoen} for the original notions and results in the case when the source space is a smooth Riemannian manifold.

\begin{definition}\label{def:ksloc}
Let $(X,\dist,\meas)$ be a metric measure space, $Y_{\bar{y}}=(Y,\dist_Y,\bar y)$ a pointed complete space, $\Omega\subset X$ open and $u\in L^2(\Omega,Y_{\bar y})$.
\\For every $r>0$ we define  $\mathrm{ks}_{2,r}[u,\Omega]:\Omega\to[0,\infty]$ as
\[
\kse_{2,r}[u,\Omega](x):=
\left\{\begin{array}{ll}
\displaystyle{\Big|\fint_{B_r(x)} \frac{\dist^2_Y(u(x),u(y))}{r^2}\,\di \meas(y) \Big|^{1/2}}&\qquad\text{if }B_r(x)\subset\Omega,\\
0&\qquad\text{otherwise}
\end{array}
\right.
\]
and say that $u\in\ks^{1,2}(\Omega,Y_{\bar y})$ provided
\begin{equation}
\label{eq:defkso}
\mathrm{E}_2^\Omega(u):=\sup\limsup_{r\downarrow0}\int_\Omega \varphi\,\kse^2_{2,r}[u,\Omega]\,\di\meas<\infty,
\end{equation}
where the $\sup$ is taken among all $\varphi:X\to[0,1]$ continuous and such that $\supp(\varphi)$ is compact and contained in $\Omega$.
\end{definition}

We recall that the notion of metric differentiability, originally formulated in \cite{Kirchheim94} for maps $u:\setR^n\to Y$, where $(Y,\dist_y)$ is a metric space, has been extended to the case when the source space is a strongly rectifiable metric measure space in \cite[Definition 3.3]{GigliTulyenev21}, with the introduction of the notion of approximately metrically differentiable map. Moreover, in \cite{GigliTulyenev21} it is proved that maps $u\in\ks^{1,2}(\Omega,Y_{\bar y})$ are $\meas$-a.e. approximately metrically differentiable. We shall denote the (approximate) metric differential of $u$ at a point $x\in X$ as  $\mathrm{md}_{x}u$. The metric differential is a semi-norm on $\setR^n$ whenever $(X,\dist,\meas)$ is an $\RCD(K,N)$ metric measure space with essential dimension $1\le n\le N$. Below we report the relevant definition of metric differentiability in the present setting and address the reader to \cite[Section 3]{GigliTulyenev21} for the more general notion of approximate metric differentiability.

\begin{definition}
Let $(X,\dist,\meas)$ be an $\RCD(K,N)$ metric measure space with essential dimension $1\le n\le N$ and let $(Y,\dist_Y)$ be a complete metric space. Let $\Omega\subset X$ be an open domain and let $u:\Omega\to Y$ be a Lipschitz map. Given a point $x\in\Omega$, we say that $u$ is metrically differentiable at $x$ if the following hold:
\begin{itemize}
\item[(i)] $x$ is an $n$-regular point of $(X,\dist,\meas)$;
\item[(ii)] the functions $g_i:X\to[0,\infty)$ defined by $g_i(y):=\dist_Y(u(y),u(x))/r_i$ considered along the pmGH converging sequence $X_i:=(X,r_i^{-1}\dist,\left(\meas(B_{r_i}(x))\right)^{-1}\meas,x)$ for $r_i\downarrow 0$ converge locally uniformly to a semi-norm $\mathrm{md}_x u:\setR^n\to [0,\infty)$, which is independent of the chosen sequence, up to composition with Euclidean isometries.
\end{itemize}
\end{definition}

\begin{definition}
Given a seminorm $\norm{\cdot}$ on $\setR^n$, we define its $2$-size as 
\begin{equation*}
S_2^2(\norm{\cdot }):=\fint_{B_1(0^n)}\norm{v}^2\di\Leb^n(v)\, .
\end{equation*}
\end{definition}

In the next statement,  we recall the existence of the energy density and its representation in terms of the $2$-size of the metric differential (proved in \cite[Theorem 3.13]{GigliTulyenev21}), and the identification result between the metric differential and differentials $\di u$ of metric valued Sobolev maps in the sense of \cite{GigliPasqualettoSoultanis20} (see also \cite[Definition 4.5]{GigliTulyenev21}) obtained in  \cite[Theorem 4.12]{GigliTulyenev21}.
Such identification plays an important role in some of the subsequent developments of the theory, see for instance \autoref{thm:laplacomp} below.

\begin{theorem}\label{thm:ksomega} Let $(X,\dist,\meas)$ be locally uniformly doubling, supporting a Poincar\'e inequality and strongly rectifiable, $\Omega\subset X$ open and $Y_{\bar{y}}=(Y,\dist_Y,{\bar y})$ a pointed and complete space.

Then the following hold:
\begin{itemize}
\item[(i)] $\ks^{1,2}(\Omega,Y_{\bar y})=W^{1,2}(\Omega,Y_{\bar y})$ as sets.
\item[(ii)] For any $u\in \ks^{1,2}(\Omega,Y_{\bar y})$, there is a function $\mathrm{e}_2[u]\in L^2(X)$, called \emph{$2$-energy density} of $u$, such that 
\[
\kse_{2,r}[u,\Omega]\quad\to\quad \mathrm{e}_2[u]\qquad\text{ $\meas$-a.e.\ and in $L^2_{\loc}(\Omega)$ as $r\downarrow 0$}\, .
\]
\item[(iii)] Any $u\in\ks^{1,2}(\Omega,Y_{\bar y})$ is approximately metrically differentiable $\meas$-a.e.\ in $\Omega$ (here we extend $u$ on the whole $X$ declaring it to be constant outside $\Omega$ to apply the definition of approximate metric differentiability) and it holds
\begin{equation}
\label{eq:endenso}
\mathrm{e}_2[u](x)=S_2(\mathrm{md}_x(u))=S_2(\di u)(x)\qquad\meas\text{-a.e.}\ x\in\Omega.
\end{equation}
\item[(iv)] The functional $\mathrm{E}^\Omega_2:L^2(\Omega,Y_{\bar y})\to[0,+\infty]$ defined by \eqref{eq:defkso} is lower semicontinuous and can be written as
\[
\mathrm{E}_2^\Omega(u)
:=\left\{\begin{array}{ll}
\displaystyle{\int_\Omega\mathrm{e}_2^2[u]\,\di\meas},&\qquad\text{ if }u\in\ks^{1,2}(\Omega,Y_{\bar y}),\\
+\infty,&\qquad\text{ otherwise}.
\end{array}
\right.
\]
\end{itemize}
\end{theorem}

We recall that a metric space $(Y,\dist_Y)$ satisfies the $\CAT(0)$ condition if  for any points $x_0,x_1\in Y$ and for any minimizing geodesic $\gamma:[0,1]\to Y$ connecting them the parallelogram inequality
\begin{equation}
\dist_Y^2(\gamma_t,y)\le (1-t)\dist_Y^2(x_0,y)+t\dist_Y^2(x_1,y)-t(1-t)\dist^2_Y(x_0,x_1)
\end{equation}
holds for any $t\in [0,1]$ and for any $y\in Y$.
\medskip

An outcome of the \emph{universal infinitesimal Hilbertianity} of $\CAT(0)$ spaces, previously proved in \cite{DiMarinoGiglietal21}, is the following representation of the energy density as Hilbert-Schmidt norm of the differential, obtained in \cite[Proposition 6.7]{GigliTulyenev21}.

\begin{proposition}[Energy density as Hilbert-Schmidt norm]
Let $(X,\dist,\meas)$ be a strongly rectifiable space of dimension $n\in\setN$ with uniformly locally doubling measure and supporting a Poincar\'e inequality (in particular these assumptions hold if it is a $\RCD(K,N)$ space for some $K\in\setR$ and $N\in[1,\infty)$) and $\Omega\subset X$ an open set. Let $Y_{\bar{y}}=(Y,\dist_Y,{\bar y})$  be a pointed $\CAT(0)$ space and $u\in\ks^{1,2}(\Omega,Y_{\bar y})$.
\\Then the energy density $\mathrm{e}_2[u]$ admits the representation formula
\[
\mathrm{e}_2[u]=(n+2)^{-\frac12}|\di u|_{\sf HS}\quad\meas\text{-a.e. },
\]
where $n$ is the dimension of $X$.
\end{proposition}

It is possible to consider Sobolev spaces with prescribed boundary values also in the metric-valued case. This is key in order to establish existence of harmonic maps, defined as minimizers of the energy functional with prescribed boundary conditions.

\begin{definition}
Let $(X,\dist,\meas)$ be a metric measure space, $Y_{\bar{y}}=(Y,\dist_Y,{\bar y})$ a pointed complete metric space, $\Omega\subset X$ open and $\bar{u}\in L^2(\Omega,Y_{\bar y})$. Then the space $\ks^{1,2}_{\bar{u}}(\Omega,Y_{\bar y})\subset \ks^{1,2}(\Omega,Y_{\bar y})$ is defined as
\begin{equation*}
\ks^{1,2}_{\bar{u}}(\Omega,Y_{\bar y}):=\left\{u\in \ks^{1,2}(\Omega,Y_{\bar y})\, :\, \dist_Y(\bar{u},u)\in W^{1,2}_0(\Omega) \right\}\, .
\end{equation*}
Moreover, the energy functional $\mathrm{E}_{2,\bar{u}}^\Omega:L^2(\Omega,Y)\to [0,\infty]$ is defined as
\[
\mathrm{E}_{2,\bar{u}}^\Omega=
\left\{\begin{array}{ll}
\int_{\Omega}\mathrm{e}_2^2[u]\,\di\meas&\qquad\text{if $u\in \ks^{1,2}_{\bar{u}}(\Omega,Y_{\bar y})$},\\
+\infty&\qquad\text{otherwise .}
\end{array}
\right.
\]

\end{definition}

We recall from \cite{GigliTulyenev21} that if $(X,\dist,\meas)$ is an $\RCD(K,N)$ metric measure space, $Y_{\bar{y}}=(Y,\dist_Y,{\bar y})$  is a pointed $\CAT(0)$ space, $\Omega\subset X$ is an open domain such that $\meas(X\setminus\Omega)>0$ and $\bar{u}\in\ks^{1,2}(\Omega,Y_{\bar y})$, then the energy functional $\mathrm{E}_{2,\bar{u}}^\Omega$ is convex and lower semicontinuous from $L^2(\Omega,Y)$ to $[0,\infty]$ and it admits a unique minimizer, see \cite[Theorem 6.4]{GigliTulyenev21}. The statement generalizes the previous \cite[Theorem 2.2]{KorevaarSchoen}, dealing with smooth source spaces. We shall call any such minimizer, for a given boundary datum, a harmonic map.
\medskip

When both $(X,\dist)$ and $(Y,\dist_Y)$ are isometric to smooth Riemannian manifolds, $u:X\to Y$ is harmonic and $f:Y\to\setR$ is a $\lambda$-convex function, then the chain rule easily yields that
\begin{equation*}
\Delta (f\circ u)\ge \lambda\abs{\di u}_{\sf HS}^2\, ,
\end{equation*}
see for instance \cite[Chapter 8]{Jost17}.\\
This is generalized to the present setting in \cite[Theorem 4.1]{GigliNobili21}. See also \cite{LytchakStadler20} for the case of maps with Euclidean source and $\CAT(0)$ target.

\begin{theorem}\label{thm:laplacomp}
Let $(X,\dist,\meas)$ be an $\RCD(K,N)$ metric measure space with essential dimension $1\le n\le N$. Let $(Y,\dist_Y)$ be a $\CAT(0)$ space and $\Omega\subset X$ be open and bounded. Let $u\in \ks^{1,2}(\Omega,Y)$ be harmonic and let $f:Y\to\setR$ be a Lipschitz and $\lambda$-convex function, for some $\lambda\in\setR$. Then $f\circ u\in D(\boldsymbol{\Delta},\Omega)$ and $\boldsymbol{\Delta}(f\circ u)$ is a signed Radon measure such that
\begin{equation}\label{eq:Laplaboundcompo}
\boldsymbol{\Delta}(f\circ u)\ge \lambda\abs{\di u}^2_{\sf HS}\meas\, .
\end{equation}
\end{theorem}

With respect to \cite[Theorem 4.1]{GigliNobili21}, there is a significant modification at the right-hand side of \eqref{eq:Laplaboundcompo}, namely in \cite[Equation (4.17)]{GigliNobili21} the bound 
\begin{equation*}
\boldsymbol{\Delta}(f\circ u)\ge \frac{\lambda}{n+2}\abs{\di u}^2_{\sf HS}\meas\, 
\end{equation*}
is obtained under the same assumptions. We report below the detailed computation showing that the stronger bound \eqref{eq:Laplaboundcompo} actually holds and fixing a typo in \cite{GigliNobili21}.
\smallskip

Borrowing the notation from \cite{GigliNobili21}, we start from the bound
\begin{equation*}
\abs{\di u_t}_{\sf HS}^2\le e^{-2\lambda tg}\left(\abs{\di u}_{\sf HS}^2-2t\langle\di g, \di f\circ g\rangle+Ct^2\right)\, ,\quad\text{$\meas$-a.e. in $\Omega$}\, .
\end{equation*}
Above, a nonnegative function $g\in\Lip_{\rm{bs}}(X)$ has been fixed and $u_t(x)\in Y$ denotes the gradient flow of $f$ starting from $u(x)$ at time $tg(x)$, which is well defined for $\meas$-a.e. $x$. 

Integrating over $\Omega$, subtracting and taking the limsup as $t\downarrow 0$, we get 
\begin{equation*}
\limsup_{t\to 0}\frac{\mathrm{E}_2^\Omega(u_t)-\mathrm{E}_2^\Omega(u)}{t}\le \frac{1}{n+2}\int_{\Omega}\left(\lambda g\abs{\di u}_{\sf HS}^2+\langle\di g,\di f\circ u\rangle\right)\di\meas\, ,
\end{equation*}
where we notice that, with respect to \cite[Equation (4.12)]{GigliNobili21}, the denominator $(n+2)$ is \emph{ in front of all the terms in the integral}. Then, since $u$ is harmonic and $u_t$ is a competitor for the variational problem in the definition of harmonic maps, 
\begin{equation*}
\mathrm{E}_2^\Omega(u_t)-\mathrm{E}_2^\Omega(u)\ge 0\, \quad\text{for any $t\ge 0$}\, ,
\end{equation*} 
therefore
\begin{equation*}
\frac{1}{n+2}\int_{\Omega}\left(\lambda g\abs{\di u}_{\sf HS}^2+\langle\di g,\di f\circ u\rangle\right)\di\meas\ge 0\, ,\quad\text{for any $g\in \Lipbs(X)$, $g\ge 0$}\, .
\end{equation*}
Hence
\begin{equation*}
\Delta f\circ u\ge \lambda \abs{\di u}_{\sf HS}^2\, ,
\end{equation*}
in the sense of distributions on $\Omega$. 

\begin{remark}\label{rm:dist2}
Let $(Y,\dist_Y)$ be a $\CAT(0)$ space. Then, for any point $P\in Y$, the function $f_P(\cdot):=\dist_Y^2(P,\cdot)$ is $2$-convex. It follows from \autoref{thm:laplacomp} that
\begin{equation*}
\boldsymbol{\Delta}(f_P\circ u)\ge 2(n+2)\mathrm{e}_2^2[u]\, \meas\, .
\end{equation*}
\end{remark}

Let us recall that non-linear harmonic maps from domains inside $\RCD(K,N)$ metric measure spaces to $\CAT(0)$ spaces are locally H\"older continuous in the interior. This has been proved in full generality in \cite[Corollary 1.7]{Guo21} extending previous results from \cite{Jost97,Lin97}.

\begin{theorem}\label{thm:continuity}
Let $(X,\dist,\meas)$ be an $\RCD(K,N)$ metric measure space. Let $(Y,\dist_Y)$ be a $\CAT(0)$ space and $\Omega\subset X$ be open and bounded. Let $u\in \ks^{1,2}(\Omega,Y)$ be harmonic. Then $u$ is locally H\"older continuous on $\Omega$.

\end{theorem}

We recall the notion of pointwise Lipschitz constant, for a continuous function $u:\Omega\to Y$ defined as
\begin{equation*}
\lip u(x):=\limsup_{y\to x}\frac{\dist_Y(u(x),u(y))}{\dist(x,y)}=\limsup_{r\to 0}\sup_{y\in B_r(x)}\frac{\dist_Y(u(x),u(y))}{r}\, .
\end{equation*}
Notice that since $u$ is continuous, the pointwise Lipschitz constant coincides with the approximate local Lipschitz constant, that, in turn, can be identified with the norm of the metric differential, see \cite{GigliTulyenev21}.

\begin{proposition}\label{prop:lipfinite}
Let $(X,\dist,\meas)$ be an $\RCD(K,N)$ metric measure space and let $(Y,\dist_Y)$ be a $\CAT(0)$ space. Let $\Omega\subset X$ be an open and bounded domain and let $u\in\ks^{1,2}(\Omega,Y)$ be a harmonic map. There exists a constant $c=c(n)$, where $n$ is the essential dimension of $(X,\dist,\meas)$ such that
\begin{equation*}
\lip u(x)=\norm{\mathrm{md}_xu}\le c(n)\abs{\di u}(x)\, ,\quad\text{for $\meas$-a.e. $x \in \Omega$}\, .
\end{equation*}

\end{proposition}

\begin{proof}
The statement has been proved in \cite{GigliTulyenev21}. Notice that the assumption that $u$ is harmonic here plays only a role in the identification between approximate local Lipschitz constant and pointwise Lipschitz constant. Otherwise the statement holds for any Sobolev function.
\end{proof}

\section{Auxiliary results}

The aim of this section is to collect some technical results that will be useful for the proof of the main theorems. Often, the statements are the natural counterpart of analogous results proved in \cite{ZhangZhu18}. The main difference is the use of the Heat Flow as a substitute for averages on balls. This choice makes the connection with Laplacian estimates more transparent and it allows to cover the \emph{weighted} case, corresponding to $\RCD(K,N)$ spaces with essential dimension $1\le n<N$, which is not considered in \cite{ZhangZhu18}.

\medskip
Let us recall a special case of the quadrilateral comparison from \cite[Corollary 2.1.3]{KorevaarSchoen} (see also the previous \cite{Reshetniak}).

\begin{lemma}\label{lemma:elemCAT}
Let $(Y,\dist_Y)$ be a $\CAT(0)$ space. Let $\{P,Q,R,S\}$ be an ordered set of points in $Y$ and let us denote by $Q_m$ the midpoint of $QR$.
Then
\begin{align*}
\left(\dist_{Y}(P,S)-\dist_Y(Q,R)\right)\dist_Y(Q,R)\ge& \left(\dist_Y^2(P,Q_m)-\dist_Y^2(P,Q)-\dist_Y^2(Q_m,Q)\right)\\
&+\left(\dist_Y^2(S,Q_m)-\dist_Y^2(S,R)-\dist_Y^2(Q_m,R)\right)\,.
\end{align*}
\end{lemma}

The following statement corresponds to \cite[Proposition 5.4]{ZhangZhu18} originally proved for an Alexandrov space with curvature bounded below as source. We report here the strategy of the proof and indicate the minor changes that are needed in the present  more general setting. We consider an $\RCD(K,N)$ metric measure space $(X,\dist,\meas)$ and endow $X\times X$ with the canonical product structure.

\begin{proposition}\label{prop:lapladist}
Let $(X,\dist,\meas)$ be an $\RCD(K,N)$ metric measure space and let $(Y,\dist_Y)$ be a $\CAT(0)$ space. Let $\Omega\subset X$ be an open and bounded domain and let $u\in\ks^{1,2}(\Omega,Y)$ be a harmonic map. Then, denoting
\begin{equation*}
f(x,y)=\dist_Y(u(x),u(y))\, ,
\end{equation*}
it holds that $f\in W^{1,2}(\Omega\times\Omega)$ and
\begin{equation*}
\boldsymbol{\Delta}_{X\times X}f\ge 0\, .
\end{equation*}
 \end{proposition}

\begin{proof}
The proof is divided into three steps. We first prove that, for any $p\in Y$, the function $\Omega\ni x\mapsto \dist(p,u(x))$ is sub-harmonic. Then we check that $f\in W^{1,2}(\Omega\times\Omega)$. Eventually, we combine the two statements to infer that $f$ is sub-harmonic.
\medskip

\textbf{Step 1.} The fact that, for any point $p\in Y$, the function $\Omega\ni x\mapsto \dist_Y(u(x),p)$ is sub-harmonic follows from the convexity of the function $\dist_Y(\cdot,p)$ on $Y$, ensured by the $\CAT(0)$ condition and \autoref{thm:laplacomp}.
\medskip

\textbf{Step 2.} In order to verify that $f\in W^{1,2}(\Omega\times\Omega)$ we just need to observe that, when $x\in\Omega$ is fixed, $f_x(y):=f(x,y)$ is in $W^{1,2}(\Omega)$ and 
\begin{equation*}
\int_{\Omega}\abs{\nabla f_x(y)}^2\di\meas(y)\le \int_{\Omega}\abs{\di u}^2\di\meas\, ,
\end{equation*} 
since $q\mapsto \dist_Y(q,p)$ is a $1$-Lipschitz function. Analogously, for any $y\in\Omega$ fixed, the function $f^y(x):=f(x,y)$ is in $W^{1,2}(\Omega)$ and it holds
\begin{equation*}
\int_{\Omega}\abs{\nabla f^y(x)}^2\di\meas\le \int_{\Omega}\abs{\di u}^2\di\meas\, .
\end{equation*}
The conclusion that $f\in W^{1,2}(\Omega\times \Omega)$ follows from the tensorization of Cheeger energies and Sobolev spaces, see for instance \cite{AGS15}.
\medskip

\textbf{Step 3.} In order to verify that $f$ is sub-harmonic on $\Omega\times \Omega$ we rely again on the tensorization of the Cheeger energies. We consider any non-negative Lipschitz function $\phi$ with compact support in $\Omega\times \Omega$. Then 
\begin{equation*}
\nabla_{X\times X} \phi\cdot \nabla_{X\times X} f(x,y)=\nabla_x\phi\cdot \nabla_xf(x,y)\, +\, \nabla_y \phi\cdot\nabla _yf(x,y)\, ,\quad\text{$\meas\otimes \meas$-a.e. on $\Omega\times \Omega$}\, .
\end{equation*}
Then we compute
\begin{align*}
\int_{\Omega\times\Omega}\nabla_{X\times X} \phi\cdot \nabla_{X\times X} f(x,y)\di\meas(x)\di\meas(y)=&\int_{\Omega}\int_{\Omega}\left(\nabla_x\phi\cdot \nabla_xf(x,y)\right)\di\meas(x)\di\meas(y)\\
&+\int_{\Omega}\int_{\Omega} \left(\nabla_y \phi\cdot\nabla _yf(x,y)\right)\di\meas(x)\di\meas(y)\\
=&-\int_{\Omega}\int_{\Omega}\phi(x,y)\di \Delta _xf(x,y)\\
&-\int_{\Omega}\int_{\Omega}\phi(x,y)\di\Delta_yf(x,y)\\
\le &\,  0\, ,
\end{align*}
where the last inequality follows from Step 1. The sub-harmonicity of $f$ on the product follows.

\end{proof}

Below we will apply the (global) Heat Flow to functions that are only locally defined on some open domain $\Omega\subset X$, where $(X,\dist,\meas)$ is an $\RCD(K,N)$ metric measure space. It is understood that a point $x\in \Omega$ is fixed and we consider any global extension with polynomial growth of $f|_{U_x}$, where $U_x$ is a neighbourhood of $x$. All the statements are not affected by the specific choice of the extension, thanks to \cite[Lemma 2.53, Lemma 2.54]{MondinoSemola21}
\medskip

We will need some asymptotic mean value inequalities, playing the counterpart of \cite[Proposition 3.2, Corollary 4.7, Corollary 5.6, Lemma 6.4]{ZhangZhu18} in the present setting. Here a key difference between the strategy in \cite{ZhangZhu18} is that we will consider the short time asymptotic of the Heat Flow, rather than the asymptotic of averages on balls for small radii. We make a brief digression in order to motivate this choice.
\smallskip

On a smooth $n$-dimensional Riemannian manifold $(M,g)$, for any smooth function $f:M\to\setR$ it holds
\begin{equation*}
\Delta f(x)=\frac{1}{2(n+2)}\left(\lim_{r\to 0}\frac{\fint_{B_r(x)}f\di\mathrm{vol}-f(x)}{r^2}\right)\, ,\quad\text{for any $x\in M$}\, ,
\end{equation*}
where $\Delta$ denotes the Laplace-Beltrami operator.\\
The connection between asymptotics of averages on balls and the Laplacian is more delicate on $\RCD(K,N)$ metric measure spaces. This is well illustrated already at the level of smooth, weighted Riemannian manifolds. Indeed, if $(M,g,e^{-\phi}\mathrm{vol})$ is a smooth weighted Riemannian manifold and, as above, $f:M\to\setR$ is a smooth function, then, denoting by $\Delta_{\phi}$ the weighted Laplacian associated to the metric measure structure $(M,\dist,e^{-\phi}\mathrm{vol})$ and by $\Delta$ the Laplace-Beltrami operator on $(M,g)$, it holds
\begin{equation*}
\Delta_{\phi}f(x)=\Delta f(x)-\nabla\phi(x)\cdot\nabla f(x)\, ,
\end{equation*}
while 
\begin{equation*}
\lim_{r\to 0}\frac{\fint_{B_r(x)}f\di(e^{-\phi}\mathrm{vol})-f(x)}{r^2}=\frac{1}{n+2}\left[\frac{1}{2}\Delta f(x)-\nabla f(x)\cdot\nabla\phi(x)\right]\, .
\end{equation*}
We notice that an extra factor depending on the gradient of the function at $x$ appears in the weighted case. Nevertheless, the identification between asymptotic of averages on balls and Laplacian, up to dimensional constants, holds at critical points. In particular it holds at local minima.

\begin{proposition}\label{prop:lappoin}
Let $(X,\dist,\meas)$ be an $\RCD(K,N)$ metric measure space for some $K\in\setR$ and $1\le N<\infty$ with essential dimension $1\le n\le N$. Let $(Y,\dist_Y)$ be a $\CAT(0)$ space and let $\Omega\subset X$ be an open and bounded domain. Let $u\in \ks^{1,2}(\Omega, Y)$ be harmonic. Then, for $\meas$-a.e. $x\in \Omega$,
\begin{equation}\label{eq:expansionheat}
P_t\dist_Y^2(u(\cdot),u(x))(x)=2 (n+2)\mathrm{e}^2_2[u](x)t+o(t)\, ,\quad\text{as $t\to 0$}\, .
\end{equation}
\end{proposition}

\begin{proof}
We rely on the theory developed in \cite{GigliTulyenev21} and on the convergence and stability theory from \cite{AmbrosioHonda17}. The strategy is to employ a blow-up argument. The rescalings of the map $u$ (more specifically, of the function $\dist_Y(u(\cdot),u(x))$) are controlled thanks to \cite{GigliTulyenev21}, while the short time behaviour of the heat kernel is controlled through a classical blow-up argument near regular points. 
\medskip

\textbf{Step 1.} We consider $n$-regular points $x\in X$. In particular, $\meas$-a.e. point  $x\in X$ is $n$-regular and it holds that
\begin{equation}\label{eq:scaling}
X_i:=\left(X,r_i^{-1}\dist,\frac{1}{\meas(B_{r_i}(x))}\meas,x\right)\to \left(\setR^n,\dist_{\mathrm{eucl}},c_n\Leb^n,0^n\right)\, ,
\end{equation}
in the pointed measured Gromov-Hausdorff sense for any sequence $(r_i)_{i}$ such that $r_i\to 0$.
\smallskip

One of the outcomes of \cite{GigliTulyenev21} is the fact that, given $u\in \ks^{1,2}(\Omega, Y)$ as in the statement, for $\meas$-a.e. $x\in\Omega$ there exists a semi-norm $\mathrm{md}_xu:\setR^n\to[0,\infty)$ such that the sequence of functions $f_i:=\dist_Y(u(\cdot),u(x))/r_i$, considered along the sequence $X_i$ defined in \eqref{eq:scaling}, converges strongly in $L^2_{\loc}$ to $\mathrm{md}_xu:\setR^n\to[0,\infty)$. This statement can be verified combining \cite[Proposition 3.6]{GigliTulyenev21} (see also \cite[Definition 3.3]{GigliTulyenev21} for the definition of approximate metric differentiability) with the continuity of $u$ from \autoref{thm:continuity} to turn approximate limits into full limits.\\
This argument extends \cite{Kirchheim94}, dealing with the case of Euclidean source space, and \cite{Cheeger99}, dealing with the case of scalar valued functions (see also \cite{Ambrosioetalembedding} for a proof tailored to the $\RCD$ setting and the recent \cite{HondaSire21}).
\medskip

\textbf{Step 2.} In this second step we compute
\begin{equation*}
\lim_{t\to 0}\frac{1}{t}P_t\left(\dist_Y^2(u(x),u(\cdot))\right)(x)
\end{equation*}
in terms of $\mathrm{md}_xu$ and relate it with 
\begin{align}\label{eq:fromGT21}
\lim_{r\to 0}\fint_{B_r(x)}\frac{\dist_Y^2(u(y),u(x))}{r^2}\di\meas(y)=&\, \mathrm{e}^2_2[u](x)=S^2_2(\mathrm{md}_x(u))\\
=&\fint_{B_1(0^n)}\abs{\mathrm{md}_xu(v)}^2\di\Leb^n(v)\, .
\end{align}
Let us notice that, denoting by $P_t^{\setR^n}$ the standard Heat Flow on $\setR^n$,
\begin{equation}\label{eq:heatRn}
P_1^{\setR^n}\abs{\mathrm{md}_xu(\cdot)}^2(0^n)=\frac{1}{(4\pi)^{\frac{n}{2}}}\int_{\setR^n}e^{-\frac{\abs{v}^2}{4}}\abs{\mathrm{md}_xu(v)}^2\di\Leb^n(v)\, .
\end{equation}
Moreover, $\mathrm{md}_xu$ is a seminorm, hence 
\begin{equation}\label{eq:homogeneous}
\fint_{B_r(0^n)}\abs{\mathrm{md}_xu(v)}^2\di\Leb^n(v)=r^2\fint_{B_1(0^n)}\abs{\mathrm{md}_xu(v)}^2\di\Leb^n(v)\, ,\quad\text{for any $r>0$}\, .
\end{equation}
The combination of  \eqref{eq:heatRn} with \eqref{eq:homogeneous} gives that
\begin{equation}\label{eq:euclidide}
P_1^{\setR^n}\abs{\mathrm{md}_xu(\cdot)}^2(0^n)=2(n+2)\fint_{B_1(0^n)}\abs{\mathrm{md}_xu(v)}^2\di\Leb^n(v)\, .
\end{equation}
Notice that, for any sequence $(r_i)_{i\in\setN}$ such that $r_i\downarrow 0$, as $f_i$ converge locally in $L^2$ to $\mathrm{md}_xu(\cdot)$, $f_i^2$ converge locally in $L^1$ to $\abs{\mathrm{md}_xu(\cdot)}^2$ along the sequence $X_i$. Therefore, since
\begin{equation*}
\frac{1}{t}P_t\dist_Y^2(u(\cdot),u(x))(x)=P_1^{X_i}f_i^2(x)\,  
\end{equation*}
for $t=r_i^2$ by the scaling properties of the Heat Flow and of the heat kernel, a stability argument using \cite[Lemma 4.11]{AmbrosioBrueSemola19} (see also \cite{AmbrosioHonda17,Ambrosioetalembedding}) proves that 
\begin{equation}\label{eq:stability}
\lim_{t\to 0}\frac{1}{t}P_t\dist_Y^2(u(\cdot),u(x))(x)=P_1^{\setR^n}\abs{\mathrm{md}_xu(\cdot)}^2(0^n)\, .
\end{equation}
The combination of \eqref{eq:stability} with \eqref{eq:euclidide} and \eqref{eq:fromGT21} proves \eqref{eq:expansionheat}.
\end{proof}

\begin{remark}
As a consistency check, we notice that, thanks to \autoref{rm:dist2},
\begin{equation*}
\Delta \dist_Y^2(u(\cdot),u(x))\ge 2 (n+2)\mathrm{e}^2_2[u]\meas\, ,
\end{equation*}
in the sense of distributions on $\Omega$.
Hence, by a slight variant of \autoref{prop:hfmeanLapla}, we can verify that
\begin{align*}
\liminf_{t\to 0}\frac{1}{t}P_t\dist_Y^2(u(\cdot),u(x))(x)
\ge & \liminf_{t\to 0}\frac{1}{t}\int _0^t\left(P_s\Delta \dist_Y^2(u(\cdot),u(x))\right)(x)\di s\\
\ge &\liminf_{t\to 0}\frac{2(n+2)}{t}\int_0^t P_s\left(\mathrm{e}^2_2[u]\right)(x)\di s\, .
\end{align*}
If $x$ is a Lebesgue point for $\mathrm{e}^2_2[u]$, which is true for $\meas$-a.e. $x\in\Omega$, then by \cite[Lemma 2.54]{MondinoSemola21}
\begin{equation*}
P_s\left(\mathrm{e}^2_2[u]\right)(x)\to \mathrm{e}^2_2[u](x)\, , \quad\text{as $s\to 0$}\, .
\end{equation*} 
Hence
\begin{equation*}
\int_0^t P_s\left(\mathrm{e}^2_2[u]\right)(x)\di s=t\mathrm{e}^2_2[u](x)+o(t)\, , \quad\text{as $t\to 0$}\, .
\end{equation*}
Therefore
\begin{equation*}
P_t\dist_Y^2(u(\cdot),u(x))(x)\ge 2(n+2)t\mathrm{e}^2_2[u](x)+o(t)\, , \quad\text{as $t\to 0$}\, ,
\end{equation*}
which yields one of the inequalities in \eqref{eq:expansionheat}.
\end{remark}

The following proposition generalises \cite[Corollary 5.6]{ZhangZhu18} to an $\RCD$ source space.

\begin{proposition}\label{prop:changeP}
Let $(X,\dist,\meas)$ be an $\RCD(K,N)$ metric measure space with essential dimension $1\le n\le N$. Let $(Y,\dist_Y)$ be a $\CAT(0)$ space and let $\Omega\subset X$ be an open and bounded domain. Let $u\in \ks^{1,2}(\Omega, Y)$ be harmonic.\\
Then for $\meas$-a.e. $x_0\in \Omega$ it holds 
\begin{equation}\label{eq:changepoint}
-P_t\dist_Y^2(u(\cdot),P)(x_0)+\dist_Y^2(u(x_0),P)\le - 2 (n+2)\mathrm{e}^2_2[u](x_0)t +o(t)\, ,\quad\text{as $t\to 0$}\, ,
\end{equation}
for any $P\in Y$.
\end{proposition}

\begin{proof}
We consider the function
\begin{equation}\label{eq:transl}
x\mapsto -\dist_Y^2(u(\cdot),P)(x_0)+\dist_Y^2(u(x_0),P) 
\end{equation}
and notice that it vanishes at $x_0$, by its very definition.

We claim that the statement holds for any point $x_0$ such that \autoref{prop:lipfinite} holds, \autoref{prop:lappoin} holds and $x_0$ is a Lebesgue point for the energy density $\mathrm{e}^2_2[u]$.\\
Indeed, since the function $-\dist_Y(\cdot,P)^2$ appearing in \eqref{eq:transl} is $(-2)$--concave by the $\CAT(0)$ condition, by \autoref{thm:laplacomp} we have
\begin{equation*}
\Delta\left(-\dist_Y^2(u(\cdot),P)(x_0)+\dist_Y^2(u(x_0),P) \right)\le -2(n+2)\mathrm{e}^2_2[u]\, \meas\, .
\end{equation*}
Hence, applying the Heat Flow and taking into account the considerations above, arguing as in the proof of \autoref{prop:hfmeanLapla} we get
\begin{equation*}
-P_t\dist_Y^2(u(\cdot),P)(x_0)+\dist_Y^2(u(x_0),P)\le -2(n+2)tP_t\left(\mathrm{e}^2_2[u]\right)(x_0)+o(t)\, ,\quad\text{as $t\to 0$}\, .
\end{equation*}
If $x_0$ is a Lebesgue point for $\mathrm{e}^2_2[u]$, then by \cite[Lemma 2.54]{MondinoSemola21}
\begin{equation*}
-2(n+2)tP_t\left(\mathrm{e}^2_2[u]\right)(x_0)+o(t)=-2(n+2)t(\mathrm{e}^2_2[u](x_0)+o(1))+o(t)\, ,\quad \quad\text{as $t\to 0$}\, .
\end{equation*}
The claimed \eqref{eq:changepoint} follows.

\end{proof}

The following asymptotic inequality generalises \cite[Lemma 6.4]{ZhangZhu18} to the present setting.

\begin{proposition}\label{prop:intw}
Let $(X,\dist,\meas)$ be an $\RCD(K,N)$ metric measure space. Let $(Y,\dist_Y)$ be a $\CAT(0)$ space and let $\Omega\subset X$ be an open and bounded domain. Let $u\in \ks^{1,2}(\Omega, Y)$ be harmonic.\\
For any $z\in\Omega$ and $P\in Y$, let us set
\begin{equation*}
w_{z,P}(\cdot):=\dist_Y^2(u(\cdot),u(z))-\dist_Y^2(u(\cdot),P)+\dist_Y^2(P,u(z))\, .
\end{equation*}
Then for $\meas$-a.e. $x_0\in\Omega$ it holds that
\begin{equation}\label{eq:asymChin}
\limsup_{t\to 0} \frac{1}{t}P_t \left(w_{x_0,P}(\cdot)\right)(x_0)\le 0\, ,
\end{equation}
for every $P\in Y$.
\end{proposition}

\begin{proof}
The statement follows from the combination of \autoref{prop:lappoin} with \autoref{prop:changeP}.
\end{proof}

Let us recall the Laplacian comparison for $\RCD(K,N)$ spaces from \cite{Gigli15}.

\begin{theorem}\label{thm:Laplaciancomparison}
Let $(X,\dist,\meas)$ be an $\RCD(K,N)$ metric measure space for some $K\in\setR$ and $1\le N<\infty$. Fix $p\in X$. Then the function $\dist^2(p,\cdot)$ admits locally measure valued Laplacian bounded from above
\begin{equation}\label{eq:Laplacecomparison}
\Delta \dist^2(p,\cdot)\le f_{K,N}(\dist(\cdot, p))\meas\, ,
\end{equation}
where $f_{K,N}:[0,\infty)$ is a continuous function. When $K=0$, we can take $f_{0,N}:=2N$ and, more in general $f_{K,N}(0)=2N$ for any $K\in\setR$.
\end{theorem}

The following observations will be useful for our future purposes.

\begin{lemma}\label{lemma:elemdistprod}
The following elementary identity holds
\begin{equation*}
\dist_X^2(x,y)=2\dist^2_{X\times X}\left((x,y),D\right)\, , \quad\text{for any $x,y\in X$}\, ,
\end{equation*}
where
\begin{equation*}
D:=\{(z,z)\, : \, z\in X\}
\end{equation*}
is the diagonal in the product space $X\times X$ and 
\begin{equation*}
\dist^2_{X\times X}\left((x,y),D\right):=\inf_{w\in X}\left\{\dist^2\left((x,y),(w,w)\right)\right\}\,.
\end{equation*} 
In particular if $(X,\dist,\meas)$ is an $\RCD(K,N)$ metric measure space, then 
\begin{equation*}
(x,y)\mapsto \dist_X^2(x,y)\, ,
\end{equation*}
as a function on $X\times X$, has locally measure valued Laplacian locally bounded from above by a continuous function.
\end{lemma}

\begin{proof}
The first part of the statement is completely elementary.
\smallskip

In order to prove the second part of the statement, we observe that $X\times X$ is an $\RCD(K,2N)$ metric measure space. Then we recall that the Laplacian comparison for distance functions (squared) from points extends naturally to a Laplacian comparison for distance functions (squared) from closed sets in this setting, see for instance \cite{CavallettiMondino20}. Hence
\begin{equation*}
\Delta_{X\times X}\dist^2(\cdot,\cdot)\le f_{K,2N}\left(\dist(\cdot,\cdot)/2\right)\, 
\end{equation*}
in the sense of distributions on $X\times X$, where $f_{K,2N}$ is the function appearing in the classical Laplacian comparison for distance functions, see \autoref{thm:Laplaciancomparison}.
\end{proof}

We recall that, given probability measures $\mu$ and $\nu$ on $X$, an admissible transport plan between $\mu$ and $\nu$ is a probability measure $\Pi$ on $X\times X$ whose push-forwards via the projections on the first and second marginal are $\mu$ and $\nu$, respectively. More precisely, if $\pi_1:X\times X\to X$ is defined by $\pi_1(x,y)=x$ and $\pi_2:X\times X\to X$ is defined by $\pi_2(x,y)=y$, then 
\begin{equation}
\Pi\left(\pi_1^{-1}(A)\right)=\mu(A)\, ,\quad \Pi\left(\pi_2^{-1}(B)\right)=\nu(B)\, ,
\end{equation}
for any Borel sets $A,B\subset X$.

\begin{lemma}\label{lemma:operatordistprod}
Let $(X,\dist,\meas)$ be an $\RCD(K,N)$ metric measure space for some $K\in\setR$ and $1\le N<\infty$. Let $(p,q)\in X\times X$ and $(x_0,y_0)\in X\times X$. For any $r>0$, let $\Pi_r$ be an admissible transport plan between the heat kernels $P_r\delta_{x_0}$ and $P_r\delta_{y_0}$. Then 
\begin{align*}
\limsup_{r\to 0}&\frac{\int \dist^2_{X\times X}\left((x,y), (p,q)\right)\di \Pi_r(x,y)-\dist^2_{X\times X}\left((x_0,y_0), (p,q)\right)}{r}\\
&\le 2f_{(K,N)}\left(\max\{\dist(x_0,p),\dist(y_0,q)\}\right)\, ,
\end{align*}
where $f_{K,N}:[0,\infty)\to (0,\infty)$ is defined in \eqref{eq:Laplacecomparison}.
\end{lemma}

\begin{proof}
By the very definition of the product metric measure structure on $X\times X$ it holds
\begin{equation*}
\dist^2_{X\times X}\left((x,y), (p,q)\right)=\dist^2_X(x,p)+\dist^2_X(y,q)\, .
\end{equation*}
Therefore, denoting by $\Pi$ an admissible plan between probabilities $\mu,\nu\in \mathcal{P}(X)$, it holds
\begin{align*}
\int_{X\times X} \dist^2_{X\times X}\left((x,y), (p,q)\right)\di \Pi(x,y)&= \int _{X\times X}\left( \dist^2_X(x,p)+\dist^2_X(y,q)\right)\di \Pi(x,y)\\
&=\int _X \dist^2_X(x,p)\di\mu(x)+\int _X\dist^2(y,q)\di\nu(y)\, .
\end{align*}
In particular, with the notation of the statement,
\begin{align*}
\int_{X\times X} \dist^2_{X\times X}\left((x,y), (p,q)\right)\di \Pi_r(x,y)&-\dist^2_{X\times X}\left((x_0,y_0), (p,q)\right)\\
= P_r\left(\dist^2(\cdot,p)\right)(x_0)-&\dist^2(x_0,p)+P_r\left(\dist^2(\cdot,q)\right)(y_0)-\dist^2(y_0,q)\, .
\end{align*}
Hence
\begin{align*}
\limsup_{r\to 0}&\frac{1}{r}\left[\int_{X\times X} \dist^2_{X\times X}\left((x,y), (p,q)\right)\di \Pi_r(x,y)-\dist^2_{X\times X}\left((x_0,y_0), (p,q)\right)\right] \\
\le\, & \limsup_{t\to 0}\frac{1}{r}\left[P_r\left(\dist^2(\cdot,p)\right)(x_0)-\dist^2(x_0,p)\right]+\limsup_{r\to 0}\frac{1}{r}\left[P_r\left(\dist^2(\cdot,q)\right)(y_0)-\dist^2(y_0,q)\right]\\
\le\, & 2f_{(K,N)}\left(\max\{\dist(x_0,p),\dist(y_0,q)\}\right)\, ,
\end{align*}
where the last inequality follows from \eqref{eq:Laplacecomparison}.
\end{proof}

\begin{lemma}
Let $(X,\dist,\meas)$ be an $\RCD(K,N)$ metric measure space. Let $p,q\in X$ and let $\Pi_r$ be an admissible transport plan between the heat kernels at time $r>0$ from $p$ and $q$, $P_r\delta_p$ and $P_r\delta_q$, respectively.
Then it holds
\begin{align}
0\le &\liminf_{r\to 0}\frac{\int_{X\times X}\left(\dist^2(x,p)+\dist^2(y,q)\right)\di\Pi_r(x,y)}{r}\label{eq:below}\\ 
\le & \limsup_{r\to 0}\frac{\int_{X\times X}\left(\dist^2(x,p)+\dist^2(y,q)\right)\di\Pi_r(x,y)}{r}\le  2N\, .\label{eq:above}
\end{align}
\end{lemma}

\begin{proof}
The inequality \eqref{eq:below} is trivial, since at the right hand side there is the integral of a non-negative function. 
\smallskip

In order to prove \eqref{eq:above} we notice that, since $\Pi_r$ is an admissible transport plan between $P_r\delta_p$ and $P_r\delta_q$, it holds
\begin{equation}
\int_{X\times X}\left(\dist^2(x,p)+\dist^2(y,q)\right)\di\Pi_r(x,y)=P_r\left(\dist^2(\cdot, p)\right)(p)+P_r\left(\dist^2(\cdot,q)\right)(q)\, .
\end{equation}
The conclusion follows from the Laplacian comparison \autoref{thm:Laplaciancomparison}.
\end{proof}

\section{Approximate maximum principles and perturbation arguments}\label{sec:perturbation}

In this section we construct perturbations of functions with measure-valued Laplacian uniformly bounded from above, with the aim of slightly moving around their minimum points. This ability will be fundamental for the subsequent developments of the note. The idea finds its roots in Jensen's approximate maximum principle for semiconcave functions \cite{Jensen88}. In the Euclidean space, perturbations are constructed thanks to affine functions, that perturb the function at the first order without affecting it at the second order, as they have vanishing Hessian.\\
The principle was later extended to smooth manifolds with sectional curvature bounded from below in \cite{Cabre98}, where the fundamental novelty is that perturbations are constructed through distance functions (squared), as affine functions do not have a natural counterpart. The idea was further developed in \cite{Kim04} and \cite{WangZhang13}, where manifolds with non-negative Ricci curvature and general lower Ricci curvature bounds, respectively, were considered.\\
A different strategy to construct well-behaved perturbations on Alexandrov spaces with lower sectional curvature bounds was proposed independently in the unpublished manuscript \cite{Petrunin96} and later studied in \cite{ZhangZhu12}. The idea is to combine two perturbation arguments. The first one is used to move the minimum at a regular point in the sense of Alexandrov geometry. Then the concave, biLipschitz coordinate functions are employed as a replacement of affine functions in the proof of the Euclidean approximate maximum principle. A key observation is that the concave coordinate functions might affect the behaviour of the original function at the second order (while affine functions do not) but the error has the right sign. 
\smallskip

Here we partially generalize the strategy put forward in \cite{Cabre98,Kim04,WangZhang13} to the setting of $\RCD(K,N)$ spaces metric measure spaces.\\ 
Before doing so, we recall the smooth Riemannian Alexandroff-Bakelman-Pucci (ABP) estimate and briefly outline its proof from \cite{WangZhang13}.

Let us introduce some notation: given a sufficiently regular function $u:\Omega\to\setR$, where $\Omega$ is an open domain inside a smooth Riemannian manifold and $E$ is a compact set, for any $a>0$ we let
\begin{equation}
A_a(E,\Omega,u):=\left\{x\in \overline{\Omega}\, :\, \exists \, y\in E\, :\, \inf_{\overline{\Omega}}\left(u+\frac{a}{2}\dist^2_y\right)=u(x)+\frac{a}{2}\dist^2_y(x) \right\}\, .
\end{equation}

\begin{theorem}\label{thm:ABPsmooth}
Let $(X,\dist,\meas)$ be a smooth metric measure space with weighted Ricci curvature bounded from below by $-K\le 0$. Let $\Omega\subset X$ be an open domain and let $u:\Omega\to\setR$ be a $C^2$ function. Then, for any compact set $E\subset X$ and for any $a>0$ such that 
\begin{equation}
A_{a}(E,\Omega,u)\subset \Omega\, ,
\end{equation} 
it holds:
\begin{equation}\label{eq:estsmooth}
\meas(E)\le \int_{A_a}\exp\left(\frac{1}{2}K\left(\frac{\abs{\nabla u}^2}{a}\right)+\frac{\Delta u}{a}\right)\di \meas\, .
\end{equation}  
\end{theorem}

We outline the strategy, borrowed from \cite{WangZhang13} (see also the previous \cite{Kim04,Cabre98,McCann01,Corderoetal01}). First, we claim that the map $T:X\to X$ defined by $T(x):=\exp_x(a^{-1}\nabla u(x))$ is surjective from $A_a(E,\Omega,u)$ to $E$. This can be easily verified since for any $y\in E$ there exists $x\in A_a(E,\Omega,u)$ such that 
\begin{equation*}
\inf_{\overline{\Omega}}\left(u+\frac{a}{2}\dist^2_y\right)=u(x)+\frac{a}{2}\dist^2_y(x)\, .
\end{equation*}
The first variation implies that 
\begin{equation*}
y=\exp_x(a^{-1}\nabla u(x))\, .
\end{equation*}
Then, we interpolate between the identity map and the map $T$ via the maps $T^t$ defined by 
\begin{equation*}
T^t(x):=\exp(ta^{-1}\nabla u(x))\, , \quad \text{for all } t\in[0,1]\,.
\end{equation*}
Since $T=T^1$ is surjective from $A_a(E,\Omega,u)$ onto $E$, in order to get \eqref{eq:estsmooth} it is sufficient to estimate its Jacobian determinant. We set
\begin{equation*}
J(t,x):=\lim_{r\to 0}\frac{\meas(T^t(B_r(x)))}{\meas(B_r(x))}\, .
\end{equation*}
Since the weighted Ricci curvature of $(X,\dist,\meas)$ is bounded from below by $-K$, setting $l(t,x):=\log J(t,x)$, $l$ is well defined for any $t\in[0,1)$ and it satisfies the second order differential inequality 
\begin{equation*}
l''(t,x)\le \frac{K}{2}\left(\frac{\abs{\nabla u(x)}}{a}\right)^2\, .
\end{equation*}
Moreover, the initial conditions
\begin{equation}\label{eq:initialsmooth}
l(0,x)=0\, ,\quad l'(0,x)=a^{-1}\Delta u(x)
\end{equation}
are met. By ODE comparison, we get a bound for $l$ at all later times and therefore we bound the Jacobian $J(t,x)$ for all later times $t\in[0,1]$. Then \eqref{eq:estsmooth} follows by integration of the Jacobian bound thanks to the area formula. 

\begin{remark}
For suitably chosen $a$ and $y$, the function 
\begin{equation*}
u_{a,y}:=u+\frac{a}{2}\dist^2_y\, ,
\end{equation*}
 is a small perturbation of $u$. In particular, for $a$ very small the minima should converge to the minima of the function $u$. Moreover, if $x$ is a minimum point of $u_{y,a}$, then, neglecting the regularity issues, 
\begin{equation*}
0\le \Delta u_{y,a}(x)=\Delta u(x)+\frac{a}{2}\Delta \dist^2_y(x)\le \Delta u(x)+\frac{a}{2}C_{K,N}\, .
\end{equation*}
Hence 
\begin{equation*}
\Delta u(x)\ge -\frac{a}{2}C_{K,N}\, .
\end{equation*}
This formal computation suggests that \autoref{thm:ABPsmooth} could be useful to move around minimum points of $u$ via the perturbation $u_{a,y}$, still controlling  the Laplacian from below. Notice that, at a qualitative level, \eqref{eq:estsmooth} shows that if $E$ has positive measure, the set of touching points $A_a(E,\Omega,u)$ has positive measure too. 
\end{remark}

A couple of deep difficulties to implement such a strategy in the non-smooth setting of $\RCD(K,N)$ spaces are that:
\begin{itemize}
 \item In a first step, one proves estimates that are \emph{pointwise along the geodesics} (using Jacobi fields computations on each single geodesic) and then, in a second step, such estimates are integrated to get the desired integral bounds. In the non-smooth setting, Jacobi fields computations are not available and typically one works with Wasserstein geodesics (which correspond to ``sets of positive measure'' in the space of geodesics) in order to take advantage of optimal transport tools. 
 \item The initial conditions \eqref{eq:initialsmooth} (which play a key role in the argument) are met since the differential at the origin of the exponential map in a smooth Riemannian manifold is the identity map of the tangent space. This observation has no clear counterpart in the non-smooth setting.
\end{itemize}
To overcome these difficulties,  we introduce several new ingredients.   A fundamental role is played by the Hopf-Lax semigroup and by the fact that it preserves Laplacian upper bounds (compare with the recent \cite{MondinoSemola21} by the authors). A key observation is that the Hopf-Lax semigroup plays a similar role of the exponential map, with the two-fold advantage of not needing smoothness of the ambient space and of communicating well with the synthetic lower Ricci bounds (thanks to a deep duality discovered by  Kuwada  \cite{Kuwada10}, see also \cite{AGS15} for the extension to the $\RCD$ setting). The theory of Regular Lagrangian Flows  \cite{AmbrosioTrevisan} (see also \cite{GigliViolo21} for some useful localised versions) is then a key tool in order to develop an  Eulerian approach based on the continuity equation (well suited for the non-smooth $\RCD$ setting) of  the smooth Lagrangian perspective given by classical Jacobi fields computation along geodesics.
\smallskip

We will not look for the sharpest possible estimate, regarding the dependence on the various parameters, but rather for a quantitative one sufficient for the subsequent purposes of the present note.

\begin{theorem}\label{thm:mainperturbJensen}
Let $(X,\dist,\meas)$ be an $\RCD(K,N)$ metric measure space for some $K\in\setR$ and $1\le N<\infty$. Let $\Omega\subset X$ be an open domain and let $u:\Omega\to\setR$ belong to $W^{1,2}_{\loc}(\Omega)\cap C(\Omega)$ with a strict minimum point on $\Omega$. Let us assume that $u$ admits measure valued Laplacian on $\Omega$ with
\begin{equation}
\Delta u\le L\meas\, \quad\text{on $\Omega$}\, ,
\end{equation}
for some constant $L>0$.
Then, for any compact set $E\subset X$ and for any $a>0$ sufficiently small so that 
\begin{equation}
A_{a}(E,\Omega,u)\subset \Omega\, ,
\end{equation} 
the following estimate holds:
\begin{equation}\label{eq:quantest}
\meas(E)\le C(K,\Omega,a,L)\, \meas(A_a(E,\Omega,u))\, ,
\end{equation}
for some explicit constant $C(K,\Omega,a,L)>0$.\\ 
In particular, if $\meas(E)>0$, then $\meas(A_a(E,\Omega,u))>0$.

\end{theorem}

\begin{proof}
The proof is divided into four steps. We are going to consider a $W_2$-geodesic formally induced by the map $T^t(x):=\exp_x(ta^{-1}\nabla u(x))$ between a suitable probability measure concentrated on $A_{a}(E,\Omega,u)\subset \Omega$ and the normalized restriction of $\meas$ to $E$ and then reversing time. We view this Wasserstein geodesic as a solution of the continuity equation, where the vector field is explicitely determined through the Hopf-Lax semigroup from $u$, see \cite{GigliHan15}. The propagation of Laplacian bounds via the Hopf-Lax semigroup (see the previous \cite{MondinoSemola21}) provides uniform one-sided estimates for the divergence of the vector field along the solution of the continuity equation. To conclude, we will combine these one sided estimates with a regularization procedure from \cite{GigliMosconi15} and with \cite[Proposition 5.3]{GigliViolo21}, which is a local version of the estimates obtained in \cite{AmbrosioTrevisan}, to get uniform bounds for the density $\rho_t$ of the interpolant $\mu_t$, that will ultimately show \eqref{eq:quantest}.
\medskip

\textbf{Step 1.} 
Let us consider a ball $B_R(q)\Subset \Omega$, where $q\in\Omega$ is the strict minimum point of $u$. Up to adding a constant, which does not affect the statement, we assume that $u(q)<0$.

Then, we consider a continuous extension with compact support of $u$ from $B_R(q)$ to $X$ such that $\inf_Xu=u(q)$. Since there is no risk of confusion, we shall keep the notation $u$ also for the global extension of the original function $u:\Omega\to\setR$.    
 Then we let
\begin{equation}
u^c_a(y):=\inf_{x\in X}\left\{u(x)+\frac{a}{2}\dist^2(x,y)\right\}\, .
\end{equation}
For $a>0$ sufficiently small and for $y\in B_R(q)$, it is easy to verify that the infimum can be restricted to the original set of definition, i.e.
\begin{equation}
u^c_a(y):=\inf_{x\in \overline{\Omega}}\left\{u(x)+\frac{a}{2}\dist^2(x,y)\right\}\, .
\end{equation}
Hence 
\begin{equation}
\frac{u^c_a(y)}{a}=\inf_{x\in\overline{\Omega}}\left\{\frac{1}{2}\dist^2(x,y)-\left(\frac{-u(x)}{a}\right)\right\}=\inf_{x\in X}\left\{\frac{1}{2}\dist^2(x,y)-\left(\frac{-u(x)}{a}\right)\right\}\, .
\end{equation}
In particular, the couple $(\phi,\psi):=(-a^{-1}u,a^{-1}u^c_a)$ verifies
\begin{equation}
\phi(x)+\psi(y)\le \frac{1}{2}\dist^2(x,y)\, ,\quad\text{for any $x,y\in X$}\, ,
\end{equation}
the function $\psi$ is $c$-concave and for any $y\in E$ there exists $x\in A_a(E,\Omega,u)$ such that
\begin{equation}\label{eq:almostdual}
\phi(x)+\psi(y)=\frac{1}{2}\dist^2(x,y)\, .
\end{equation}
Let $\mu_0:=\frac{1}{\meas(E)}\meas\res E$ be the probability measure with constant density w.r.t. $\meas$ concentrated on $E$.\\
Since $\psi$ is a $c$-concave function and $\mu_0$ is absolutely continuous with respect to $\meas$ with bounded density and bounded support, we can consider the Wasserstein geodesic $(\mu_s)_{s\in[0,1]}$ induced by exponentiation from $\mu_0$ by $\psi$, borrowing the terminology from \cite{GigliRajalaSturm}.\\ 
We shall denote by $\psi^c$ the function obtained from $\psi$ by $c$-duality, i.e. 
\begin{equation}
\psi^c(x):=\inf_{y\in X}\left\{\frac{\dist^2(x,y)}{2}-\psi(y)\right\}\, , \quad\text{for any $x\in X$}\, . 
\end{equation}
It is easy to verify that $\psi$ and $\psi^c$ are Lipschitz functions with compact support. Moreover, 
\begin{equation}
\psi^c(x)+\psi(y)\le\frac{1}{2}\dist^2(x,y)\, ,\quad\text{for any $x,y\in X$}
\end{equation}
and for any $y\in E$ there exists $x\in A_a(E,\Omega,u)$ such that
\begin{equation}\label{eq:dual}
\psi^c(x)+\psi(y)=\frac{1}{2}\dist^2(x,y)\, .
\end{equation}
The above imply that $-a^{-1}u\le \psi^c$ everywhere and $-a^{-1}u=\psi^c$ on $ A_a(E,\Omega,u)$.\\
Furthermore, $\mu_1$ is concentrated on $A_a(E,\Omega,u)$ and, by \cite{GigliRajalaSturm} and \eqref{eq:dual}, we get that $(\psi,\psi^c)$ is an optimal couple of Kantorovich potentials for the geodesic $(\mu_t)_{t\in[0,1]}$. Again by \cite{GigliRajalaSturm}, $\mu_t\ll\meas$ for every $t\in[0,1)$.

By the general theory of Wasserstein geodesics on geodesic metric spaces, it holds
\begin{align*}
&\mathcal{Q}_t(-\psi)+\mathcal{Q}_{1-t}(-\psi^c)\ge 0, \quad \text{everywhere}\\
&\mathcal{Q}_t(-\psi)+\mathcal{Q}_{1-t}(-\psi^c)=0\, ,\quad\text{on $\mathrm{supp}\, \mu_t$, for any $t\in [0,1]$.}
\end{align*}
Above we employ the standard notation for the semigroup $\mathcal{Q}_t$ which is defined by
\begin{equation}
\mathcal{Q}_tf(x):=\inf_{y\in X}\left\{\frac{\dist^2(x,y)}{2t}+f(y)\right\}\, .
\end{equation}
Moreover, it is easy to verify that
\begin{align}
&\mathcal{Q}_{1-t}(a^{-1}u)\ge \mathcal{Q}_{1-t}(-\psi^c)\, ,  \quad \text{everywhere} \label{eq:conf1}\\
&\mathcal{Q}_{1-t}(a^{-1}u)= \mathcal{Q}_{1-t}(-\psi^c)\, ,\quad\text{on $\mathrm{supp}\, \mu_t$ for every $t\in[0,1]$.} \label{eq:conf2}
\end{align}

\medskip
\textbf{Step 2.} In this step we introduce regularized potentials that induce the Wasserstein geodesic $(\mu_t)_{t\in[0,1]}$, whose existence is guaranteed by \cite{GigliMosconi15}. Indeed, for any $t\in (0,1)$ by \cite[Theorem 3.13]{GigliMosconi15} there exists a Lipschitz function with compact support $\eta_t:X\to\setR$ such that 
\begin{equation}\label{eq:inteta}
-\mathcal{Q}_t(-\psi)\le \eta_t\le \mathcal{Q}_{1-t}(-\psi^c)
\end{equation}\label{eq:roughlapla}
and $\eta_t\in D(\Delta)$ with
\begin{equation}\label{eq:roughlapla}
\norm{\Delta \eta_t}_{L^{\infty}}\le C(t)<\infty\, ,
\end{equation}
where $C:(0,1)\to(0,\infty)$ is a continuous function depending on $\norm{\phi}_{\infty}$, $K$ and $N$, blowing up near to the boundary points.\\ 
The combination of \eqref{eq:inteta} with \eqref{eq:conf1} and \eqref{eq:conf2} proves that 
\begin{equation}\label{eq:pinchetat}
 \eta_t\le \mathcal{Q}_{1-t}(a^{-1}u)\, ,
\end{equation}
everywhere and 
\begin{equation}\label{eq:touchetat}
 \eta_t= \mathcal{Q}_{1-t}(a^{-1}u)\, ,\quad\text{on $\mathrm{supp}\, \mu_t\, ,$ for every $t\in[0,1]$.}
\end{equation}
Arguing as in the proof of \cite[Proposition 3.7]{BrueSemolaCPAM} it is possible to verify that the couple $(\mu_t,-\nabla\eta_t)$ is a solution of the continuity equation
\begin{equation}
\frac{\partial\mu_t}{\partial t}+\div (-\nabla \eta_t\mu_t)=0\, ,\quad\text{for every $t\in (0,1)$}\, .
\end{equation}
Thanks to \eqref{eq:roughlapla}, the theory of Regular Lagrangian Flows on $\RCD$ spaces from \cite{AmbrosioTrevisan} can be applied between any intermediate times $0<s<r<1$ along the solution of the continuity equation $(\mu_t,-\nabla\eta_t)$. 

\medskip
\textbf{Step 3.}
The goal of this step is to uniformly control the positive part of the Laplacian of $\eta_t$ for any $t\in[0,1]$.\\
This follows indeed from the assumption that $u$ has measure valued Laplacian with uniformly bounded positive part combined with \eqref{eq:pinchetat}, \eqref{eq:touchetat}, and with a variant of the argument introduced in \cite[Section 4]{MondinoSemola21}, in turn building on top of  \cite{Kuwada10,AGS15}.

We will first obtain a uniform bound for the positive part of the Laplacian of $\mathcal{Q}_s(a^{-1}u)$, for any $s\in [0,1]$.\\
Let us consider $s\in[0,1]$ and for any $x\in B_R(q)$ we let $x_s\in B_R(q)$ be a point such that 
\begin{equation*}
\mathcal{Q}_s(a^{-1}u)(x)=\inf_{y\in X}\left\{\frac{\dist^2(x,y)}{2s}+a^{-1}u(y)\right\}=\frac{\dist^2(x,x_s)}{2s}+a^{-1}u(x_s)\, 
\end{equation*}
and $x_s$ minimizes the distance from $x$ among all points with the property above.
Then 
\begin{equation}\label{eq:ineqgeneral}
\mathcal{Q}_s(a^{-1}u)(z)\le \frac{\dist^2(z,y)}{2s}+a^{-1}u(y)\, ,\quad\text{for any $z,y\in X$}
\end{equation}
and 
\begin{equation}\label{eq:idr0}
\mathcal{Q}_s(a^{-1}u)(x)=\frac{\dist^2(x,x_s)}{2s}+a^{-1}u(x_s)\, .
\end{equation}
In particular, we can easily deduce the classical estimate for the Hopf-Lax semigroup
\begin{equation}\label{eq:boundHL}
\frac{\dist^2(x,x_s)}{s}\le a^{-1}\abs{u(x)-u(x_s)}\le a^{-1}\mathrm{osc}_{\Omega}u\, ,\quad\text{for any $x\in B_R(q) $}\, .
\end{equation}
For any $r\ge 0$, we let $\Pi_r$ be the optimal transport plan for quadratic cost between probability measures $P_r\delta_x$ and $P_r\delta_{x_s}$, where we denoted by $P_r\delta_{p}$ the heat kernel measure with centre $p\in X$ at time $r$. Then we integrate both sides of \eqref{eq:ineqgeneral} with respect to $\Pi_r$ on $X\times X$ and get
\begin{align}\label{eq:conseqcontract}
P_r\mathcal{Q}_s\left(a^{-1}u\right)(x)\le &\, \frac{W_2^2(P_r\delta_x,P_r\delta_{x_s})}{2s}+a^{-1}P_ru(x_s)\\
\le &\, \frac{e^{-2Kr}\dist^2(x,x_s)}{2s}+a^{-1}P_ru(x_s)\, .
\end{align}
Subtracting \eqref{eq:idr0} from both sides, dividing by $r$ and taking the $\limsup$ as $r\downarrow 0$, taking into account the Wasserstein contractivity of the Heat Flow in the Wasserstein space under the $\RCD(K,\infty)$ condition \cite{AGSDuke,AGMS}, we formally get
\begin{equation*}
\Delta \mathcal{Q}_s\left(a^{-1}u\right)(x)\le -\frac{Kr\dist^2(x,x_s)}{s}+a^{-1}\Delta u(x_s)\, ,
\end{equation*}
which gives a uniform upper bound for the positive part of the Laplacian of $\mathcal{Q}_s\left(a^{-1}u\right)$ as soon as we can uniformly bound $\dist^2(x,x_s)/s$ with respect to $x$ and $s$ and uniformly bound the positive part of $\Delta u$. Next, we make the above argument rigorous.
\smallskip

Recall that by construction of the regularized Kantorovich potentials, $\eta_t$ belongs to the domain of the Laplacian. In particular, see \cite[Lemma 2.56]{MondinoSemola21} for instance, for $\meas$-a.e. $x\in X$ it holds
\begin{equation*}
\Delta \eta_t(x)=\lim_{r\to 0}\frac{P_r\eta_t(x)-\eta_t(x)}{r}\, . 
\end{equation*}
Since $\eta_t(x)=\mathcal{Q}_{1-t}(a^{-1}u)(x)$ for every $x\in \supp(\mu_t)$, we can use \eqref{eq:conseqcontract} and \eqref{eq:idr0} to get
\begin{equation*}
\frac{P_r\eta_t(x)-\eta_t(x)}{r}\le \left(\frac{e^{-2Kr}-1}{2r}\right)\frac{\dist^2(x,x_{1-t})}{1-t}+a^{-1}\frac{P_ru(x_{1-t})-u(x_{1-t})}{r}\, ,\quad\text{for any $r>0$}\,.
\end{equation*} 
Hence, taking the $\limsup$ as $r\to 0$ we infer that
\begin{equation*}
\Delta\eta_t(x)\le -K\frac{\dist^2(x,x_{1-t})}{1-t}+a^{-1}\limsup_{r\to 0}\frac{P_ru(x_s)-u(x_s)}{r}\, ,\quad\text{for $\meas$-a.e. $x\in\supp (\mu_t)$}\, .
\end{equation*}
Since by assumption $\Delta u\le L$ on $\Omega$, employing \autoref{prop:hfmeanLapla} and then \eqref{eq:boundHL}, we obtain that 
\begin{equation}\label{eq:uniformupper}
\Delta\eta_t(x)\le -K\frac{\dist^2(x,x_t)}{1-t}+a^{-1}L\le -Ka^{-1}\mathrm{osc}_{\Omega}u+a^{-1}L\, ,
\end{equation}
for $\meas$-a.e. $x\in\supp(\mu_t)$ and for every $t\in(0,1]$.

\medskip

\textbf{Step 4.} Let us complete the proof of \eqref{eq:quantest} combining the previous ingredients with a limiting argument and \cite[Proposition 5.3]{GigliViolo21}.
\smallskip

We assume $K\le 0$ and set
\begin{equation}
C:=-Ka^{-1}\mathrm{osc}_{\Omega}u+a^{-1}L\, >\, 0\, .
\end{equation}
Notice that, by the very definition of $\mu_0:=\rho_0\meas$ it holds
\begin{equation}
\rho_0\equiv \frac{1}{\meas(E)}\, ,\quad\text{ $\meas$-a.e. on $E$}\, ,\quad\rho_0\equiv 0\, ,\quad \text{$\meas$-a.e. outside of $E$}\, ,
\end{equation}
hence $\norm{\rho_0}_{\infty}=1/\meas(E)$.\\
By the $\RCD(K,N)$ condition (actually essentially non-branching plus $\MCP(K,N)$ would suffice, see \cite[Theorem 1.1]{CavallettiMondino17}, after \cite{Rajala12}), 
\begin{equation}\label{eq:cont0}
\norm{\rho_t}_{L^{\infty}}\le 1/\meas(E)+o(1)\, , \quad\text{as $t\to 0$.} 
\end{equation}

On the other hand, we can apply \cite[Proposition 5.3]{GigliViolo21} (see in particular equation (5.8) therein) in combination with the upper bound for the positive part of the Laplacian of $\eta_t$ in \eqref{eq:uniformupper} to obtain that 
\begin{equation}
\sup_{t\in [s,r]}\norm{\rho_t}_{L^{\infty}}\le \norm{\rho_s}_{L^{\infty}}e^{(r-s)C}\, ,\quad\text{for any $0<s<r<1$}\, .
\end{equation}
Taking the limit as $s\to 0$ and taking into account \eqref{eq:cont0}, we get
\begin{equation}\label{eq:unidensest}
\sup_{t\in [0,r]}\norm{\rho_t}_{L^{\infty}}\le \norm{\rho_0}_{L^{\infty}}e^{rC}\le \norm{\rho_0}_{L^{\infty}}e^C\, ,\quad\text{for any $0\le r<1$}\, .
\end{equation}
To conclude, we observe that the probability measures $\mu_t$ weakly converge to $\mu_1$ as $t\to 1$, as they converge in Wasserstein distance. Then \eqref{eq:unidensest} implies that $\mu_1\ll\meas$ and, setting $\mu_1=\rho_1\meas$, it holds
\begin{equation}
\norm{\rho_1}_{L^{\infty}}\le \frac{e^C}{\meas(E)}\, .
\end{equation}
As $\mu_1$ is concentrated on $A_a(E,\Omega,u)$, we conclude that 
\begin{equation}
\meas(E)\le \meas(A)\, e^C\, .
\end{equation}
\end{proof}

\begin{remark}
It seems likely that a refinement of the proof of \autoref{thm:mainperturbJensen} could lead to a sharper estimate more in the spirit of \eqref{eq:estsmooth}, when we additionally assume that $u$ has Laplacian locally in $L^2$. However this is not needed for the main results of the present note and thus, for the sake or brevity, we leave it for future investigation.
\end{remark}

\begin{corollary}\label{cor:perturb}
Let $(X,\dist,\meas)$ be an $\RCD(K,N)$ metric measure space. Let $\Omega\subset X$ be an open domain and let $u\in W^{1,2}_{\loc}(\Omega)\cap C(\Omega)$ be such that: 
\begin{itemize}
\item[(i)] $u$ admits locally measure valued Laplacian on $\Omega$ with $\Delta u\le L\meas$ in the sense of distributions on $\Omega$ for some constant $L\ge 0$; 
\item[(ii)]$u$ has a strict local minimum at some $x\in \Omega$. 
\end{itemize}
Then for any set of full $\meas$-measure $B\subset \Omega$ and for any natural number $n>0$ there exist $0<a_n<1/n$ and $y_n\in\Omega$ with $\dist(x,y_n)<1/n$ such that the function
\begin{equation}
\Omega\ni z\mapsto u(z)+a_n\dist^2(z,y_n)
\end{equation}
admits a minimum at a point $\bar{x}_n\in B$ with $\dist(\bar{x}_n,x)<1/n$. 
\end{corollary}

\begin{proof}
It is sufficient to apply \autoref{thm:mainperturbJensen} with $E=B_{1/n}(x)$ and $a$ sufficiently small so that, for any $y\in B_{1/n}(x)$ it holds $\dist(x,\bar{x})<1/n$ for any minimum point of the function
\begin{equation}
\Omega\ni z\mapsto u(z)+a\, \dist^2(z,y)\, .
\end{equation}
As \autoref{thm:mainperturbJensen} shows that the set of all possible minimum points when $y$ varies in $B_{1/n}(x)$ has positive measure, clearly it intersects any set of full measure.
\end{proof}

\section{The key propagation theorem}\label{sec:propagation}

In this section we consider an $\RCD(K,N)$ metric measure space $(X,\dist,\meas)$, for some $K\in\setR$ and $1\le N<\infty$,  a $\CAT(0)$ space $(Y,\dist_Y)$, an open domain $\Omega\subset X$ and a harmonic map $u:\Omega\to Y$, as in the statement of \autoref{thm:main}. The goal is to prove a propagation estimate analogous to \cite[Lemma 6.7]{ZhangZhu18}, see \autoref{prop:mainestimate}. This is a key technical step for the proof of the Lipschitz continuity of harmonic maps from $\RCD(K,N)$ spaces to $\CAT(0)$ spaces. 

Let us stress some fundamental differences between our proof and the one in \cite{ZhangZhu18}.
The proof in \cite{ZhangZhu18} uses the second variation formula and parallel transport for Alexandrov spaces \cite{Petrunin98}. It builds on a strategy introduced in \cite{Petrunin96} (see also \cite{ZhangZhu12}). Moreover, it relies on a delicate perturbation argument finding its roots again in \cite{Petrunin96} in the Alexandrov case and similar to the one used in the Euclidean viscosity theory of PDEs to prove the approximate maximum principle for semiconvex functions, see \cite{Jensen88,CaffarelliCabre95}.\\
In order to prove the estimate on non smooth $\RCD$ spaces, we give a completely new argument. The substitute for the second variation formula is the analysis of the interplay between optimal transport and the Heat Flow under lower Ricci curvature bounds finding its roots in \cite{SturmVonRenesse,Kuwada10,AGS15} and further explored by the authors in \cite{MondinoSemola21}. The substitute of the perturbation argument in the Alexandrov theory has been obtained in \autoref{sec:perturbation}.

\smallskip

We first introduce some terminology.
\\For any domain $\Omega'$ compactly contained in $\Omega\subset X$, for any $t>0$ and for any $0\le \lambda\le 1$ we set
\begin{equation}\label{eq:introft}
f_t(x,\lambda):=\inf_{y\in\Omega'}\left\{\frac{e^{-2K\lambda}\dist^2(x,y)}{2t}-\dist_Y(u(x),u(y))\right\}\, .
\end{equation}
We denote by $S_t(x,\lambda)$ the set of those points where the infimum is achieved, i.e.
\begin{equation*}
S_t(x,\lambda):=\left\{y\in \Omega'\, :\, f_t(x,\lambda)=\frac{e^{-2K\lambda}\dist^2(x,y)}{2t}-\dist_Y(u(x),u(y))\right\}\, .
\end{equation*}
Notice that
\begin{equation}\label{eq:singosc}
0\ge f_t(x,\lambda)\ge-\mathrm{osc}_{\overline{\Omega'}}u=-\max_{x,y\in \overline{\Omega'}}\dist_Y(u(x),u(y))\, ,
\end{equation}
where we recall that $u$ is continuous, see \autoref{thm:continuity}, and thus its oscillation is finite on any bounded set.
\smallskip

The following is obtained in \cite[Lemma 6.1]{ZhangZhu18} and the proof works without any modification in the present setting, so we omit it. It is a variant of the classical mild regularity properties for the evolution via the Hopf-Lax semigroup (see for instance \cite{AGS14}) in this non-linear setting.

\begin{lemma}\label{lemma:firstft}
With the notation above, let us set for any open domain $\Omega''\Subset\Omega'$,
\begin{equation}\label{eq:C*t0}
C_*:=2\, \mathrm{osc}_{\overline{\Omega'}}\, u+2\, , \quad t_0:=\frac{\dist^2(\Omega'',\partial\Omega')}{4C_*}\, .
\end{equation}
Then, for each $t\in (0,t_0)$,  the following hold:
\begin{itemize}
\item[(i)] For each $\lambda\in [0,1]$ and $x\in\Omega''$, it holds $S_t(x,\lambda)\neq\emptyset$ and
\begin{equation*}
f_t(x,\lambda)=\min_{\overline{B_{\sqrt{C_*t}}(x)}}\left\{\frac{e^{-2K\lambda}\dist^2(x,y)}{2t}-\dist_Y(u(x),u(y))\right\}\, .
\end{equation*}
\item[(ii)] For each $\lambda\in[0,1]$, the function $x\mapsto f_t(x,\lambda)$ is in $C(\Omega'')\cap W^{1,2}(\Omega'')$ and satisfies the following energy estimate
\begin{equation*}
\int_{\Omega''}\abs{\nabla_x f_t(x,\lambda)}^2\di\meas\le C(K,N)\frac{\diam^2(\Omega')}{t^2}\meas(\Omega'')+C(K,N)\int_{\Omega''}\abs{\di u}_{\mathrm{HS}}^2\di\meas\, .
\end{equation*}
\item[(iii)] For any $x\in\Omega''$, the function $\lambda\mapsto f_t(x,\lambda)$ is Lipschitz with
\begin{equation*}
\abs{f_t(\lambda,x)-f_t(\lambda',x)}\le C_*e^{-2K}\abs{\lambda-\lambda'}\, ,\quad\text{for any $\lambda,\lambda'\in [0,1]$}\, .
\end{equation*}
\item[(iv)] The function $(x,\lambda)\mapsto f_t(x,\lambda)$ is in $C(\Omega''\times[0,1])$ and $W^{1,2}(\Omega''\times[0,1])$, where  $X\times [0,1]$ is endowed with the canonical product structure.
\end{itemize}

\end{lemma}
For $\Omega''\Subset\Omega'$, let $t_0>0$ be given by \eqref{eq:C*t0} and, for all $t\in (0, t_{0})$, consider the function $L_{t,\lambda}:\Omega''\to[0,\infty)$ defined by
\begin{equation*}
L_{t,\lambda}(x):=\dist(x,S_t(x,\lambda))=\min_{y\in S_t(x,\lambda)}\dist(x,y)\, .
\end{equation*}

The following corresponds to \cite[Lemma 6.2]{ZhangZhu18}, whose proof works with no modifications in the present setting, as it relies only on metric arguments, therefore it is omitted.

\begin{lemma}\label{lemma:tech2}
Let $\Omega''\Subset\Omega'$ and $t_0>0$ be given as above. Then for any $t\in (0,t_0)$ it holds that
\begin{itemize}
\item[(i)] The function $(x,\lambda)\mapsto L_{t,\lambda}(x)$ is lower semicontinuous on $\Omega''\times [0,1]$.
\item[(ii)] For each $\lambda \in [0,1]$,
\begin{equation*}
\norm{L_{t,\lambda}}_{L^{\infty}(\Omega'')}\le \sqrt{C_*t}\, .
\end{equation*}
\end{itemize}
\end{lemma}

Below we report \cite[Lemma 6.3]{ZhangZhu18}, whose proof works again with no modifications in the present setting, therefore it is omitted.

\begin{lemma}\label{lemma:tech3}
Let $\Omega''\Subset\Omega'$ and $t_0>0$ be given as above. Then for any $t\in (0,t_0)$ it holds that
\begin{equation}\label{eq:lemmatech3}
\liminf_{\mu\to 0^+}\frac{f_t(x,\lambda+\mu)-f_t(x,\lambda)}{\mu}\ge -e^{-2K\lambda}\frac{K}{t}L^2_{t,\lambda}(x)\, ,
\end{equation}
for any $\lambda\in [0,1)$ and any $x\in\Omega''$.
\end{lemma}

\begin{remark}
Since $\lambda\mapsto f_t(x,\lambda)$ is a Lipschitz function for every $x\in\Omega''$, \eqref{eq:lemmatech3} can be turned into an inequality between derivatives valid $\Leb^1$-a.e. on $(0,1)$.
\end{remark}

In the case when $(X,\dist,\meas)$ is an $\RCD(0,N)$ metric measure space, we can remove the dependence on the additional parameter $\lambda$ and set
\begin{equation*}
f(t,x):=\inf_{y\in\Omega'}\left\{\frac{\dist^2(x,y)}{2t}-\dist_Y(u(x),u(y))\right\}\, .
\end{equation*}
Our goal is to prove that for any $p\in\Omega'$ there exists a neighbourhood $U_p$ of $p$ and a positive time $t_p>0$ such that $f_t$ is superharmonic on $U_p$ for any $0<t<t_p$. See \autoref{prop:mainestimate} below for the general statement when $(X,\dist,\meas)$ is an $\RCD(K,N)$ space for general $K\in\setR$.

\smallskip

As mentioned in the discussion at the beginning of the section, the proof of the propagation estimate \autoref{prop:mainestimate} needs several new ideas with respect to \cite{ZhangZhu18}.

We introduce the first new ingredient avoiding the technicalities for the sake of this presentation. In particular, for simplicity of presentation, we do not care about the maps being defined only on open domains and assume harmonicity to hold globally.\\
Assume that 
\begin{equation}\label{eq:eqtime0}
f(t,x_0)=\frac{\dist^2(x_0,y_0)}{2t}-\dist_Y(u(x_0),u(y_0))\, ,
\end{equation}
and 
\begin{equation}\label{eq:ineqsem}
f(t,x)\le \frac{\dist^2(x,y)}{2t}-\dist_Y(u(x),u(y))\, ,\quad\text{for any $x, y\in X$}\, .
\end{equation}

Let us consider then the evolutions of Dirac deltas through the Heat Flow at points $x_0$ and $y_0$ and denote them by $P_s\delta_{x_0}$, $P_s\delta_{y_0}$ respectively. Let us also denote by $\Pi_{s,x_0,y_0}$ an optimal transport plan for quadratic cost between $P_s\delta_{x_0}$ and $P_s\delta_{y_0}$. Notice that $\Pi_{s,x_0,y_0}$ is a probability measure on $X\times X$ whose first and second marginals are $P_s\delta_{x_0}$ and $P_s\delta_{y_0}$, respectively.
\medskip

Let us integrate both sides in \eqref{eq:ineqsem} with respect to $\Pi_{s,x_0,y_0}$. Then we obtain
\begin{equation}\label{eq:1tobound}
\int _{X\times X}f(t,x)\di\Pi_{s,x_0,y_0}(x,y)\le \int_{X\times X}\left(\frac{\dist^2(x,y)}{2s}-\dist_Y(u(x),u(y))\right)\di\Pi_{s,x_0,y_0}(x,y)\, .
\end{equation}
Notice now that the integrand at the left hand side is independent of $y$. The first marginal of $\Pi_{s,x_0,y_0}$ is $P_s\delta_{x_0}$, hence
\begin{equation*}
\int _{X\times X}f(t,x)\di\Pi_{s,x_0,y_0}(x,y)=\int_Xf(t,x)\di P_s\delta_{x_0}=P_s(f(t,\cdot))(x_0)\, .
\end{equation*}
Since $\Pi_{s_0,x_0,y_0}$ is an optimal transport plan between $P_s\delta_{x_0}$ and $P_s\delta_{y_0}$ for quadratic cost, the $\RCD(0,\infty)$ condition implies that
\begin{equation}\label{eq:2usecontr}
\int_{X\times X}\frac{\dist^2(x,y)}{2t}\di \Pi_{s,x_0,y_0}(x,y)=\frac{1}{2t}W_2^2(P_s\delta_{x_0},P_s\delta_{y_0})\le \frac{1}{2t}\dist^2(x_0,y_0)\, ,
\end{equation}
for any $s>0$.
\smallskip

We are left to bound the term
\begin{equation}\label{eq:explaintobound}
-\int_{X\times X}\dist_Y(u(x),u(y))\di \Pi_{s,x_0,y_0}(x,y)\, .
\end{equation}
There are two possibilities. Either $u(x_0)=u(y_0)$ and then \eqref{eq:explaintobound} is trivially bounded above by $0$, or $u(x_0)\neq u(y_0)$ in which case we can argue as follows. 

Let us apply \autoref{lemma:elemCAT} with points $\{u(x),u(x_0),u(y_0),u(y)\}$.
Then, denoting by $z_0$ the midpoint of $u(x_0)u(y_0)$, we obtain
\begin{align*}
&\left(\dist_{Y}(u(x),u(y))-\dist_Y(u(x_0),u(y_0))\right)\dist_Y(u(x_0),u(y_0))
\\& \qquad  \ge \left(\dist_Y^2(u(x),z_0)-\dist_Y^2(u(x),u(x_0))-\dist_Y^2(z_0,u(x_0))\right)\\
&\qquad \quad +\left(\dist_Y^2(u(y),z_0)-\dist_Y^2(u(y),u(y_0))-\dist_Y^2(z_0,u(y_0))\right)\, ,
\end{align*}
for any $x,y$.

With the notation introduced in \autoref{prop:intw}, this can be rewritten as
\begin{equation*}
\left(\dist_{Y}(u(x),u(y))-\dist_Y(u(x_0),u(y_0))\right)\ge -\frac{1}{\dist_Y(u(x_0),u(y_0))}\left(w_{x_0,z_0}(x)+w_{y_0,z_0}(y)\right)\, ,
\end{equation*}
for any $x,y\in X$. Hence
\begin{equation}\label{eq:rewritten}
-\dist_Y(u(x),u(y))\le -\dist_Y(u(x_0),u(y_0))+\frac{1}{\dist_Y(u(x_0),u(y_0))}\left(w_{x_0,z_0}(x)+w_{y_0,z_0}(y)\right)\, .
\end{equation}
Integrating both sides of \eqref{eq:rewritten} w.r.t. $\Pi_{s,x_0,y_0}$, we can estimate
\begin{align}
\nonumber -\int_{X\times X}&\dist_Y(u(x),u(y))\di \Pi_{s,x_0,y_0}(x,y)\le -\dist_Y(u(x_0),u(y_0))\\
\nonumber &+\frac{1}{\dist_Y(u(x_0),u(y_0))}\int_{X\times X}\left(w_{x_0,z_0}(x)+w_{y_0,z_0}(y)\right)\di \Pi_{s,x_0,y_0}\\
\le -\dist_Y&(u(x_0),u(y_0))+\frac{1}{\dist_Y(u(x_0),u(y_0))}\left(P_sw_{x_0,z_0}(\cdot)(x_0)+P_sw_{y_0,z_0}(\cdot)(y_0)\right)\, ,\label{eq:3bounded}
\end{align}
where the last inequality is due to the fact that the marginals of $\Pi_{s,x_0,y_0}$ are the heat kernel measures centred at $x_0$ and $y_0$ at time $s$.

The combination of \eqref{eq:1tobound} with \eqref{eq:2usecontr}, \eqref{eq:3bounded} and \autoref{prop:intw}, assuming for the sake of this presentation that $x_0$ and $y_0$ are such that the asymptotic estimate \eqref{eq:asymChin} holds, proves that
\begin{equation}
\limsup_{s\to 0}\frac{\left(P_sf(t,\cdot)(x_0)-f(t,x_0)\right)}{s}\le 0\, ,
\end{equation}
that yields, formally, super-harmonicity of $f(t,\cdot)$.
\medskip

We will encounter a key additional difficulty to make rigorous the strategy above, due to the fact that the asymptotic in \autoref{prop:intw} is not available for any point, but rather for $\meas$-a.e. point. In order to deal with this issue, we will rely on \autoref{thm:mainperturbJensen} and \autoref{cor:perturb}. Moreover, there will be error terms to deal with in the case of an $\RCD(K,N)$ space, with general $K\in\setR$. 
\smallskip

For the proof of the key propagation results we will rely on the following technical statement. 

\begin{lemma}\label{lemma:extend}
Let $(X,\dist,\meas)$ be an $\RCD(K,N)$ metric measure space for some $K\in\setR$ and $1\le N<\infty$. Let $(x,y)\in X\times X$ and $U\subset X\times X$ be an open neighbourhood of $(x,y)$. Let $F:U\to\setR$ be a continuous function with a local minimum at $(x,y)$. For any $r>0$, let us denote by $\Pi_r$ an admissible transport plan between $P_r\delta_x$ and $P_r\delta_y$ (which is a probability measure on $X\times X$). Then, for any open set $V\Subset U$ and for any function $G:X\times X\to \setR$ continuous and bounded such that $F\equiv G$ on $V$, it holds 
\begin{equation}\label{eq:liminfPi}
\liminf_{r\to 0}\frac{\int_{X\times X}G(z,w)\di \Pi_r(z,w)-G(x,y)}{r}\ge 0\, .
\end{equation}
\end{lemma}

\begin{proof}
We divide the proof in two steps. First we verify that the value of the $\liminf$ in \eqref{eq:liminfPi} depends only on the behaviour of $G$ in a neighbourhood of $(x,y)$. Then we complete the proof with a particular choice of the global bounded and continuous extension of $F$. 
\medskip

\textbf{Step 1}. It is sufficient to prove the following claim. If $H\in \Cb(X\times X)$ and $H\equiv 0$ in a neighbourhood of $(x,y)$, then
\begin{equation}\label{eq:Hprove}
\lim_{r\to 0}\frac{\int_{X\times X}H(z,w)\di \Pi_r(z,w)}{r}=0\, .
\end{equation}
We can assume without loss of generality that $H\ge 0$ on $X\times X$, up to substituting $H$ with $\abs{H}$. If $H$ is non-negative, continuous, bounded and it vanishes in a neighbourhood of $(x,y)$, then there exist bounded and continuous, non-negative functions $H_1,H_2:X\to [0,\infty)$ identically vanishing in a neighbourhood of $x$ and $y$ respectively, such that 
\begin{equation*}
0\le H(z,w)\le H_1(z)+H_2(w)\, ,\quad\text{for any $z,w\in X$}\, .
\end{equation*}
Since $\Pi_r$ has marginals $P_r\delta_x$ and $P_r\delta_y$ on the first and second component respectively, it holds
\begin{align}\label{eq:split}
\int_{X\times X}H(z,w)\di\Pi_r(z,w)\le& \int_{X\times X}\left(H_1(z)+H_2(w)\right)\di \Pi_r(z,w)\\
= &\, P_rH_1(x)+P_rH_2(y)\, .
\end{align}
Since $H_1$ vanishes identically in a neighbourhood of $x$ and $H_2$ vanishes identically in a neighbourhood of $y$, it follows from \cite[Lemma 2.53]{MondinoSemola21} that
\begin{equation}\label{eq:useMS21}
\lim_{r\to 0}\frac{P_rH_1(x)}{r}=\lim_{r\to 0}\frac{P_rH_2(y)}{r}=0\, .
\end{equation}
Combining \eqref{eq:split} with \eqref{eq:useMS21} we obtain \eqref{eq:Hprove}.
\medskip

\textbf{Step 2.} Given what we obtained in the previous step, it is sufficient to choose any open domain $W\Subset U$ such that 
\begin{equation*}
F(x,y)=\min_{(p,q)\in W}F(p,q)\, .
\end{equation*}
Then, for any $n\in\setN$ we set $F_n:W\to\setR$ by
\begin{equation*}
F_n(p,q):=F(p,q)+\frac{1}{n}\left(\dist^2(x,p)+\dist^2(y,q)\right)\, .
\end{equation*}
For any $n\in\setN$, $F_n$ admits a strict minimum at $(x,y)$ on $W$. Then we can extend $F_n$ to a global function $G_n:X\times X\to\setR$ such that $G_n$ admits a minimum at $(x,y)$. In particular
\begin{equation}\label{eq:estforn}
\liminf_{r\to 0}\frac{\int_{X\times X}G_n(z,w)\di \Pi_r(z,w)-G_n(x,y)}{r}\ge 0\, ,
\end{equation}
since $G_n(z,w)\ge G_n(x,y)$ for any $z,w\in X$ and $\Pi_r$ is a probability measure.\\
Taking into account \autoref{lemma:operatordistprod}, \eqref{eq:estforn} shows that
\begin{equation*}
\liminf_{r\to 0}\frac{\int_{X\times X}G(z,w)\di \Pi_r(z,w)-G(x,y)}{r}\ge -\frac{2N}{n}\, ,\quad\text{for any $n\in\setN$, $n\ge 1$}\, ,
\end{equation*}
where $G$ denotes any bounded and continuous extension of $F$ to $X\times X$. Taking the limit as $n\to\infty$ we obtain \eqref{eq:liminfPi}.

\end{proof}

The statement below is the counterpart of \cite[Lemma 6.7]{ZhangZhu18} in the present setting. It is the main technical tool for the proof of the Lipschitz continuity of harmonic maps from $\RCD(K,N)$ spaces to $\CAT(0)$ spaces.\\
As we anticipated, there are some fundamental differences between our proof and the proof in \cite{ZhangZhu18}. The first one is that in \cite{ZhangZhu18} the authors build on the parallel transport and second variation formula for Alexandrov spaces from \cite{Petrunin98}, in order to estimate second variations in each direction and then average up the estimates. We will estimate averages w.r.t the heat kernel directly (i.e. Laplacians) relying on the interplay between Heat Flow and optimal transport on $\RCD$ spaces. The second fundamental difference is in the perturbation argument pursued in \autoref{sec:perturbation}.

\begin{proposition}\label{prop:mainestimate}
Let $(X,\dist,\meas)$ be an $\RCD(K,N)$ metric measure space for some $K\le 0$ and $1\le N<\infty$ and let $(Y,\dist_Y)$ be a $\CAT(0)$ space. Let $\Omega\subset X$ be an open domain and let $u:\Omega\to Y$ be a harmonic map. Let $f_t(x,\lambda)$ be as in \eqref{eq:introft}. Then for any $p\in \Omega'$ there exist a neighbourhood $U_p=B_{R_p}(p)$ of $p$ and a time $t_p>0$ such that, for any $0<t<t_p$ and any $\lambda\in [0,1]$, the function $U\ni x\mapsto f_t(x,\lambda)$ is a super-solution of the Poisson equation
\begin{equation*}
\Delta f_t(\cdot,\lambda)=-e^{-2K\lambda}\frac{K}{t}L_{t,\lambda}^2\, , \quad \text{on } U_{p}\, .
\end{equation*}
\end{proposition}

\begin{proof}
The proof will be divided into three steps. In Step 1 we set up the contradiction argument via an auxiliary function. In Step 2 we perturb the auxiliary function to achieve a strict minimum at a \emph{sufficiently regular} point. In Step 3 we reach a contradiction with a second variation argument.
\smallskip

\textbf{Step 1}.
Following the proof in \cite{ZhangZhu18}, we notice that it is sufficient to prove that $U\ni x\mapsto f_t(x,\lambda)$ is a super-solution of the Poisson equation
\begin{equation*}
\Delta f_t(\cdot,\lambda)=-e^{-2K\lambda}\frac{K}{t}L_{t,\lambda}^2+\theta \quad \text{ on $U_p$, for any $\theta>0$. }
\end{equation*}
Let us suppose by contradiction that the claim above is not true for some $t>0$, $\lambda\in [0,1]$ and $\theta_0>0$. Then, by \autoref{prop:poissonhelp}, there exists an open domain $B\subset U_p$ such that, denoting by $v\in W^{1,2}(B)$ the solution of the Poisson problem
\begin{equation}\label{eq:solvePoiss}
\Delta v= -e^{-2K\lambda}\frac{K}{t}L_{t,\lambda}^2+\theta_0\, ,\quad\text{on $B$}\, ,
\end{equation}
with $v=f_t(\cdot,\lambda)$ on $\partial B$ (i.e. $v-f_t(\cdot,\lambda)\in W^{1,2}_0(B)$), it holds that 
\begin{equation*}
\min_{x\in B}\left\{f_t(x,\lambda)-v(x)\right\}<0=\min_{x\in \partial B}\left\{f_t(x,\lambda)-v(x)\right\}\, .
\end{equation*} 
In particular, $f_t(\cdot,\lambda)-v$ achieves a minimum in the interior of $B$. Let $\bar{x}\in B$ be any such a minimum point. Define the function $H:B\times U\to \setR$ by
\begin{equation*}
H(x,y):=\frac{e^{-2K\lambda}\dist^2(x,y)}{2t}-\dist_Y(u(x),u(y))-v(x)\, .
\end{equation*}
Let $\bar{y}\in S_{t}(\bar{x},\lambda)\Subset U$ be such that 
\begin{equation*}
\dist(\bar{x},\bar{y})=L_{t,\lambda}(x)\, .
\end{equation*}
By the very definition of $S_{t}(\bar{x},\lambda)$, $H$ has a minimum at $(\bar{x},\bar{y})$. 
\smallskip

\textbf{Step 2}.
We perturb $H$ to achieve a strict minimum at $(\bar{x},\bar{y})$, with a controlled perturbation. To this aim, we consider
\begin{equation}\label{eq:choice delta}
0<\delta_0<C(K,N,\diam(U))\theta_0\, 
\end{equation}
and define the function $H_1:B\times U\to\setR$ by
\begin{equation*}
H_1(x,y):=H(x,y)+\delta_0\dist^2(\bar{x},x)+\delta_0\dist^2(\bar{y},y)\, .
\end{equation*}
Since $(\bar{x},\bar{y})$ is a minimum for $H$, $(\bar{x},\bar{y})$ is the unique strict minimum for $H_1$ in $B\times U$.
\smallskip

The next goal is to perturb again $H_1$ in order to make it achieve its minimum at a point $(\tilde{x},\tilde{y})$ such that $\tilde{x}$ and $\tilde{y}$ are \emph{good points} for \autoref{prop:intw} and $\bar{x}$ is a Lebesgue point for 
\begin{equation*}
x\mapsto e^{-2K\lambda}\frac{K}{t}L_{t,\lambda}^2(x)\, .
\end{equation*}
We notice that $\meas\otimes\meas$-a.e. point in $B\times U$ verifies these two properties.
\smallskip

We wish to construct perturbations with the tools developed in \autoref{sec:perturbation}.\\
Let us observe that $H_1$ has measure valued Laplacian on $B\times U$ with positive part bounded from above by a constant 
\begin{equation*}
C=C\left(\diam U,\diam B,\lambda,t,\delta_0,\norm{L_{t,\lambda}}_{L^{\infty}(B)}\right)\, .
\end{equation*}
In order to verify this claim we consider the various terms appearing in the definition of $H_1$ separately. The function
\begin{equation*}
(x,y)\mapsto \frac{e^{-2K\lambda}\dist^2(x,y)}{2t}
\end{equation*}
has measure valued Laplacian on $X\times X$ with positive part bounded above by a constant $C(K,N,t,\diam U,\diam B)$ thanks to \autoref{lemma:elemdistprod}. The function
\begin{equation}\label{eq:introG}
(x,y)\mapsto \delta_0\dist^2(\bar{x},x)+\delta_0\dist^2(\bar{y},y)=:G(x,y)
\end{equation}
has measure valued Laplacian with positive part bounded by $C(K,N,\delta_0,\diam U,\diam B)$ thanks to the Laplacian comparison and the tensorization of Sobolev spaces. The function
\begin{equation*}
(x,y)\mapsto -\dist_Y(u(x),u(y))
\end{equation*}
has non-positive measure valued Laplacian, thanks to \autoref{prop:lapladist}. Moreover,
\begin{equation*}
\Delta_{x,y}(-v)=-\Delta_xv=e^{-2K\lambda}\frac{K}{t}L_{t,\lambda}^2+\theta_0\, ,
\end{equation*} 
by the very construction of $v$, see \eqref{eq:solvePoiss}. Hence the function $(x,y)\mapsto -v(x)$ has Laplacian bounded by $C(K,N,t,\lambda,\diam U,\diam B)$, thanks to \autoref{lemma:tech2} (ii).
\smallskip

For $\meas\otimes \meas$-a.e. $(z,z')\in B\times U$, $z$ is a Lebesgue point of
\begin{equation}\label{eq:Lebztilde}
z\mapsto e^{-2K\lambda}\frac{K}{t}L_{t,\lambda}^2(z)\, 
\end{equation}
and both $z$ and $z'$ are such that, setting 
\begin{equation*}
w_{q,P}(\cdot):=\dist_Y^2(u(\cdot),u(q))-\dist_Y^2(u(\cdot),P)+\dist_Y^2(P,u(q))\, ,
\end{equation*}
it holds
\begin{equation*}
\limsup_{t\to 0} \frac{1}{t}P_t \left(w_{q,P}(\cdot)\right)(q)\le 0\, ,
\end{equation*}
for every $P\in Y$, for $q=z$ and $q=z'$. This is a consequence of \autoref{prop:intw}.

Combining the above observations with \autoref{cor:perturb} we obtain that, for every $\mu>0$ sufficiently small, there exist $a_{\mu}\in B$, $b_{\mu}\in U$ such that the function $H_{a,b,\mu}:B\times U\to\setR$ defined by
\begin{equation}\label{eq:introG1}
H_{a,b,\mu}(x,y):=H_1(x,y)+\mu\dist^2(a_{\mu},x)+\mu\dist^2(b_{\mu},y)=:H_1(x,y)+G_{1,\mu}(x,y)\, 
\end{equation}
achieves a strict minimum at a point $(\tilde{x}_{\mu},\tilde{y}_{\mu})\in B\times U$ which verifies the good properties mentioned above. We choose any of the triples $(a,b,\mu)$ such that these conditions are met and set $H_{2,\mu}:=H_{a,b,\mu}$.
\smallskip

\textbf{Step 3}. Let us see how to reach a contradiction from the construction of the previous two steps.\\
For the first part of this step, $\mu>0$ will be fixed and we will avoid the subscript $\mu$ for the minimum point $(\tilde{x}_{\mu},\tilde{y}_{\mu})$ to simplify the notation.\\
For any $s>0$, let us denote by $P_s\delta_{\tilde{x}}$ and $P_s\delta_{\tilde{y}}$ the heat kernel measures at time $s$ centred at $\tilde{x}$ and $\tilde{y}$ respectively. Moreover, we denote by $\Pi_s$ the optimal transport plan for quadratic cost between $P_s\delta_{\tilde{x}}$ and $P_s\delta_{\tilde{y}}$. In particular, $\Pi_s$ is a probability measure on $X\times X$ whose first and second marginals are $P_s\delta_{\tilde{x}}$ and $P_s\delta_{\tilde{y}}$, respectively and 
\begin{equation*}
\int_{X\times X}\dist^2(x,y)\di \Pi_s(x,y)\le \int_{X\times X}\dist^2(x,y)\di \Pi(x,y)\, ,
\end{equation*}
for any probability measure $\Pi$ on $X\times X$ with the same marginals.\\ 
From the Wasserstein contractivity of the Heat Flow on $\RCD(K,\infty)$ spaces, see \cite{AGS15}, we get
\begin{equation*}
\int_{X\times X}\dist^2(x,y)\di \Pi_s(x,y)\le  e^{-2Ks}\dist^2(\tilde{x},\tilde{y})\, .
\end{equation*}

From now on, when integrating with respect to heat kernel measures, or optimal transport plans between heat kernel measures, continuous functions that are not globally defined, we always understand that we are choosing global extensions with controlled growth at infinity. The independence of the particular choice is justified by \cite[Lemma 2.53]{MondinoSemola21} and the first step of the proof of \autoref{lemma:extend}.

We set 
\begin{equation}\label{eq:minpoint}
H_{2,\mu}(x,y):=\frac{e^{-2K\lambda}\dist^2(x,y)}{2t}+F_{\mu}(x,y)
\end{equation}
and claim that 
\begin{equation*}
\liminf_{s\to 0}\frac{\int_{X\times X}H_{2,\mu}(x,y)\di \Pi_s(x,y)-H_{2,\mu}(\tilde{x},\tilde{y})}{s}\ge 0\, ,
\end{equation*}
since $(\tilde{x},\tilde{y})$ is a minimum of $H_{2,\mu}$ on $B\times U$. This is indeed a consequence of \autoref{lemma:extend}.
\smallskip

On the other hand, let us estimate
\begin{align}
\nonumber \limsup_{s\to 0}&\frac{\int_{X\times X}H_{2,\mu}(x,y)\di \Pi_s(x,y)-H_{2,\mu}(\tilde{x},\tilde{y})}{s}\\
\le& \frac{e^{-2K\lambda}}{2t}\limsup_{s\to 0}\frac{\int_{X\times X}\dist^2(x,y)\di\Pi_s(x,y)-\dist^2(\tilde{x},\tilde{y})}{s}\\
+&\limsup_{s\to 0}\frac{\int_{X\times X}-\dist_Y(u(x),u(y))\di \Pi_s+\dist_Y(u(\tilde{x}),u(\tilde{y}))}{s}\label{eq:lapla2dY}\\
+&\limsup_{s\to 0}\frac{\int_{X\times X}\left(-v(x)\right)\di\Pi_s(x,y)+v(\tilde{x})}{s}\label{eq:lapla2v}\\
+&\limsup_{s\to 0}\frac{\int_{X\times X}G(x,y)\di\Pi_s(x,y)-G(\tilde{x},\tilde{y})}{s}\label{eq:lapla2G}\\
+&\limsup_{s\to 0}\frac{\int_{X\times X}G_{1,\mu}(x,y)\di\Pi_s(s,y)-G_{1,\mu}(\tilde{x},\tilde{y})}{s}\, ,\label{eq:lapla2G1}
\end{align}
where the functions $G$ and $G_1$ have been introduced in \eqref{eq:introG} and \eqref{eq:introG1} respectively. We estimate each of the five terms separately.\\
We start observing that
\begin{equation}\label{eq:estbyContraction}
 \frac{e^{-2K\lambda}}{2t}\limsup_{s\to 0}\frac{\int_{X\times X}\dist^2(x,y)\di\Pi_s(x,y)-\dist^2(\tilde{x},\tilde{y})}{s}\le -K \frac{e^{-2K\lambda}}{t}\dist^2(\tilde{x},\tilde{y})\, .
\end{equation}
Let us deal with \eqref{eq:lapla2dY}. There are two possibilities. Either $u(\tilde{x})=u(\tilde{y})$ in which case $\Pi_s$ is concentrated on the diagonal of $X\times X$ and therefore it trivially holds
\begin{equation*}
\limsup_{s\to 0}\frac{\int_{X\times X}-\dist_Y(u(x),u(y))\di \Pi_s+\dist_Y(u(\tilde{x}),u(\tilde{y}))}{s}\le 0\, .
\end{equation*}
Otherwise $u(\tilde{x})\neq u(\tilde{y})$ in which case we can argue as follows. 

For any $(x,y)\in X\times X$, let us apply \autoref{lemma:elemCAT} with points $\{u(x),u(\tilde{x}),u(\tilde{y}),u(y)\}$.
Then, denoting by $\tilde{z}$ the midpoint of the segment $u(\tilde{x})u(\tilde{y})$, we obtain
\begin{align*}
&\left(\dist_{Y}(u(x),u(y))-\dist_Y(u(\tilde{x}),u(\tilde{y}))\right)\dist_Y(u(\tilde{x}),u(\tilde{y}))
\\ &\quad \ge \left(\dist_Y^2(u(x),\tilde{z})-\dist_Y^2(u(x),u(\tilde{x}))-\dist_Y^2(\tilde{z},u(\tilde{x}))\right)\\
&\qquad +\left(\dist_Y^2(u(y),\tilde{z})-\dist_Y^2(u(y),u(\tilde{y}))-\dist_Y^2(\tilde{z},u(\tilde{y}))\right)\, , \quad \text{ for any $x,y\in X$.}
\end{align*}
With the notation introduced in \autoref{prop:intw}, this can be rewritten as
\begin{equation*}
\left(\dist_{Y}(u(x),u(y))-\dist_Y(u(\tilde{x}),u(\tilde{y}))\right)\ge -\frac{1}{\dist_Y(u(\tilde{x}),u(\tilde{y}))}\left(w_{\tilde{x},\tilde{z}}(x)+w_{\tilde{y},\tilde{z}}(y)\right)\, ,
\end{equation*}
for any $x,y\in X$. Hence
\begin{equation*}
-\dist_Y(u(x),u(y))\le -\dist_Y(u(\tilde{x}),u(\tilde{y}))+\frac{1}{\dist_Y(u(\tilde{x}),u(\tilde{y}))}\left(w_{\tilde{x},\tilde{z}}(x)+w_{\tilde{y},\tilde{z}}(y)\right)\, .
\end{equation*}
Integrating w.r.t. $\Pi_s$, we can estimate
\begin{align*}
-\int_{X\times X}&\dist_Y(u(x),u(y))\di \Pi_s(x,y)\le -\dist_Y(u(\tilde{x}),u(\tilde{y}))\\
&+\frac{1}{\dist_Y(u(\tilde{x}),u(\tilde{y}))}\int_{X\times X}\left(w_{\tilde{x},\tilde{z}}(x)+w_{\tilde{y},\tilde{z}}(y)\right)\di \Pi_s(x,y)\\
\le& -\dist_Y(u(\tilde{x}),u(\tilde{y}))+\frac{1}{\dist_Y(u(\tilde{x}),u(\tilde{y}))}\left(P_sw_{\tilde{x},\tilde{z}}(\cdot)(\tilde{x})+P_sw_{\tilde{y},\tilde{z}}(\cdot)(\tilde{y})\right)\, .
\end{align*}
Therefore
\begin{align*}
\limsup_{s\to 0}&\frac{-\int_{X\times X}\dist_Y(u(x),u(y))\di \Pi_s(x,y)+\dist_Y(u(\tilde{x}),u(\tilde{y}))}{s} \\
 &\leq \frac{1}{\dist_Y(u(\tilde{x}),u(\tilde{y}))}\left(\limsup_{s\to 0}\frac{1}{s}P_sw_{\tilde{x},\tilde{z}}(\cdot)(\tilde{x})+\limsup_{s\to 0}\frac{1}{s}P_sw_{\tilde{y},\tilde{z}}(\cdot)(\tilde{y})\right)\le 0\, ,
\end{align*}
where the last inequality follows from the good choice of the points $\tilde{x}$ and $\tilde{y}$ in combination with \autoref{prop:intw}.
\smallskip

In order to estimate \eqref{eq:lapla2v} we observe that, by \eqref{eq:solvePoiss}, the assumption that $\tilde{x}$ is a Lebesgue point as in \eqref{eq:Lebztilde} and a minor variant of \cite[Lemma 2.56]{MondinoSemola21} (see also \autoref{prop:hfmeanLapla}),
\begin{align*}
\limsup_{s\to 0}\frac{\int_{X\times X}\left(-v(x)\right)\di\Pi_s(x,y)+v(\tilde{x})}{s}=&\limsup_{s\to 0}\frac{\int_X\left(-v(x)\right)\di P_s\delta_{\tilde{x}}(x)+v(\tilde{x})}{s}\\
=&\limsup_{s\to 0}\frac{-P_sv(\tilde{x})+v(\tilde{x})}{s}=K\frac{e^{-2K\lambda}}{t}L^2_{t,\lambda}(\tilde{x})-\theta_0\, .
\end{align*}
Above, the first equality is justified since the first marginal of $\Pi_s$ is $P_s\delta_{\tilde{x}}$, by the very definition.
\smallskip

We are left to estimate \eqref{eq:lapla2G} and \eqref{eq:lapla2G1} that are dealt with similar arguments. In order to estimate \eqref{eq:lapla2G}, notice that we can pass to the marginals to obtain
\begin{equation*}
\int _{X\times X}G(x,y)\di \Pi_s(x,y)-G(\tilde{x},\tilde{y})=\delta_0\left(P_s\dist^2(\bar{x},\cdot)(\tilde{x})-\dist^2(\bar{x},\tilde{x})+P_s\dist^2(\bar{y},\cdot)(\tilde{y})-\dist^2(\bar{y},\tilde{y})\right)\, .
\end{equation*}
Hence, by the Laplacian comparison,
\begin{equation*}
\limsup_{s\to 0}\frac{\int _{X\times X}G(x,y)\di \Pi_s(x,y)-G(\tilde{x},\tilde{y})}{s}\le \delta_0\, C(K,N,\diam U,\diam B)\, .
\end{equation*}
By similar reasons
\begin{equation*}
\limsup_{s\to 0}\frac{\int_{X\times X}G_{1,\mu}(x,y)\di\Pi_s(s,y)-G_{1,\mu}(\tilde{x},\tilde{y})}{s}\le \mu\,  C(K,N,\diam U,\diam B)\, .
\end{equation*}
Combining the various terms controlled above and taking into account \eqref{eq:minpoint}, we obtain that
\begin{equation}\label{eq:lastalign}
\begin{split}
0\le & \limsup_{s\to 0}\frac{\int_{X\times X}H_{2,\mu}(x,y)\di \Pi_s(x,y)-H_{2,\mu}(\tilde{x},\tilde{y})}{s}\\
\le& -K \frac{e^{-2K\lambda}}{t}\dist^2(\tilde{x},\tilde{y})
+K\frac{e^{-2K\lambda}}{t}L^2_{t,\lambda}(\tilde{x})-\theta_0+\delta_0C(K,N,\diam U,\diam B)\\
&+\mu C(K,N,\diam U,\diam B)\, .
\end{split}
\end{equation}
Now we let $\mu\downarrow 0$. Recall that there is an implicit dependence of the minimum point $(\tilde{x},\tilde{y})$ of $H_{2,\mu}$ from the parameter $\mu$ in all the computations above. Since
\begin{equation*}
H_{2,\mu}=H_1+G_{1,\mu}\, ,
\end{equation*}
where $G_{1,\mu}$ has been introduced in \eqref{eq:introG1} and $H_1$ has a unique strict minimum at $(\bar{x},\bar{y})$ it is elementary to verify that the minimum points $(\tilde{x}_{\mu},\tilde{y}_{\mu})$ of $H_{2,\mu}$ converge to $(\bar{x},\bar{y})$ as $\mu\to 0$.\\ 
Let us rewrite \eqref{eq:lastalign} with the explicit dependence of the minimum points from the parameter $\mu$ as
\begin{align*}
0\le& -K \frac{e^{-2K\lambda}}{t}\dist^2(\tilde{x}_{\mu},\tilde{y}_{\mu})
+K\frac{e^{-2K\lambda}}{t}L^2_{t,\lambda}(\tilde{x}_{\mu})-\theta_0\\
&+\delta_0C(K,N,\diam U,\diam B)
+\mu C(K,N,\diam U,\diam B)\, .
\end{align*}
Then we notice that $\dist^2(\tilde{x}_{\mu},\tilde{y}_{\mu})\to\dist^2(\bar{x},\bar{y})$ as $\mu\to 0$ and that 
\begin{equation*}
\dist^2(\bar{x},\bar{y})=L^2_{t,\lambda}(\bar{x})\le \liminf_{\mu\to 0}L^2_{t,\lambda}(\tilde{x}_{\mu})\, ,
\end{equation*}
by lower semicontinuity, see \autoref{lemma:tech2} (i). Taking the limit as $\mu\downarrow 0$ we obtain
\begin{equation*}
\theta_0\le \delta_0C(K,N,\diam U,\diam B)\, ,
\end{equation*}
which is a contradiction as soon as $\delta_0>0$ is sufficiently small.

\end{proof}

Arguing along the lines of the proof of \cite[Corollary 6.9]{ZhangZhu18}, it is possible to extend \autoref{prop:mainestimate} to any domain $\Omega''\Subset\Omega'$. 

\begin{corollary}\label{cor:tomoreglobal}
With the same notation of \autoref{prop:mainestimate} above, for any open domain $\Omega''\Subset\Omega'$, there exists $t_1>0$ such that for any $0<t<t_1$ and for any $\lambda\in[0,1]$ the function
\begin{equation*}
x\mapsto f_t(x,\lambda) 
\end{equation*} 
is a super-solution of the equation
\begin{equation*}
\Delta f_t(\cdot,\lambda)=-e^{-2K\lambda}\frac{K}{t}L_{t,\lambda}^2\, ,\quad\text{on $\Omega''$}\, .
\end{equation*}

\end{corollary}

\section{Lipschitz continuity}\label{sec:Lip}

In this section we complete the proof of \autoref{thm:main}. The main differences with \cite{ZhangZhu18} are contained in the previous \autoref{sec:perturbation} and \autoref{sec:propagation}, corresponding to the key ingredients for the proof. In this section, we adapt \cite[Section 6]{ZhangZhu18} with minor modifications.
\medskip

We keep the notation of the previous section. In particular we recall that $(X,\dist,\meas)$ is an $\RCD(K,N)$ metric measure space, $\Omega\subset X$ is an open domain, $(Y,\dist_Y)$ is a $\CAT(0)$ space and $u:\Omega\to Y$ is a harmonic map. Moreover, we recall that the functions $f_t(\cdot,\cdot)$ were defined in \eqref{eq:introft}.
\smallskip

We consider local weak solutions and super/subsolutions of the heat equation in space time. Let us introduce some terminology.

Given an open domain $G\subset X$ and an open interval $(a,b)\subset \setR$ we shall denote the domain $G\times I\subset X\times \setR$ as a parabolic cylinder in space time. When $G=B_r(x_0)$ for some $x_0\in X$ and $r>0$ and $I=I_{r}(\lambda_0)=(\lambda_0-r^2,\lambda_0+r^2)$, we use the notation
\begin{equation*}
Q_r(x_0,\lambda_0):=B_r(x_0)\times I_{r}(\lambda_0)=B_r(x_0)\times (\lambda_0-r^2,\lambda_0+r^2)\, .
\end{equation*}

We recall the notion of weak solution of the heat equation and of weak solutions and super/subsolutions adopted in \cite{ZhangZhu18}.

\begin{definition}
Let $Q=G\times I$ be a parabolic cylinder in space time, for some open domain $G\subset X$ and open interval $I\subset (0,\infty)$. A function $g\in W^{1,2}_{\loc}(G)$ is said to be a weak super-solution of the heat equation
\begin{equation*}
\Delta g(x,\lambda)=\frac{\partial g}{\partial \lambda}
\end{equation*} 
if it satisfies
\begin{equation*}
- \int_Q\nabla g\cdot \nabla \phi \di\meas \di\Leb^1\le \int _Q\frac{\partial g}{\partial \lambda}\phi\di\meas\di\Leb^1\, ,
\end{equation*}
for any non-negative function $\phi\in\Lip_c(Q)$. We call $g$ a sub-solution if $-g$ is a super-solution. We call $g$ a solution if it is both a sub-solution and a super-solution.
\end{definition}

The following is \cite[Lemma 6.12]{ZhangZhu18}, that works without any modification in the present context.

\begin{lemma}
Let $Q=G\times I$ be a parabolic cylinder. Let us consider a function $g\in W^{1,2}_{\loc}(G\times I)$. If for $\Leb^1$-a.e. $\lambda\in I$ it holds that $g(\cdot,\lambda)$ is a super-solution of the equation
\begin{equation}
\Delta_xg(\cdot,\lambda)=\frac{\partial g}{\partial\lambda}(\cdot,\lambda)\, ,\quad\text{on $G$}
\end{equation} 
then $g$ is a super-solution of the heat equation
\begin{equation}
\Delta g=\frac{\partial g}{\partial \lambda}\, ,\quad\text{on $Q$}\, .
\end{equation}
\end{lemma}

Arguing as in the proof of \cite[Proposition 6.13]{ZhangZhu18}, combining \autoref{lemma:tech3} with \autoref{cor:tomoreglobal}, we obtain the following.

\begin{proposition}\label{prop:superheat}
Let $\Omega''\Subset \Omega'$ and $t_*:=\min\{t_0,t_1\}$, where $t_0$ and $t_1$ are given by \autoref{lemma:firstft} and \autoref{cor:tomoreglobal}, respectively. Then, for each $t\in (0,t_*)$, the function $f_t(\cdot,\cdot)$ is a super-solution of the heat equation
\begin{equation}
\Delta f_t=\frac{\partial f_t}{\partial\lambda}\, ,\quad\text{on $\Omega''\times(0,1)$}\, .
\end{equation}

\end{proposition}

The strategy to obtain local Lipschitz continuity from the previous results is borrowed from \cite{ZhangZhu18}. We outline the main steps, referring to \cite{ZhangZhu18} for the details of the proofs.\\
In the case $K=0$, where there are no additional error terms, the outcome of our previous constructions is that all the functions $f_t$ are super-solutions of the Laplace equation $\Delta f_t=0$. The strategy is to use this information in combination with a Harnack inequality to promote integral estimates to point-wise estimates. Taking the derivative w.r.t. $t$ of $f_t$ and applying again Harnack's inequality we will obtain uniform estimates on the point-wise Lipschitz constant of $u$. 
\medskip

Let us consider $0<R\le 1$ and let us assume that $B_{2R}(q)\Subset \Omega'$ for some $q\in X$. Let $t_*>0$ be given by \autoref{prop:superheat} for $\Omega''=B_{2R}(q)$ and, for each $t\in (0,t_*)$ and each $\lambda\in (0,1)$, we define the function $x\mapsto \abs{\nabla^-f_t(x,\lambda)}$ on $B_{2R}(q)$ by
\begin{equation*}
\abs{\nabla ^-f_t(x,\lambda)}:=\limsup_{r\to 0}\sup_{y\in B_r(x)}\frac{\left(f_t(x,\lambda)-f_t(y,\lambda)\right)_+}{r}\, ,
\end{equation*} 
where $(a)_+:=\max\{0,a\}$. Set 
\begin{equation*}
\bar{t}:=\min\left\{t_*,\frac{R^2}{64+64\mathrm{osc}_{\overline{\Omega}'}u} \right\}
\end{equation*}
and 
\begin{equation*}
v(t,x,\lambda):=-f_t(x,\lambda)\, ,\quad\text{for any $(t,x,\lambda)\in (0,\bar{t})\times B_{R/2}(q)\times[0,1]$}\, .
\end{equation*}
Below we state the counterpart of \cite[Sublemma 6.16]{ZhangZhu18} in our setting. The proof in \cite{ZhangZhu18} is based only on metric arguments and the assumption that $(X,\dist)$ is an Alexandrov space with curvature bounded from below is never used, therefore it works verbatim in the present setting.

\begin{lemma}\label{lemma:derivat}
With the notation  above, for any $(t,x,\lambda)\in (0,\bar{t},B_{R/4}(q)\times (0,1))$ it holds
\begin{equation*}
\frac{\partial^+}{\partial t}v(t,x,\lambda):=\limsup_{s\to 0}\frac{v(t+s,x,\lambda)-v(t,x,\lambda)}{s}\le \left(\lip u(x)\right)^2+\abs{\nabla ^-f_t(x,\lambda)}^2\, .
\end{equation*}
\end{lemma}

Below we state the counterpart of \cite[Sublemma 6.17]{ZhangZhu18} in our setting. Also in this case, the proof in \cite{ZhangZhu18} is based only on metric arguments, relying only on \autoref{lemma:firstft} and \autoref{lemma:tech2} to obtain the inequality
\begin{equation*}
\abs{v(t,x,\lambda)-v(t',x,\lambda)}\le e^{-2K}\frac{\diam^2(\Omega')}{2a^2}\abs{t-t'}\, ,
\end{equation*}
for any $t,t'\ge a>0$ and any $(x,\lambda)\in B_{R/4}(q)\times (0,1)$.

\begin{lemma}\label{lemma:calHt}
With the notation above, let $\mathcal{H}:(0,\bar{t})\to\setR$ be defined by
\begin{equation}\label{eq:defHt}
\mathcal{H}(t):=\frac{1}{\meas(B_{R/4}(q))}\int_{B_{R/4}(q)\times (\frac14,\frac34)}v(t,x,\lambda)\di\meas(x)\di\Leb^1(\lambda)\, .
\end{equation}
Then the function $\mathcal{H}$ is locally Lipschitz on $(0,\bar{t})$.
\end{lemma}
The next step is to get integral estimates for $(x,\lambda)\mapsto \abs{\nabla ^-f_t(x,\lambda)}^2$ and $x\mapsto \lip^2u(x)$, and to employ them in combination with \autoref{lemma:derivat} to bound the derivative with respect to time of $t\mapsto \mathcal{H}(t)$. 
\smallskip

By \autoref{prop:lipfinite}, $\lip^2u(x)\le c(n)^2\abs{\di u(x)}^2$ for $\meas$-a.e. $x\in\Omega$, where $n$ is the essential dimension of $(X,\dist,\meas)$. By integration
\begin{equation}\label{eq:estlipintegrated}
\int_{B_{R/4}(q)}\lip^2u(x)\di\meas(x)\le C(n)\int_{B_{R/4}(q)}\abs{\di u(x)}^2\di\meas(x)\, .
\end{equation}
The integral bound for $(x,\lambda)\mapsto \abs{\nabla ^-f_t(x,\lambda)}^2$ is based on \autoref{prop:superheat} and the following Harnack inequality for sub-solutions of the heat equation, see \cite{Sturm95II}, or \cite{MarolaMasson13} for a proof under doubling and Poincar\'e conditions.
 
\begin{proposition}\label{prop:Harnack}
Let $G\times I$ be a parabolic cylinder in $X\times\setR$ and let $g$ be a non-negative locally bounded sub-solution of the heat equation $\Delta g=\frac{\partial g}{\partial \lambda}$ on $Q_r\subset G\times I$. Then there exists a constant $C=C(K,N,\diam G)$ such that 
\begin{equation}
\mathrm{ess}\sup_{Q_{r/2}}g\le \frac{C}{r^2\meas(B_r(x))}\int_{Q_r}g\di\meas\di\Leb^1=\bar{C}\fint_{Q_r}g\di\meas\di\Leb^1\, . 
\end{equation}

\end{proposition}

The statement below corresponds to \cite[Lemma 6.15]{ZhangZhu18}, whose proof works in the present setting with no modifications. We outline the strategy addressing the reader to \cite{ZhangZhu18} for more details.

\begin{proposition}\label{prop:estgradf+}
With the notation above, there exists a constant $C=C(K,N,R)>0$ such that 
\begin{equation}\label{eq:intestmain}
\frac{1}{\meas(B_R(q))}\int_{B_{R}(q)\times (\frac14,\frac34)}\abs{\nabla^-f_t(x,\lambda)}^2\di\meas(x)\di\Leb^1(\lambda)\le C(K,N,R)\left(\mathrm{osc}_{\overline{\Omega}'}u\right)^2\, , 
\end{equation}
for any $0<t<t_*$.

\end{proposition}

\begin{proof}
We start recalling that, by \eqref{eq:singosc}, 
\begin{equation}\label{eq:uniboundft}
0\ge f_t(x,\lambda)\ge-\mathrm{osc}_{\overline{\Omega'}}u\, ,
\end{equation}
for any domain $\Omega'\Subset\Omega$ and for any $t>0$ and any $0\le \lambda\le 1$.
\smallskip

The first step in the proof of \cite[Lemma 6.15]{ZhangZhu18} is based on a maximal function argument, that works verbatim in the present setting, as it relies only on the local doubling and Poincar\'e properties of the metric measure space $(X,\dist,\meas)$ that are guaranteed by the $\RCD(K,N)$ condition for $1\le N<\infty$. We refer to \cite{GigliTulyenev21} for similar arguments.
\smallskip

The argument leading to equation (6.52) in the second step of the proof of \cite[Lemma 6.15]{ZhangZhu18} builds on \autoref{lemma:firstft} (ii), (iii), the outcome of the previous step and the uniform boundedness of the local maximal function operator from $L^2$ to $L^2$ on balls, which is a consequence of the uniform local doubling property of $\RCD(K,N)$ spaces. Therefore the proof works verbatim in the present setting.
\smallskip

In order to get equation (6.55) in \cite{ZhangZhu18} the authors apply a parabolic Caccioppoli inequality, which is obtained in \cite[Lemma 4.1]{MarolaMasson13} and holds also in the present setting since it requires only doubling and Poincar\'e conditions, to $-f_t(\cdot,\cdot)$, which is a non-negative sub-solution of the heat equation thanks to \autoref{prop:superheat}. Combined with \eqref{eq:uniboundft}, this leads to 
\begin{equation}\label{eq:propint}
\int_{B_R(q)\times (\frac14,\frac34)}\abs{\nabla f_t(x,\lambda)}^2\di\meas(x)\di\Leb^1(\lambda)\le C(K,N,R)\meas(B_{2R}(q))\left(\mathrm{osc}_{\overline{\Omega}'}u\right)^2\, .
\end{equation}
In order to get equation (6.56) in \cite{ZhangZhu18}, the authors fix $(x,\lambda)\in B_R(q)\times (0,1)$ and observe that the function $\left(f_t(x,\lambda)-f_t(\cdot,\cdot)\right)_+$ is a non-negative sub-solution of the heat equation on $B_R(q)\times (0,1)$, thanks to  \autoref{prop:superheat}. Then they apply the Harnack inequality \autoref{prop:Harnack} to obtain a uniform estimate, see equation (6.56), that is integrated over $B_R(q)\times (\frac14,\frac34)$ and combined with \eqref{eq:propint} and the outcome of the first step to get the sought
\begin{equation}
\int_{B_{R}(q)\times (\frac14,\frac34)}\abs{\nabla^-f_t(x,\lambda)}^2\di\meas(x)\di\Leb^1(\lambda)\le C(K,N,R)\meas(B_{2R}(q))\left(\mathrm{osc}_{\overline{\Omega}'}u\right)^2\, , 
\end{equation}
from which \eqref{eq:intestmain} follows by the uniform local doubling property of $(X,\dist,\meas)$.
\end{proof}

We are ready to prove the main theorem of this paper, that we restate below for the sake of readability.

\begin{theorem}\label{mainthcore}
Let $(X,\dist,\meas)$ be an $\RCD(K,N)$ metric measure space for some $K\in\setR$ and $1\le N<\infty$. Let $(Y,\dist_Y)$ be a $\CAT(0)$ space and let $\Omega\subset X$ be an open domain. Assume that $u:\Omega\to Y$ is a harmonic map. Then for any $0<R\le 1$ there exists a constant $C=C(K,N,R)>0$ such that if $B_{2R}(q)\Subset \Omega$ for some point $q\in X$, then for any $x,y\in B_{R/16}(q)$ it holds
\begin{equation*}
\dist_Y(u(x),u(y))\le C(K,N,R)\left(\left(\fint_{B_{R}(q)}\abs{\di u(z)}^2\di\meas(z)\right)^{\frac{1}{2}}+\mathrm{osc}_{\overline{B}_R(q)}u\right)\dist(x,y)\, .
\end{equation*}
\end{theorem}

\begin{proof}
The proof follows closely the one of \cite[Theorem 1.4]{ZhangZhu18}, without relevant modifications. We outline the strategy.
\smallskip

Let $\mathcal{H}(t)$ be the function defined in \eqref{eq:defHt}. Thanks to \autoref{lemma:calHt} we can apply the dominated convergence theorem to estimate
\begin{align*}
\frac{\di^+}{\di t}\mathcal{H}(t):=&\limsup_{s\downarrow 0}\frac{\mathcal{H}(t+s)-\mathcal{H}(t)}{s}\\
\le &\frac{1}{\meas(B_{R/4}(q))}\int_{B_{R/4}(q)\times (\frac14,\frac34)}\left[\left(\lip u(x)\right)^2+\abs{\nabla ^-f_t(x,\lambda)}^2\right]\di\meas(x)\di\Leb^1(\lambda)\, ,
\end{align*}
where the inequality follows from \autoref{lemma:derivat}. 
\smallskip

Combining \autoref{prop:estgradf+} with \eqref{eq:estlipintegrated}, for any $t\in (0,\bar{t})$, we can estimate 
\begin{equation}\label{eq:boundd+}
\frac{\di^+}{\di t}\mathcal{H}(t)\le C(K,N,R)\left(\fint_{B_{R}(q)}\abs{\di u(x)}^2\di\meas(x)+\left(\mathrm{osc}_{\overline{\Omega}'}u\right)^2\right)\, .
\end{equation}
Borrowing the notation from \cite{ZhangZhu18}, we set
\begin{equation*}
\mathcal{A}_{u,R}:=\left(\fint_{B_{R}(q)}\abs{\di u(x)}^2\di\meas(x)\right)^{\frac{1}{2}}+\mathrm{osc}_{\overline{B}_R(q)}u\, .
\end{equation*}
Then \eqref{eq:boundd+} implies that 
\begin{equation}\label{eq:boundrightderivative}
\frac{\di^+}{\di t}\mathcal{H}(t)\le 2C(K,N,R)\mathcal{A}_{u,R}^2\, ,\quad\text{for any $0\le t\le \bar{t}$}\, .
\end{equation} 
It easily follows from \autoref{lemma:firstft} (i) and the continuity of $u$ that 
\begin{equation*}
\lim_{t\to 0}v(t,x,\lambda)=0\, ,\quad\text{for any $(x,\lambda)\in B_{R/4}(q)\times (0,1)$}\, .
\end{equation*}
Since $v$ is uniformly bounded by \eqref{eq:uniboundft}, we can apply the dominated convergence theorem to infer that
\begin{equation*}
\lim_{t\downarrow 0}\mathcal{H}(t)=0\, .
\end{equation*}
Combining with \eqref{eq:boundrightderivative} and the local Lipschitz continuity of $\mathcal{H}$, see \autoref{lemma:calHt}, we get 
\begin{equation*}
\mathcal{H}(t)\le 2C(K,N,R)\mathcal{A}_{u,R}^2t\, ,\quad\text{for any $t\in (0,\bar{t})$}\, .
\end{equation*}
Let us notice that, for any $t\in (0,\bar{t})$, the function $v(t,\cdot,\cdot)$ is a non-negative sub-solution of the heat equation on the cylinder $B_{R/2}(q)\times (0,1)$ by \autoref{prop:superheat}, hence so is $v/t$. Using the Harnack inequality \autoref{prop:Harnack} we obtain
\begin{align}\label{eq:estvt}
\nonumber\sup_{B_{R/8}(q)\times (\frac38,\frac58)}\frac{v(t,x,\lambda)}{t}\le& \frac{C}{R^2B_{R/4}(q)}\int_{B_{R/4}(q)\times (\frac14,\frac34)}\frac{v(t,x,\lambda)}{t}\di\meas(x)\di\Leb^1(\lambda)\\
\le & \bar{C} \mathcal{A}_{u,R}^2\, ,\quad\text{for any $t\in(0,\bar{t})$}\, .
\end{align}
Let us see how to complete the local Lipschitz estimate. Let us consider $x,y\in B_{R/8}(q)$. We apply \eqref{eq:estvt} with $\lambda=1/2$ and get
\begin{equation}\label{eq:estalmostfinal}
\frac{\dist_Y(u(x),u(y))}{t}-e^{-K}\frac{\dist^2(x,y)}{2t^2}\le \frac{v(t,x,\frac{1}{2})}{t}\le \bar{C}\mathcal{A}^2_{u,R}\, .
\end{equation}
In particular, if 
\begin{equation*}
\dist(x,y)<e^{K/2}\mathcal{A}_{u,R}\bar{t}\, ,
\end{equation*}
then employing \eqref{eq:estalmostfinal} with $t:=\dist(x,y)/(e^{K/2}\mathcal{A}_{u,R})$, we obtain 
\begin{equation}\label{eq:finalalm}
\dist_Y(u(x),u(y))\le\left(\bar{C}+\frac{1}{2}\right)e^{-K/2}\mathcal{A}_{u,R}\dist(x,y)=\tilde{C}\dist(x,y)\, . 
\end{equation}
The above shows that the local Lipschitz estimate holds uniformly for points sufficiently close, i.e. when $\dist(x,y)\le e^{K/2}\mathcal{A}_{u,R}\bar{t}$.\\ 
If $\dist(x,y)>e^{K/2}\mathcal{A}_{u,R}\bar{t}$, then we consider a minimizing geodesic $\gamma:[0,\dist(x,y)]\to X$ connecting $x$ to $y$. Then we choose $N\ge 1$ and points $\gamma(t_i)$ along $\gamma$ with $\gamma(t_0)=\gamma(0)=x$ and $\gamma(t_N)=\gamma(\dist(x,y))=y$ in such a way that $\dist(\gamma(t_i),\gamma(t_{i+1}))<e^{K/2}\mathcal{A}_{u,R}\bar{t}$ and we apply repeatidly \eqref{eq:finalalm} between $\gamma(t_i)$ and $\gamma(t_i)$ to get
\begin{equation*}
\dist_Y(u(x),u(y))\le\sum_{i=0}^{N-1}\dist_Y(u(\gamma(t_i)),u(\gamma(t_{i+1})))\le \tilde{C} \sum_{i=0}^{N-1}\dist(\gamma(t_i),\gamma(t_{i+1}))=\tilde{C}\dist(x,y)\, ,
\end{equation*}
which concludes the proof of the local Lipschitz continuity of $u$. Above we used the triangle inequality for the first inequality, \eqref{eq:finalalm} for the second inequality and the choice of the points $\gamma(t_i)$ along the minimizing geodesic between $x$ and $y$ for the last equality.
\end{proof}

\section{Bochner inequality with Hessian-type term}\label{sec:Bochner}

The goal of this section is to prove the following.

\begin{theorem}\label{thm:Bochner}
Let $(X,\dist,\meas)$ be an $\RCD(K,N)$ metric measure space for some $K\in\setR$, $1\le N<\infty$, and let $(Y,\dist_Y)$ be a $\CAT(0)$ space. Let $\Omega\subset X$ be an open domain and let $u:\Omega\to\setR$ be a harmonic map. Then $\lip u\in W^{1,2}_{\loc}(\Omega)\cap L^{\infty}_{\loc}(\Omega)$ and 
\begin{equation}\label{eq:bochnerwithhessian}
\Delta \frac{\abs{\lip u}^2}{2}\ge \abs{\nabla \lip u}^2+K\abs{\lip u}^2\, ,\quad\text{on $\Omega$}
\end{equation}
in the sense of distributions.
\end{theorem}

The above \eqref{eq:bochnerwithhessian} is a weak Bochner inequality. The term $\abs{\nabla \lip u}^2$ at the right hand side is a Hessian type term and the appearance of such terms for Bochner inequalities for harmonic maps between singular spaces is a delicate issue.
\smallskip 

Already for scalar valued maps defined on a non-smooth $\RCD$ space, the validity of the Bochner inequality (even without Hessian term) is a deep result: it was proved for $\RCD(K,\infty)$ spaces in \cite{AGSDuke} (see also \cite{AGS15} for the reverse implication); the dimensional improvement for $\RCD^*(K,N)$ spaces was established independently in \cite{EKS} and \cite{AmbrosioMondinoSavare19} (together with the reverse implication). The fact that the scalar Bochner inequality (without Hessian) ``self-improves'' to  estimate the norm of the Hessian was noticed in the smooth setting of $\Gamma$-calculus in \cite{Bakry85} and then obtained in the non-smooth setting of $\RCD$ spaces in \cite{Savare14} and \cite{Gigli18}.

For \emph{smooth} harmonic maps between \emph{smooth} Riemannian manifolds, a Bochner-type identity was proved in the seminal work \cite{EellsSampson64}. For harmonic maps into singular spaces, obtaining a Bochner inequality is a delicate problem.
When the domain $\Omega$ has non-negative sectional curvature and the target $Y$ is a non-positively curved simplicial complex, some weak forms of Bochner-type inequalities have been obtained in \cite{Chen95, KorevaarSchoen}.
The list of contributions in the topic is then quite long, until \cite{ZhangZhongZhu19} proved the validity of the Bochner-type inequality \eqref{eq:bochnerwithhessian} for harmonic maps $u:\Omega\to Y$, where $\Omega$ is a \emph{smooth} domain of an $n$-dimensional Riemannian manifold with $\mathrm{Ric}\geq K$, and $Y$ is a $\CAT$ space. 

To the best of our knowledge, \autoref{thm:Bochner} is the first result about validity of a Bochner-type inequality with Hessian-type term for harmonic maps when both the source and the target spaces are non-smooth.
\medskip

The proof of \autoref{thm:Bochner} will follow the strategy in \cite{ZhangZhongZhu19}, dealing with the case of smooth source spaces. The fundamental novelty, in the same spirit as in the previous sections, will be the use of the interplay between optimal transport and heat flow on $\RCD$ spaces as a replacement of computations via the second variation of arc length and parallel transport in the smooth setting.\\
For any $q\in(1,2]$, for any domain $\Omega'$ compactly contained in $\Omega$ and for any $t>0$, we consider the auxiliary function $f_t:\Omega'\to\setR$ defined by
\begin{equation}\label{eq:defftB}
f_t(x):=\inf_{y\in\Omega'}\left\{\frac{\dist^p(x,y)}{pt^{p-1}}-\dist_Y(u(x),u(y))\right\}\, ,
\end{equation}
where $p:=q/(q-1)$. Notice that if $q\in (1,2]$, then $p\in (2,\infty]$. We will avoid stressing the dependence on $q$, as it will be always clear from the context.\\
Moreover, we shall denote by $S_t(x)$ the set of those points attaining the infimum in \eqref{eq:defftB} and we introduce a function $L_t:\Omega'\to\setR$ via
\begin{equation*}
L_t(x):=\min_{z\in S_t(x)}\dist(x,z)\, .
\end{equation*}
We choose $B_R(o)\subset X$ such that $B_{2R}(o)\Subset \Omega$. We set
\begin{equation*}
l_0:=\sup_{x,y\in B_{2R}(o)}\frac{\dist_Y(u(x),u(y))}{\dist(x,y)}<\infty\, ,
\end{equation*}
where finiteness of the local Lipschitz constant follows from \autoref{mainthcore}.\\
The proof of \autoref{thm:Bochner} will depend on some intermediate results.

\begin{lemma}\label{lemma:elem}
There exists a constant $C=C(p,l_0)>0$ such that for any $0<t<t^*$, it holds
\begin{equation}
L_t\le Ct\,,  \quad\, 0\le -f_t\le Ct\, ,\quad\text{on $B_R(o)$}\, .
\end{equation}
Moreover, $f_t$ is Lipschitz on $B_R(o)$ and $L_t$ is lower semicontinuous on $B_R(o)$.
\end{lemma}

The proof of \autoref{lemma:elem} is completely elementary and based on the local Lipschitz continuity of $u:\Omega\to Y$, therefore we omit it. We refer to \cite[Lemma 4.1]{ZhangZhongZhu19} for the detailed proof in the context of maps from smooth Riemannian manifolds to $\CAT(k)$ spaces, which is based on metric arguments and therefore works verbatim in the present setting.
\medskip

\begin{proposition}
Let $(X,\dist,\meas)$ be an $\RCD(K,N)$ metric measure space and let $(Y,\dist_Y)$ be a $\CAT(0)$ space. Let $\Omega\subset X$ be an open domain and let $u:\Omega\to Y$ be a harmonic map. Then $u$ is metrically differentiable at $\meas$-a.e. $x\in\Omega$.  
\end{proposition}

\begin{proof}
The statement follows from \autoref{thm:ksomega} (iii) (see also the proof of \autoref{prop:lappoin}), which gives approximate metric differentiability for general Sobolev functions, combined with \autoref{mainthcore}. Indeed, as the function $u$ is locally Lipschitz, the functions $X_i\ni y\mapsto \dist_Y(u(x),u(y))/r_i$, where $X_i:=(X,r_i^{-1}\dist,\left(\meas(B_{r_i}(x))\right)^{-1}\meas,x)$, are uniformly Lipschitz on $B^{X_i}_{1}(x)$. Moreover, they all vanish at $x$. Therefore Ascoli-Arzel\'a's theorem for pmGH converging sequences of metric spaces shows that they converge locally uniformly up to subsequences. Since we already know that the sequence converges to a semi-norm on $\setR^n$ in $H^{1,2}_{\loc}$, the convergence is uniform.
\end{proof}

\begin{proposition}\label{prop:nonlinearHL}
Let $q\in (1,\infty)$ and $p\in (1,\infty)$ be such that $1/p+1/q=1$ and let $f_t$ be as above. Then, for any $x\in B_R(o)$, it holds
\begin{equation}\label{eq:lowerbound}
\liminf_{t\to 0}\frac{f_t(x)}{t}\ge -\frac{1}{q}\left(\lip u(x)\right)^q\, .
\end{equation}
Moreover, if $u$ is metrically differentiable at $x$, then 
\begin{equation}\label{eq:limit0}
\lim_{t\to 0}\frac{f_t(x)}{t}=-\frac{1}{q}\left(\lip u(x)\right)^q\, ,\quad \lim_{t\to 0}\frac{L_t(x)}{t}=\left(\lip u(x)\right)^{\frac{q}{p}}\, .
\end{equation}
\end{proposition}

\begin{proof}
The proof of the \eqref{eq:lowerbound} is elementary and based on the inequality
\begin{equation*}
\frac{a^p}{p}+\frac{b^q}{q}\ge ab\, \quad\text{for any $a,b\in [0,\infty)$}\, .
\end{equation*}
We refer to the proof of \cite[Lemma 4.4]{ZhangZhongZhu19} for the detailed argument.
\\In order to prove \eqref{eq:limit0}, let us fix a metric differentiability point $x\in \Omega$. We choose $\xi\in\setR^n$ with $\norm{\xi}=1$ such that 
\begin{equation*}
\mathrm{md}_x u(\xi)=\norm{\mathrm{md}_x u}=\lip u(x)\, .
\end{equation*}
Then we consider points $y_t$ such that $\dist(x,y_t)=t\left(\lip u(x)\right)^{q/p}$ and $y_t$ converge to the point $\left(\lip u\right)^{q/p}\xi\in \setR^n$ along the family $X_t:=(X,t^{-1}\dist,\left(\meas(B_t(x))\right)^{-1}\meas,z)$ as $t\downarrow 0$.\\
By metric differentiability,
\begin{equation*}
\dist_Y(u(y_t),u(x))=\dist(x,y_t)\mathrm{md}_x(u)(\xi)+o(\dist(x,y_t))=t\left(\lip u(x)\right)^q+o(t)\, ,\quad\text{as $t\to 0$}\, .
\end{equation*}
By the very definition of $f_t$,
\begin{align*}
\frac{f_t(x)}{t}\le & \frac{\dist^p(x,y_t)}{pt^{p}}-\frac{\dist_Y(u(x),u(y_t))}{t}\\
=&\frac{\left(\lip u(x)\right)^q}{p}-\left(\lip u(x)\right)^q+o(1)\\
=&-\frac{\left(\lip u(x)\right)^q}{q}+o(1)\, ,
\end{align*}
which proves that 
\begin{equation*}
\limsup_{t\to 0}\frac{f_t(x)}{t}\le -\frac{\left(\lip u(x)\right)^q}{q}\, .
\end{equation*}
The verification of the second inequality in \eqref{eq:limit0} is completely analogous and we refer to the proof of \cite[Lemma 4.4]{ZhangZhongZhu19} for the detailed argument.
\end{proof}

\begin{proposition}\label{prop:Ellipticequation}
With the same notation introduced above and under the same assumptions, it holds
\begin{equation}\label{eq:ellipticbis}
\Delta f_t\le -K\frac{L_t^p}{t^{p-1}}\, ,\quad\text{on $B_{2R}(o)$}\, ,
\end{equation}
in the sense of distributions.
\end{proposition}

\begin{proof}
The proof is completely analogous to the proof of \autoref{prop:mainestimate}, building on the contractivity of the Heat Flow in Wasserstein spaces of order $p$, instead of the Wasserstein spaces of order $2$. Therefore we omit the details and point out the only relevant differences in the argument.
\medskip

The only modification needed with respect to Step 2 in the proof of \autoref{prop:mainestimate} is the observation that the function
\begin{equation*}
B_R(o)\times B_R(o)\ni (x,y)\mapsto \dist^p(x,y)
\end{equation*}
has measure valued Laplacian bounded above by a constant $C(K,N,p,R)$ on $B_R(o)\times B_R(o)$, for any $p>2$.  The statement follows from the chain rule for the measure valued Laplacian, writing $\dist^p(x,y)=\dist^2(x,y)\cdot\dist^{p-2}(x,y)$ and recalling \autoref{lemma:elemdistprod}. 
\smallskip

With respect to Step 3 in the proof of \autoref{prop:mainestimate} here we consider optimal transport plans between heat kernels for the cost $\dist^p$. The Wasserstein $W_p$ contraction estimate for the Heat Flow under the $\RCD(K,\infty)$ condition (see \cite[Theorem 4.4]{Savare14}) guarantees that for any couple of points $w,z\in X$ and for any $s>0$ there exists an admissible transport plan $\Pi_s$ between the heat kernels $P_s\delta_w$ and $P_s\delta_z$ such that 
\begin{equation*}
\int_{X\times X}\dist^p(x,y)\di\Pi_s\le e^{-Kpt}\dist^p(w,z)\, .
\end{equation*}
Then we replace the optimal transport plans for quadratic cost with optimal transport plan for cost $\dist^p$ in the proof of \autoref{prop:mainestimate}. This replaces the estimate \eqref{eq:estbyContraction}. All the subsequent estimates work verbatim as they only rely on the fact that $\Pi_s$ is an admissible transport plan between the heat kernel measures and not on the optimality for a specific cost.
\smallskip

Following the proof of \autoref{prop:mainestimate} we obtain \eqref{eq:ellipticbis}.
\end{proof}
\medskip

\begin{proof}[Proof of \autoref{thm:Bochner}]
In order to prove \eqref{eq:bochnerwithhessian} we use two limiting arguments. The first one will be aimed at proving that $\left(\lip u\right)^q\in W^{1,2}(B_{R}(o))$ and 
\begin{equation*}
\Delta \frac{\left(\lip u\right)^q}{q}\ge K\left(\lip u\right)^q\, \quad\text{on $B_{R}(o)$}\, ,
\end{equation*}
in the sense of distributions for any $q\in(1,2]$. In the second step we will take the limit as $q\to 1$ and obtain \eqref{eq:bochnerwithhessian}.
\medskip

\textbf{Step 1}.
We notice that the functions $-f_t/t$ are uniformly bounded by \autoref{lemma:elem}. Moreover
\begin{equation*}
\Delta \left(-\frac{f_t}{t}\right)\ge K\frac{L_t^p}{t^p}\ge -\abs{K}C\, ,\quad\text{on $B_{2R}(o)$}\, ,
\end{equation*}
for some constant $C>0$, thanks to \eqref{eq:ellipticbis} and \autoref{lemma:elem} again. By Caccioppoli's inequality, the energies
\begin{equation*}
\frac{1}{t^2}\int _{B_{R}(o)}\abs{\nabla f_t}^2\di\meas
\end{equation*}
are uniformly bounded. Hence, taking the limit as $t\to 0$ and taking into account \eqref{eq:limit0} we obtain that $\left(\lip u\right)^q\in W^{1,2}(B_R(o))$. Moreover, with the help of \autoref{prop:nonlinearHL}, we can divide by $t>0$ and pass to the limit as $t\downarrow 0$ in \eqref{eq:ellipticbis}, to obtain that, for any $q\in(1,2]$,
\begin{equation}\label{eq:Bochnerq}
\Delta \frac{\left(\lip u\right)^q}{q}\ge K\left(\lip u\right)^q\,, \quad\text{on $B_{R}(o)$}\, ,
\end{equation}
in the sense of distributions.

\medskip

\textbf{Step 2.} In this step we argue as in the first one, proving uniform estimates with respect to $q\in(1,2]$ and then taking the limit as $q\downarrow 1$.\\ 
Notice that the functions $\lip u^q/q$ are uniformly bounded. Moreover, they have Laplacians uniformly bounded from below, thanks to \eqref{eq:Bochnerq}. Hence, by the Caccioppoli inequality they have uniformly bounded $W^{1,2}$ energies on $B_{R/2}(o)$. Therefore we can pass to the $L^2$ limit as $q\downarrow 1$, to obtain that $\lip u\in W^{1,2}(B_{R/2}(o))$ and 
\begin{equation}\label{eq:Bochner1}
\Delta \lip u\ge K\lip u\, ,\quad\text{on $B_{R/2}(o)$ ,}
\end{equation}
in the sense of distributions.\\
By the chain rule, \eqref{eq:Bochner1} implies that 
\begin{equation*}
\Delta \frac{\abs{\lip u}^2}{2}\ge  \abs{\nabla \lip u}^2+K\abs{\lip u}^2\,, \quad\text{on $B_{R/2}(o)$}\, ,
\end{equation*}
in the sense of distributions.\\
As the statement is clearly local, the proof is complete.
\end{proof}

\end{document}